\documentclass[12pt,a4paper]{article}	

\usepackage[english]{babel}
\usepackage{amsmath}
\usepackage{amssymb}
\usepackage{mathdots}
\usepackage{mathabx}
\usepackage{tikz}
\usepackage{tikz-cd}
\usepackage{enumitem}
\usepackage[whole]{bxcjkjatype}
\usepackage{csquotes}
\usepackage{changepage}
\usepackage{pdflscape}
\usepackage{geometry}
\usepackage{color}
\definecolor{linkcolor}{rgb}{0,0,0.6}
\usepackage[colorlinks=true,
pdfstartview=FitV,
linkcolor= linkcolor,
citecolor= linkcolor,
urlcolor= linkcolor,
hyperindex=true,
hyperfigures=false]
{hyperref}
\usepackage{scalerel}
\usepackage{amsthm}

\title{On the cohomology of the basic unramified PEL unitary Rapoport-Zink space of signature $(1,n-1)$}
\author{J.Muller}
\date{}

\usepackage[
backend=bibtex,
style=alphabetic,
]{biblatex}
\usepackage{ytableau}
\usepackage{scalerel}
\def\msquare{\mathord{\scalerel*{\Box}{gX}}}

\begin{document}

\newtheorem{theo}{Theorem}[section]
\newtheorem{prop}[theo]{Proposition}
\newtheorem{lem}[theo]{Lemma}
\newtheorem{corol}[theo]{Corollary}
\newtheorem{conj}[theo]{Conjecture}
\newtheorem*{theo*}{Theorem}
\newtheorem*{prop*}{Proposition}
\newtheorem*{corol*}{Corollary}

\theoremstyle{remark}
\newtheorem{rk}[theo]{Remark}
\newtheorem{rks}[theo]{Remarks}
\newtheorem{ex}[theo]{Example}

\theoremstyle{definition}
\newtheorem{defi}[theo]{Definition}
\newtheorem*{notation}{Notation}
\newtheorem*{notations}{Notations}


\newcommand{\thetamax}{\theta_{\mathrm{max}}}


\maketitle

\begin{center}

\parbox{16cm}{\small
\textbf{Abstract} : \it In this paper, we study the cohomology of the unitary unramified PEL Rapoport-Zink space of signature $(1,n-1)$ at maximal level. Our method revolves around the spectral sequence associated to the open cover by the analytical tubes of the closed Bruhat-Tits strata in the special fiber, which were constructed by Vollaard and Wedhorn. The cohomology of these strata, which are isomorphic to generalized Deligne-Lusztig varieties, has been computed in \cite{muller}. This spectral sequence allows us to prove the semisimplicity of the Frobenius action and the non-admissibility of the cohomology in general. Via $p$-adic uniformization, we relate the cohomology of the Rapoport-Zink space to the cohomology of the supersingular locus of a Shimura variety with no level at $p$. In the case $n=3$ or $4$, we give a complete description of the cohomology of the supersingular locus in terms of automorphic representations.}

\vspace{0.5cm}
\end{center}

\tableofcontents

\vspace{1.5cm}

\textbf{\textsc{Introduction:}} By defining moduli problems classifying deformations of $p$-divisible groups with additional structures, Rapoport and Zink have constructed their eponymous spaces which consist in a projective system $(\mathcal M_{K_p})$ of non-archimedean analytic spaces. The set of data defining the moduli problem determines two $p$-adic groups $G(\mathbb Q_p)$ and $J(\mathbb Q_p)$ which both act on the tower. Its cohomology is therefore equipped with an action of $G(\mathbb Q_p) \times J(\mathbb Q_p) \times W$ where $W$ is the absolute Weil group of a finite extension of $\mathbb Q_p$, called the local reflex field. This is expected to give a geometric incarnation of the local Langlands correspondence. So far, relatively little is known about the cohomology of Rapoport-Zink spaces in general. The Kottwitz conjecture describes the $G(\mathbb Q_p)\times J(\mathbb Q_p)$-supercuspidal part of the cohomology but it is only known in a handful of cases. It was first proved for the Lubin-Tate tower in \cite{boyer99} and in \cite{harris}, from which the Drinfeld case follows by duality. The case of basic unramified EL Rapoport-Zink spaces has been treated in \cite{fargues} and \cite{shin}. As for the PEL case, it was proved for basic unramified unitary Rapoport-Zink spaces with signature $(1,n-1)$ with $n$ odd in \cite{nguyen}, and in \cite{nguyenmeli} for an arbitrary signature with an odd number of variables. Beyond the Kottwitz conjecture, one would like to understand the individual cohomology groups of the Rapoport-Zink spaces entirely. This has been done in \cite{boyer2} for the Lubin-Tate case (and, dually, for the Drinfeld case as well) using a vanishing cycle approach. Boyer's results were later used in \cite{dat} to recover the action of the monodrony and give an elegant form of geometric Jacquet-Langlands correspondence. However, this method relied heavily on the particuliar geometry of the Lubin-Tate tower, and we are faced with technical issues in other situations where we do not have a satisfactory understanding of the geometry of the Rapoport-Zink spaces.\\
In this paper, we aim at pursuing the goal of describing the individual cohomology groups of the Rapoport-Zink spaces in the basic PEL unramified unitary case with signature $(1,n-1)$. Here, $G(\mathbb Q_p)$ is an unramified group of unitary similitudes in $n$ variables and $J(\mathbb Q_p)$ is an inner form of $G(\mathbb Q_p)$. In fact, $J(\mathbb Q_p)$ is isomorphic to $G(\mathbb Q_p)$ when $n$ is odd and $J(\mathbb Q_p)$ is the non quasi-split inner form when $n$ is even. Our approach is based on the geometric description of the reduced special fiber $\mathcal M_{\mathrm{red}}$ given in \cite{vw1} and \cite{vw2}. In these papers, Vollaard and Wedhorn built the Bruhat-Tits stratification $\{\mathcal M_{\Lambda}\}_{\Lambda}$ on $\mathcal M_{\mathrm{red}}$ which is interesting for two reasons: 

\begin{enumerate}[label={--},noitemsep,topsep=0pt]
\item the closed strata $\mathcal M_{\Lambda}$ are indexed by the vertices of the Bruhat-Tits building $\mathrm{BT}(J,\mathbb Q_p)$ of $J(\mathbb Q_p)$. The combinatorics of the stratification can be read on the building. 
\item each individual stratum $\mathcal M_{\Lambda}$ is isomorphic to a generalized Deligne-Lusztig variety for a finite group of Lie type of the form $\mathrm{GU}_{2\theta+1}(\mathbb F_p)$, arising in the maximal reductive quotient of the maximal parahoric subgroup $J_{\Lambda} := \mathrm{Fix}_J(\Lambda)$. Here $1\leq 2\theta +1 =: t(\Lambda) \leq n$ is an odd integer called the orbit type of $\Lambda \in \mathrm{BT}(J,\mathbb Q_p)$.
\end{enumerate}

Let $\thetamax := \left\lfloor\frac{n-1}{2}\right\rfloor$ so that we have $0 \leq \theta \leq \thetamax$ for all vertices $\Lambda \in \mathrm{BT}(J,\mathbb Q_p)$. In \cite{muller}, by exploiting the Ekedahl-Oort stratification on a given stratum $\mathcal M_{\Lambda}$, we computed the cohomology groups $\mathrm H^{\bullet}(\mathcal M_{\Lambda}\otimes \overline{\mathbb F_p},\overline{\mathbb Q_{\ell}})$ in terms of representations of $\mathrm{GU}_{2\theta+1}(\mathbb F_p)$ with a Frobenius action. We consider the Rapoport-Zink space $\mathcal M^{\mathrm{an}} := \mathcal M_{K_0}$ at maximal level, where $K_0 \subset G(\mathbb Q_p)$ is a hyperspecial maximal open compact subgroup. Then $\mathcal M^{\mathrm{an}}$ is an analytic space of dimension $n-1$. It admits an open cover by the analytical tubes $U_{\Lambda}$ of the closed Bruhat-Tits strata $\mathcal M_{\Lambda}$. This induces a $J(\mathbb Q_p)\times W$-equivariant Čech spectral sequence computing the cohomology of $\mathcal M^{\mathrm{an}}$
$$E_{1}^{a,b}: \bigoplus_{\gamma \in I_{-a+1}} \mathrm H^b_c(U_{\Lambda(\gamma)} \widehat{\otimes} \, \mathbb C_p,\overline{\mathbb Q_{\ell}}) \implies \mathrm H^{a+b}_c(\mathcal M^{\mathrm{an}},\overline{\mathbb Q_{\ell}}),$$ 
where for $s\geq 1$ the index set is given by
$$I_{s} := \left\{\gamma = (\Lambda^1,\ldots ,\Lambda^s) \in \mathrm{BT}(J,\mathbb Q_p)^s \,|\, \forall i, t(\Lambda^i) = 2\thetamax+1 \text{ and } U(\gamma) := \bigcap_{i=1}^s U_{\Lambda^i} \not = \emptyset\right\}.$$
In the remaining of the introduction, we omit the Using Berkovich's comparison theorem, the cohomology of the tubes $U_{\Lambda}$ can be identified, up to a shift in indices and a suitable Tate twist, with the cohomology of the closed Bruhat-Tits strata $\mathcal M_{\Lambda}$. Let $\mathrm{Frob} \in W$ be a lift of the geometric Frobenius and let $\tau$ denote the action of the element $(p\cdot\mathrm{id},\mathrm{Frob}) \in J(\mathbb Q_p)\times W$ on the cohomology. We refer to $\tau$ as the \enquote{rational Frobenius}. Then the action of $\tau$ on the cohomology of $U_{\Lambda}$ is identified with the Frobenius action on the cohomology of $\mathcal M_{\Lambda}$. 

\begin{prop*}
The spectral sequence degenerates on the second page $E_2$. For $0 \leq b \leq 2(n-1)$, the induced filtration on $\mathrm H_c^b(\mathcal M^{\mathrm{an}}\widehat{\otimes} \, \mathbb C_p,\overline{\mathbb Q_{\ell}})$ splits, ie. we have an isomorphism 
$$\mathrm H_c^b(\mathcal M^{\mathrm{an}}\widehat{\otimes} \, \mathbb C_p,\overline{\mathbb Q_{\ell}}) \simeq \bigoplus_{b \leq b' \leq 2(n-1)} E_2^{b-b',b'}.$$
The action of $W$ on $\mathrm H_c^b(\mathcal M^{\mathrm{an}}\widehat{\otimes} \, \mathbb C_p,\overline{\mathbb Q_{\ell}})$ is trivial on the inertia subgroup and the action of the rational Frobenius $\tau$ is semisimple. The subspace $E_2^{b-b',b'}$ is identified with the eigenspace of $\tau$ associated to the eigenvalue $(-p)^{b'}$.
\end{prop*}

Let us fix a maximal simplex $\{\Lambda_0,\ldots ,\Lambda_{\thetamax}\}$ in $\mathrm{BT}(J,\mathbb Q_p)$ such that $t(\Lambda_{\theta}) = 2\theta +1$ for all $0\leq \theta \leq \thetamax$, and let us write $J_{\theta}$ instead of $J_{\Lambda_{\theta}}$. In order to study the $J(\mathbb Q_p)$-action, we rewrite the terms $E_1^{a,b}$ using compactly induced representations
$$E_1^{a,b} \simeq \bigoplus_{\theta = 0}^{\thetamax} \mathrm{c-Ind}_{J_{\theta}}^J \, \left( \mathrm H_c^b(U_{\Lambda_{\theta}}\widehat{\otimes} \, \mathbb C_p,\overline{\mathbb Q_{\ell}}) \otimes \overline{\mathbb Q_{\ell}}[K_{-a+1}^{(\theta)}]\right).$$
Here for $s\geq 1$ and $0 \leq \theta \leq \thetamax$ the finite set $K_{s}^{(\theta)} \subset I_{s}$ is given by 
$$K_s^{(\theta)} := \{\gamma \in I_s \,|\, U(\gamma) = U_{\Lambda_{\theta}}\}.$$
It is equipped with an action of $J_{\theta}$ and $\overline{\mathbb Q_{\ell}}[K_s^{(\theta)}]$ is the associated permutation module. The various $J_{\theta}$'s are maximal parahoric subgroups of $J(\mathbb Q_p)$, and the representations $\mathrm H_c^b(U_{\Lambda_{\theta}}\widehat{\otimes} \, \mathbb C_p,\overline{\mathbb Q_{\ell}}) \otimes \overline{\mathbb Q_{\ell}}[K_{-a+1}^{(\theta)}]$ are trivial on the unipotent radical $J_{\theta}^+$. In particular, they are representations of the finite group of Lie type $\mathcal J_{\theta} := J_{\theta}/J_{\theta}^+ \simeq \mathrm{G}(\mathrm U_{2\theta+1}(\mathbb F_p)\times \mathrm U_{n-2\theta-1}(\mathbb F_p))$. By exploiting this spectral sequence and the underlying combinatorics of the Bruhat-Tits building of $J(\mathbb Q_p)$, we are able to compute the cohomology groups of $\mathcal M^{\mathrm{an}}$ of highest degree $2(n-1)$, and when $n=3$ or $4$ the group of degree $2(n-1)-1$ as well. We denote by $J^{\circ}$ the subgroup of $J(\mathbb Q_p)$ consisting of all the unitary similitudes in $J(\mathbb Q_p)$ whose multipliers are a unit. We note that $J^{\circ}$ is normal in $J(\mathbb Q_p)$ with quotient $J/J^{\circ} \simeq \mathbb Z$.

\begin{prop*}
There is an isomorphism 
$$\mathrm H_c^{2(n-1)}(\mathcal M^{\mathrm{an}}\widehat{\otimes} \, \mathbb C_p,\overline{\mathbb Q_{\ell}}) \simeq \mathrm{c-Ind}_{J^{\circ}}^J \, \mathbf 1,$$
and the rational Frobenius $\tau$ acts via multiplication by $p^{2(n-1)}$. 
\end{prop*}

For $\lambda$ a partition of $2\thetamax+1$, we denote by $\rho_{\lambda}$ the associated irreducible unipotent representation of $\mathrm{GU}_{2\thetamax+1}(\mathbb F_p)$ via the classification of \cite{ls} which we recall in Section 2. We also write $\rho_{\lambda}$ for its inflation to the maximal parahoric subgroup $J_{\thetamax}$. In particular, if $2\thetamax+1$ is equal to $\frac{t(t+1)}{2}$ for some integer $t\geq 1$, we write $\Delta_t := (t,t-1,\ldots ,1)$ for the partition of $2\thetamax+1$ whose Young diagram is a staircase. The unipotent representation $\rho_{\Delta_t}$ of $\mathrm{GU}_{2\thetamax+1}(\mathbb F_p)$ is cuspidal. 

\begin{theo*}
Assume that $\thetamax = 1$, ie. $n=3$ or $4$. We have 
$$\mathrm H_c^{2(n-1)-1}(\mathcal M^{\mathrm{an}}\widehat{\otimes} \, \mathbb C_p,\overline{\mathbb Q_{\ell}}) \simeq \mathrm{c-Ind}_{J_1}^{J} \, \rho_{\Delta_2},$$
with the rational Frobenius $\tau$ acting via multiplication by $-p^{2(n-1)-1}$.
\end{theo*}

In general, the terms $E_2^{a,b}$ in the second page may be difficult to compute. However, the terms corresponding to $a=0$ and $b \in \{2(n-1-\thetamax),2(n-1-\thetamax)+1\}$ are not touched by any non-zero differential in the alternating version of the \v{C}ech spectral sequence, making their computations accessible. We note that $2(n-1-\thetamax)$ is equal to the middle degree when $n$ is odd, and to one plus the middle degree when $n$ is even. 

\begin{prop*}
We have an isomorphism of $J(\mathbb Q_p)$-representations
$$E_2^{0,2(n-1-\thetamax)} \simeq \mathrm{c-Ind}_{J_{\thetamax}}^J \, \rho_{(2\thetamax+1)}.$$
If $n\geq 3$ then we also have an isomorphism
$$E_2^{0,2(n-1-\thetamax)+1} \simeq \mathrm{c-Ind}_{J_{\thetamax}}^J \, \rho_{(2\thetamax,1)}.$$
\end{prop*}

We note that the representation $\rho_{(2\thetamax+1)}$ is the trivial representation. Using type theory, we may describe the inertial supports of the irreducible subquotients of such compactly induced representations. An inertial class is a pair $[L,\tau]$ where $L$ is a Levi complement of $J(\mathbb Q_p)$ and $\tau$ is a supercuspidal representation of $L$, up to conjugation and twist by an unramified character. Any smooth irreducible representation $\pi$ of $J(\mathbb Q_p)$ determines a unique inertial class $\ell(\pi)$. If $\mathfrak s$ is an inertial class, let $\mathrm{Rep}^{\mathfrak s}(J(\mathbb Q_p))$ be the category of smooth representations of $J(\mathbb Q_p)$ all of whose irreducible subquotients $\pi$ satisfy $\ell(\pi) = \mathfrak s$. For $\mathfrak S$ a set of inertial classes, let $\mathrm{Rep}^{\mathfrak S}(J(\mathbb Q_p))$ be the direct product of the categories $\mathrm{Rep}^{\mathfrak s}(J(\mathbb Q_p))$ for $\mathfrak s \in \mathfrak S$.\\
Let $(\mathbf V,\{\cdot,\cdot\})$ be the $n$-dimensional $\mathbb Q_{p^2}$-hermitian space whose group of unitary similitudes is $J(\mathbb Q_p)$. The Witt index of $\mathbf V$ is $\thetamax$. Let 
$$\mathbf V = H_1 \oplus \ldots \oplus H_{\thetamax} \oplus \mathbf V^{\mathrm{an}}$$
be a Witt decomposition, where for all $1\leq i \leq \thetamax$, $H_i$ is a hyperbolic plane and where $\mathbf V^{\mathrm{an}}$ is anisotropic. Note that $\mathbf V^{\mathrm{an}}$ has dimension $1$ or $2$ depending on whether $n$ is odd or even respectively. For $0 \leq f \leq \thetamax$, we consider
$$L_f := \mathrm{G}\left(\mathrm{U}(H_1 \oplus \ldots \oplus H_f  \oplus \mathbf V^{\mathrm{an}})\times T_{f+1} \times \ldots \times T_{\thetamax} \right),$$
where for $1\leq i \leq \thetamax$, $T_i \subset \mathrm{GU}(H_i)$ is a maximal torus. Then $L_f$ can be seen as a Levi complement in $J(\mathbb Q_p)$, and $L_{\thetamax} = J(\mathbb Q_p)$. In particular $L_0$ is a minimal Levi complement. Let $\tau_0$ denote the trivial representation of $L_0$, and let $\tau_1$ denote the representation of $L_1$ obtained by letting the $T_i$'s for $i\geq 2$ act trivially, and $\mathrm{GU}(H_1 \oplus \mathbf V^{\mathrm{an}})$ act through the compact induction of the inflation to a special maximal parahoric subgroup of the unique cuspidal unipotent representation $\rho_{\Delta_2}$ of $\mathrm{GU}_3(\mathbb F_p)$. For $f=0,1$, the irreducible representation $\tau_f$ of $L_f$ is supercuspidal. For $V$ a smooth representation of $J(\mathbb Q_p)$ and $\chi$ a continuous character of the center $\mathrm Z(J(\mathbb Q_p))$, we denote by $V_{\chi}$ the maximal quotient of $V$ on which the center acts like $\chi$. Combining our previous proposition with an analysis of the inertial supports via type theory, we obtain the following proposition.

\begin{prop*}
Let $\chi$ be an unramified character of $\mathrm Z(J)$. 
\begin{enumerate}[label={--},noitemsep]
\item Assume that $n\geq 3$. The representation $(E_2^{0,2(n-1-\thetamax)})_{\chi}$ contains no non-zero admissible subrepresentation, and is not $J(\mathbb Q_p)$-semisimple. Moreover, any irreducible subquotient has inertial support $[L_0,\tau_0]$. If $n\geq 5$, then the same statement holds for $(E_2^{0,2(n-1-\thetamax)+1})_{\chi}$ with the inertial support being $[L_1,\tau_1]$.
\item For $n = 1,2,3,4$, let $b = 0,2,3,5$ respectively. We have $\thetamax = 0$ if $n=1,2$ and $\thetamax = 1$ if $n=3,4$. Let $\chi$ be an unramified character of $\mathrm{Z}(J(\mathbb Q_p))$. The twist $\tau_{\thetamax,\chi}$ of $\tau_{\thetamax}$ by $\chi$ is an irreducible supercuspidal representation of $J(\mathbb Q_p)$, and we have 
$$(E_2^{0,b})_{\chi} \simeq 
\begin{cases}
\tau_{\thetamax,\chi} & \text{if } n = 1,3,4,\\
\tau_{\thetamax,\chi}\oplus\chi_0\tau_{m,\chi} & \text{if } n=2.
\end{cases}$$ 
\end{enumerate}
\end{prop*}

Here, when $n=2$ the subgroup $\mathrm Z(J(\mathbb Q_p))J_0$ has index $2$ in $J(\mathbb Q_p)$. In this situation, $\chi_0$ denotes the unique non-trivial character of $J(\mathbb Q_p)$ which is trivial on $\mathrm Z(J)J_0$. This proposition yields the following important corollary. 

\begin{corol*}
Let $\chi$ be an unramified character of $\mathrm Z(J(\mathbb Q_p))$. If $n\geq 3$ then $\mathrm H_c^{2(n-1-\thetamax)}(\mathcal M^{\mathrm{an}}\widehat{\otimes} \, \mathbb C_p,\overline{\mathbb Q_{\ell}})_{\chi}$ is not $J(\mathbb Q_p)$-admissible. If $n\geq 5$ then the same holds for $\mathrm H_c^{2(n-1-\thetamax)+1}(\mathcal M^{\mathrm{an}}\widehat{\otimes} \, \mathbb C_p,\overline{\mathbb Q_{\ell}})_{\chi}$.
\end{corol*}

Thus the cohomology of Rapoport-Zink spaces need not be admissible nor $J(\mathbb Q_p)$-semisimple in general. Lastly, we introduce the unramified unitary PEL Shimura variety of signature $(1,n-1)$ with no level structure at $p$. It is defined over a quadratic extension $F$ of $\mathbb Q$ in which the prime $p$ is inert. The corresponding Shimura datum gives rise to a reductive group $\mathbb G$ over $\mathbb Q$ such that $\mathbb G_{\mathbb Q_p} = G$ and $\mathbb G(\mathbb R) \simeq \mathrm{GU}(1,n-1)$. The Shimura varieties are indexed by the open compact subgroups $K^p \subset \mathbb G(\mathbb A_f^p)$ which are small enough. Kottwitz constructed integral models $\mathrm S_{K^p}$ at $p$ of these Shimura varieties. Their special fibers are stratified by the Newton strata, and the unique closed stratum is called the supersingular locus, which we denote by $\overline{\mathrm S}_{K^p}^{\mathrm{ss}}$ since it coincides with the supersingular locus. It has dimension $\thetamax$. The $p$-adic uniformization theorem of \cite{RZ} gives a geometric identity between the special fiber $\mathcal M_{\mathrm{red}}$ of the Rapoport-Zink space $\mathcal M$ and the supersingular locus $\overline{\mathrm S}_{K^p}^{\mathrm{ss}}$. In \cite{fargues}, Fargues constructed a Hochschild-Serre spectral sequence associated to this geometric identity, computing the cohomology of the supersingular locus.\\
Let $\xi$ be an irreducible algebraic finite dimensional representation of $\mathbb G$, and let $\overline{\mathcal L_{\xi}}$ be the associated local system on the Shimura variety, restricted to the special fiber. It is a pure sheaf of some weight $w(\xi) \in \mathbb Z_{\geq 0}$. Let $I$ be the inner form of $\mathbb G$ such that $I_{\mathbb Q_p} = J$, $I_{\mathbb A_f^p} = \mathbb G_{\mathbb A_f^p}$ and $I(\mathbb R) \simeq \mathrm{GU}(0,n)$. We denote by $\mathcal A_{\xi}(I)$ the set of automorphic representations of $I$ of type $\widecheck{\xi}$ at infinity, and counted with multiplicities. Fargues' spectral sequence is given in the second page by
$$F_2^{a,b} = \bigoplus_{\Pi\in\mathcal A_{\xi}(I)} \mathrm{Ext}_{J}^a \left (\mathrm H_c^{2(n-1)-b}(\mathcal M^{\mathrm{an}}\widehat{\otimes} \, \mathbb C_p, \overline{\mathbb Q_{\ell}})(1-n), \Pi_p\right) \otimes \Pi^p \implies \mathrm{H}^{a+b}(\overline{\mathrm S}^{\mathrm{ss}} \otimes \, \overline{\mathbb F_p}, \overline{\mathcal L_{\xi}}),$$
where $\mathrm H^{\bullet}(\overline{\mathrm S}^{\mathrm{ss}}\otimes \,\overline{\mathbb F_p},\overline{\mathcal L_{\xi}}) := \varinjlim_{K^p}\mathrm H^{\bullet}(\overline{\mathrm S}_{K^p}^{\mathrm{ss}} \otimes \, \overline{\mathbb F_p},\overline{\mathcal L_{\xi}})$. We point out that the abutment is just the cohomology of the supersingular locus with coefficients in $\overline{\mathcal L_{\xi}}$ because the nearby cycles are trivial thanks to the smoothness of the integral model $\mathrm S_{K^p}$. It is $\mathbb G(\mathbb A_f^p)\times W$-equivariant. When $n=3$ or $4$ this sequence degenerates on the second page, and our knowledge on the cohomology of the Rapoport-Zink space $\mathcal M^{\mathrm{an}}$ allows us to compute every term. We obtain a description of the cohomology of the supersingular locus in terms of automorphic representations.\\
A smooth character of $J(\mathbb Q_p)$ is said to be unramified if it is trivial on all compact subgroups of $J(\mathbb Q_p)$. Let $X^{\mathrm{un}}(J(\mathbb Q_p))$ denote the set of unramified characters of $J(\mathbb Q_p)$. Let $\mathrm{St}_J$ denote the Steinberg representation of $J(\mathbb Q_p)$. If $\Pi\in \mathcal A_{\xi}(I)$, we define $\delta_{\Pi_p} := \omega_{\Pi_p}(p^{-1}\cdot\mathrm{id})p^{-w(\xi)} \in \overline{\mathbb Q_{\ell}}^{\times}$ where $\omega_{\Pi_p}$ is the central character of $\Pi_p$, and $p^{-1}\cdot\mathrm{id}$ lies in the center of $J(\mathbb Q_p)$. For any isomorphism $\iota:\overline{\mathbb Q_{\ell}} \simeq \mathbb C$ we have $|\iota(\delta_{\Pi_p})| = 1$. Eventually, if $x\in \overline{\mathbb Q_{\ell}}^{\times}$, we denote by $\overline{\mathbb Q_{\ell}}[x]$ the $1$-dimensional representation of the Weil group $W$ where the inertia acts trivially and $\mathrm{Frob}$ acts like multiplication by the scalar $x$.

\begin{theo*}
Assume that $n = 3$ or $4$, so that $\overline{\mathrm S}^{\mathrm{ss}}$ is one dimensional. There are $\mathbb G(\mathbb A_f^p) \times W$-equivariant isomorphisms
\begin{align*}
\mathrm{H}^{0}(\overline{\mathrm S}^{\mathrm{ss}} \otimes \, \overline{\mathbb F_p}, \overline{\mathcal L_{\xi}}) & \simeq \bigoplus_{\substack{\Pi\in\mathcal A_{\xi}(I) \\ \Pi_p \in X^{\mathrm{un}}(J)}} \Pi^p \otimes \overline{\mathbb Q_{\ell}}[\delta_{\Pi_p}p^{w(\xi)}], \\
\mathrm{H}^{1}(\overline{\mathrm S}^{\mathrm{ss}} \otimes \, \overline{\mathbb F_p}, \overline{\mathcal L_{\xi}}) & \simeq \bigoplus_{\substack{\Pi\in\mathcal A_{\xi}(I) \\ \exists \chi \in X^{\mathrm{un}}(J),\\ \Pi_p = \chi\cdot\mathrm{St}_J}} \Pi^p \otimes \overline{\mathbb Q_{\ell}}[\delta_{\Pi_p}p^{w(\xi)}] \oplus \bigoplus_{\substack{\Pi\in\mathcal A_{\xi}(I) \\ \exists \chi \in X^{\mathrm{un}}(J),\\ \Pi_p = \chi\cdot\tau_1}} \Pi^p \otimes \overline{\mathbb Q_{\ell}}[-\delta_{\Pi_p}p^{w(\xi)+1}],\\
\mathrm{H}^{2}(\overline{\mathrm S}^{\mathrm{ss}} \otimes \, \overline{\mathbb F_p}, \overline{\mathcal L_{\xi}}) & \simeq \bigoplus_{\substack{\Pi\in\mathcal A_{\xi}(I) \\ \Pi_p^{J_1}\not = 0}} \Pi^p \otimes \overline{\mathbb Q_{\ell}}[\delta_{\Pi_p}p^{w(\xi)+2}].
\end{align*}
\end{theo*}

\textbf{\textsc{Notations:}} Throughout the paper, we fix an integer $n\geq 1$ and we write $\thetamax := \lfloor \frac{n-1}{2} \rfloor$ so that $n = 2\thetamax+1$ or $2(\thetamax+1)$ according to whether $n$ is odd or even. We also fix an odd prime number $p$. If $k$ is a perfect field of characteristic $p$, we denote by $W(k)$ the ring of Witt vectors and by $W(k)_{\mathbb Q}$ its fraction field, which is an unramified extension of $\mathbb Q_p$. We denote by $\sigma: x \mapsto x^p$ the Frobenius on $k$ or its lift to $W(k)$. If $q = p^e$ is a power of $p$, we write $\mathbb F_{q}$ for the field with $q$ elements. In the special case where $q=p^2$, we also use the alternative notation $\mathbb Z_{p^2} = W(\mathbb F_{p^2})$ and $\mathbb Q_{p^2} = W(\mathbb F_{p^2})_{\mathbb Q}$. We fix an algebraic closure $\mathbb F$ of $\mathbb F_p$. For $k\geq 1$, the $k\times k$ identity matrix is denoted by $I_k$, and the matrix with $1$ in the antidiagonal and $0$ everywhere else is denoted by $A_k$. In various situations, the symbol $\mathbf 1$ will always represent the trivial representation of the group we are considering. The symmetric group of $\{1,\ldots ,k\}$ is denoted $\mathfrak S_k$.\\

\textbf{\textsc{Acknowledgement:}} This paper is part of a PhD thesis under the supervision of Pascal Boyer and Naoki Imai. I am grateful for their wise guidance throughout the research. I also wish to address special thanks to Jean-Loup Waldspurger for helpful discussions regarding the structure of compactly induced representations.

\section{The Bruhat-Tits stratification on the PEL unitary \\ Rapoport-Zink space of signature $(1,n-1)$}

\subsection{The PEL unitary Rapoport-Zink space $\mathcal M$ of signature $(1,n-1)$}

In \cite{vw2}, the authors introduce the PEL unitary Rapoport-Zink space $\mathcal M$ of signature $(1,n-1)$ as a moduli space, classifying the deformations of a given $p$-divisible group equipped with additional structures. We briefly recall the construction. Let $E$ be a quadratic unramified extension of $\mathbb Q_p$ with ring of integers $\mathcal O_E$ and with nontrivial Galois involution $a \mapsto a^*$. Let $\varphi_0: K \xrightarrow{\sim} \mathbb Q_{p^2}$ be a $\mathbb Q_p$-linear isomorphism and let $\varphi_1 := \sigma \circ \varphi_0$. Let $\mathrm{Nilp}$ denote the category of schemes over $\mathbb Z_{p^2}$ where $p$ is locally nilpotent. For $S\in \mathrm{Nilp}$, a \textbf{unitary $p$-divisible group of signature $(1,n-1)$} over $S$ is a triple $(X,\iota_X,\lambda_X)$ where 
\begin{enumerate}[label={--}]
\item $X$ is a $p$-divisible group over $S$.
\item $\iota_X: \mathcal O_E\rightarrow \mathrm{End}(X)$ is a $\mathcal O_E$-action on $X$ such that the induced action on its Lie algebra satisfies the \textbf{signature $(1,n-1)$ condition}: for every $a \in \mathcal O_E$, the characteristic polynomial of $\iota_X(a)$ acting on $\mathrm{Lie}(X)$ is given by
$$(T-\varphi_0(a))^1(T-\varphi_1(a))^{n-1} \in \mathbb Z_{p^2}[T] \subset \mathcal O_{S}[T].$$
\item $\lambda_X:X \xrightarrow{\sim} {}^tX$ is an $\mathcal O_E$-linear polarization where ${}^tX$ denotes the Serre dual of $X$. 
\end{enumerate}

The $\mathcal O_E$-linearity of $\lambda_X$ is with respect to the $\mathcal O_E$-actions $\iota_X$ and the induced action $\iota_{{}^tX}$ on the dual. A specific example of unitary $p$-divisible group over $\mathbb F_{p^2}$ is given in \cite{vw2} 2.4 by means of covariant Dieudonné theory. We denote it by $(\mathbb X,\iota_{\mathbb X},\lambda_{\mathbb X})$ and call it the \textbf{standard unitary $p$-divisible group}. The $p$-divisible group $\mathbb X$ is superspecial. The following set-valued functor $\mathcal M$ defines a moduli problem classifying deformations of $\mathbb X$ by quasi-isogenies. More precisely, for $S \in \mathrm{Nilp}$ the set $\mathcal M(S)$ consists of all isomorphism classes of tuples $(X,\iota_X,\lambda_X,\rho_X)$ such that 
\begin{enumerate}[label={--}]
\item $(X,\lambda_X,\rho_X)$ is a unitary $p$-divisible group of signature $(1,n-1)$ over $S$.
\item $\rho_X: X\times_S \overline{S} \rightarrow \mathbb X\times_{\mathbb F_{p^2}} \overline{S}$ is an $\mathcal O_{E}$-linear quasi-isogeny compatible with the polarizations, in the sense that ${}^t\rho_X \circ \lambda_{\mathbb X} \circ \rho_X$ is a $\mathbb Q_p^{\times}$-multiple of $\lambda_X$.
\end{enumerate}

In the second condition, $\overline{S}$ denotes the special fiber of $S$. By \cite{RZ} Corollary 3.40, this moduli problem is represented by a separated formal scheme $\mathcal M$ over $\mathrm{Spf}(\mathbb Z_{p^2})$, called a \textbf{Rapoport-Zink space}. It is formally locally of finite type, and since the associated PEL datum is unramified it is also formally smooth over $\mathbb Z_{p^2}$. The \textbf{reduced special fiber} of $\mathcal M$ is the reduced $\mathbb F_{p^2}$-scheme $\mathcal M_{\mathrm{red}}$ defined by the maximal ideal of definition. Rational points of $\mathcal M$ over a perfect field extension $k$ of $\mathbb F_{p^2}$ can be understood in terms of semi-linear algebra by means of Dieudonné theory. We denote by $M(\mathbb X)$ the (covariant) Dieudonné module of $\mathbb X$, this is a free $\mathbb Z_{p^2}$-module of rank $2n$. We denote by $N(\mathbb X) := M(\mathbb X)\otimes \mathbb Q_{p^2}$ its isocrystal. By construction, the Frobenius $\mathbf F$ and the Verschiebung $\mathbf V$ agree on $N(\mathbb X)$. In particular, we have $\mathbf F^2 = p\cdot\mathrm{id}$ on the isocrystal. The $\mathcal O_E$-action $\iota_{\mathbb X}$ induces a $\mathbb Z/2\mathbb Z$-grading $M(\mathbb X) = M(\mathbb X)_0 \oplus M(\mathbb X)_1$ as a sum of two free $\mathbb Z_{p^2}$-modules of rank $n$, such that $a \in \mathcal O_E$ acts via $\varphi_i(a)$ on $M(\mathbb X)_i$ for $i = 0,1$. The same goes for the isocrystal $N(\mathbb X) = N(\mathbb X)_0 \oplus N(\mathbb X)_1$ where $N(\mathbb X)_i = M(\mathbb X)_i\otimes \mathbb Q_{p^2}$ for $i=0,1$. The polarization $\lambda_{\mathbb X}$ induces a perfect alternating $\mathbb Q_{p^2}$-bilinear pairing $\langle\cdot,\cdot\rangle$ on $N(\mathbb X)$ such that 
\begin{align*}
\forall x,y \in N(\mathbb X), \forall a \in E, & & \langle \mathbf Fx,y\rangle = \langle x,\mathbf Vy\rangle^{\sigma} \text{ and } \langle ax,y \rangle = \langle x, a^*y\rangle.
\end{align*}
Moreover $\langle\cdot,\cdot\rangle$ restricts to a perfect $\mathbb Z_{p^2}$-pairing on the lattice $M(\mathbb X)$. The pieces $N(\mathbb X)_i$ are totally isotropic for $i=0,1$ and dual of each other. Moreover, the Frobenius $\mathbf F$ is $1$-homogeneous with respect to this grading. As in \cite{vw2} 2.6, we define
\begin{align*}
\forall x,y \in N(\mathbb X)_0, && \{x,y\} := \delta\langle x,\mathbf Fy\rangle,
\end{align*}
where $\delta \in \mathbb Z_{p^2}^{\times}$ is a scalar satisfying $\delta^{\sigma} = -\delta$. The pairing $\{\cdot,\cdot\}$ is a perfect $\sigma$-hermitian form on $N(\mathbb X)_0$. 

\begin{notation}
From now on, we will write $\mathbf V := N(\mathbb X)_0$ and $\mathbf M := M(\mathbb X)_0$. 
\end{notation}

Then $\mathbf V$ is a $\mathbb Q_{p^2}$-hermitian space of dimension $n$, and $\mathbf M$ is a given $\mathbb Z_{p^2}$-lattice, ie. a finitely generated $\mathbb Z_{p^2}$-submodule containing a basis of $\mathbf V$. Given two lattices $M_1$ and $M_2$, the notation $M_1 \overset{d}{\subset} M_2$ means that $M_1\subset M_2$ and the quotient module $M_2/M_1$ has length $d$. The integer $d$ is called the \textbf{index} of $M_1$ in $M_2$, and is denoted $d = [M_2:M_1]$. Given a lattice $M\subset \mathbf V$, we define the dual lattice $M^{\vee} := \{v \in \mathbf V \,|\, \{v,M\} \subset \mathbb Z_{p^2}\}$.By construction the lattice $\mathbf M$ satisfies 
$$p\mathbf M^{\vee} \overset{1}{\subset} \mathbf M \overset{n-1}{\subset} \mathbf M^{\vee}.$$
Consider the matrices
$$T_{\text{odd}}:= A_{2\thetamax+1},
\quad \quad 
T_{\text{even}}:= 
\left(\begin{matrix} 
     & & & A_{\thetamax} \\
     & 1 & 0 &   \\
     & 0 & p &   \\
     A_{\thetamax} & & &
\end{matrix} \right).
$$
By \cite{vw1} Proposition 1.15, there exists a basis of $\mathbf V$ such that $\{\cdot,\cdot\}$ is represented by the matrix $T_{\text{odd}}$ is $n$ is odd and by $T_{\text{even}}$ if $n$ is even. A \textbf{Witt decomposition} on $\mathbf V$ is a set $\{L_i\}_{i\in I}$ of isotropic lines in $\mathbf V$ such that the following conditions are satisfied:
\setlist{nolistsep}
\begin{enumerate}[label={--},noitemsep]
\item for every $i \in I$, there is a unique $i'\in I$ such that $\{L_i,L_{i'}\} \not = 0$,
\item the sum of the $L_i$'s is direct,
\item the orthogonal of the direct sum of the $L_i$'s is an anisotropic subspace of $\mathbf V$. 
\end{enumerate}
Since each line $L_i$ is isotropic, in the first condition one necessarily has $(i')' = i$ and $i \not = i'$. As a consequence, we have $\#I = 2w(\mathbf V)$ for some integer $w = w(\mathbf V$ called the \textbf{Witt index} of $\mathbf V$. It does not depend on the choice of a Witt decomposition. We write $L^{\mathrm{an}}$ for the orthogonal of the direct sum of the $L_i$'s. The dimension of $L^{\mathrm{an}}$ is $n^{\mathrm{an}} := n - 2w$. Given a Witt decomposition of $\mathbf V$, one may find vectors $e_i \in L_i$ such that $\{e_i,e_j\} = \delta_{j,i'}$. Together with a choice of an orthogonal basis for $L^{\mathrm{an}}$, these vectors define a basis of $\mathbf V$ which is said to be adapted to the Witt decomposition. For any $i\in I$, the direct sum $L_i \oplus L_{i'}$ is isometric to the hyperbolic plane $\mathbf H$. Therefore, we obtain a decomposition 
$$\mathbf V = w\mathbf H \oplus L^{\mathrm{an}}.$$
We may always rearrange the index set so that $I = \{-w,\ldots,-1,1,\ldots,w\}$ and $i' = -i$ for all $i \in I$. In this context, we write $L_0$ instead of $L^{\mathrm{an}}$.\\
We fix once and for all a basis $e$ of $\mathbf V$ in which the hermitian form is represented by the matrix $T_{\text{odd}}$ or $T_{\text{even}}$. In the case $n = 2\thetamax+1$ is odd, we will denote it 
$$e = (e_{-\thetamax} , \ldots , e_{-1} , e_0^{\mathrm{an}} , e_1 , \ldots , e_{\thetamax}),$$
and in the case $n = 2(\thetamax+1)$ is even we will denote it 
$$e = (e_{-\thetamax} , \ldots , e_{-1} , e_0^{\mathrm{an}} , e_1^{\mathrm{an}} , e_1 , \ldots , e_{\thetamax}).$$
The choice of such a basis gives a Witt decomposition with $L_i := \mathbb Q_{p^2}e_i$ and $L_0$ the subspace generated by $e_0^{\mathrm{an}}$, and when $n$ is even by $e_1^{\mathrm{an}}$ as well. In particular, $w(\mathbf V) = \thetamax$ and $n^{\mathrm{an}} = 1$ or $2$ depending on whether $n$ is odd or even respectively.\\

Given a perfect field extension $k$ of $\mathbb F_{p^2}$, we denote by $\mathbf V_k$ the base change $\mathbf V \otimes_{\mathbb Q_{p^2}} W(k)_{\mathbb Q}$. The form may be extended to $\mathbf V_k$ by the formula 
$$\{v\otimes x,w\otimes y\} := xy^{\sigma}\{v,w\}\in W(k)_{\mathbb Q}$$
for all $v,w\in \mathbf V$ and $x,y\in W(k)_{\mathbb Q}$. The notions of index and duality for $W(k)$-lattices can be extended as well. By \cite{vw1} Proposition 1.10, the rational points of the Rapoport-Zink space are described by the following statement.

\begin{prop}
Let $k$ be a perfect field extension of $\mathbb F_{p^2}$. There is a natural bijection between $\mathcal M(k)$ and the set of $W(k)$-lattices $M$ in $\mathbf V_k$ such that for some integer $i\in \mathbb Z$, we have 
$$
p^{i+1}M^{\vee} \overset{1}{\subset} M \overset{n-1}{\subset} p^i M^{\vee}.
$$
\end{prop}

There is a decomposition $\mathcal M = \bigsqcup_{i\in \mathbb Z} \mathcal M_i$ into formal connected subschemes which are open and closed. The rational points of $\mathcal M_i$ are those lattices $M$ satisfying the relation above with the given integer $i$. In particular, the lattice $\mathbf M$ defined in the previous paragraph is an element of $\mathcal M_0(\mathbb F_{p^2})$. By \cite{vw1} Proposition 1.7, the formal scheme $\mathcal M_i$ is empty if $ni$ is odd.\\
Let $J = \mathrm{GU}(\mathbf V)$ be the group of unitary similitudes of $\mathbf V$, seen as a reductive group over $\mathbb Q_p$. Then $J(\mathbb Q_p)$ consists of all $g\in \mathrm{GL}_{\mathbb Q_{p^2}}(\mathbf V)$ which preserve the hermitian form up to a unit $c(g)\in \mathbb Q_{p}^{\times}$, called the \textbf{multiplier}. By Dieudonné theory, the group $J(\mathbb Q_p)$ is also identified with the group of quasi-isogenies $\mathbb X \to \mathbb X$ of unitary $p$-divisible groups. The space $\mathcal M$ is endowed with a natural action of $J(\mathbb Q_p)$. At the level of points, the element $g$ acts by sending a lattice $M$ to $g(M)$. For $g\in J(\mathbb Q_p)$, let $\alpha(g)$ be the $p$-adic valuation of $c(g)$. This defines a continous homomorphism
$$\alpha: J\rightarrow \mathbb Z$$
where $\mathbb Z$ is given the discrete topology. Then $g$ induces an isomorphism $\mathcal M_i \xrightarrow{\sim} \mathcal M_{i+\alpha(g)}$.  According to \cite{vw1} 1.17 the image of $\alpha$ is $\mathbb Z$ if $n$ is even, and $2\mathbb Z$ if $n$ is odd. The center $\mathrm Z(J(\mathbb Q_p))$ consists of all the scalar matrices, so that it is identified with $\mathbb Q_{p^2}^{\times}$. If $\lambda \in \mathbb Q_{p^2}^{\times}$, then $c(\lambda\cdot\mathrm{id}) = \mathrm{Norm}(\lambda) \in \mathbb Q_{p}^{\times}$, where $\mathrm{Norm}$ is the norm map relative to the quadratic extension $\mathbb Q_{p^2}/\mathbb Q_{p}$. In particular, $\alpha(\mathrm Z(J)) = 2\mathbb Z$. Thus, the restriction of $\alpha$ to the center is surjective onto $\mathrm{Im}(\alpha)$ only when $n$ is odd. When $n$ is even, we define the following element 
$$g_0 :=
\left(\begin{matrix} 
     & & & I_{\thetamax} \\
     & 0 & p &   \\
     & 1 & 0 &   \\
     pI_{\thetamax} & & &
\end{matrix} \right).$$
Then $g_0 \in J(\mathbb Q_p)$ and $c(g_0) = p$ so that $\alpha(g_0) = 1$. Moreover $g_0^2 = p\cdot\mathrm{id}$ belongs to $\mathrm Z(J(\mathbb Q_p))$. Let $i$ and $i'$ be two integers such that $ni$ and $ni'$ are even. We consider the multiplication $p^{\frac{i'-i}{2}}:\mathbb X\to \mathbb X$ when $i\equiv i' \mod 2$, and the quasi-isogeny $p^{\frac{i'-i-1}{2}}g_0:\mathbb X \to \mathbb X$ when $i\not \equiv i' \mod 2$. This is well defined as the second case may only happen when $n$ is even. It induces a morphism $\psi_{i,i'}:\mathcal M_i \rightarrow \mathcal M_{i'}$.  By \cite{vw1} Proposition 1.18, the map $\psi_{i,i'}$ is an isomorphism between $\mathcal M_i$ and $\mathcal M_{i'}$, and if $i,i'$ and $i''$ are three integers such that $ni, ni'$ and $ni''$ are even, then we have $\psi_{i',i''}\circ \psi_{i,i'} = \psi_{i,i''}$.
 
\subsection{The Bruhat-Tits stratification of the special fiber $\mathcal M_{\mathrm{red}}$}

We now recall the construction of the Bruhat-Tits stratification on $\mathcal M_{\mathrm{red}}$ as in \cite{vw2}. Let $i$ be an integer such that $ni$ is even. We define 
$$\mathcal L_i := \{\Lambda\subset \mathbf V \text{ a } \mathbb Z_{p^2}-\text{lattice}\,|\, p^{i+1}\Lambda^{\vee}\subsetneq \Lambda \subset p^i\Lambda^{\vee}\}.$$
If $\Lambda\in \mathcal L_i$, we define its \textbf{orbit type} $t(\Lambda):= [\Lambda:p^{i+1}\Lambda^{\vee}]$. We also call it the type of $\Lambda$. In particular, the lattices in $\mathcal L_i$ of type $1$ are precisely the $\mathbb F_{p^2}$-rational points of $\mathcal M_{i}$. By sending $\Lambda$ to $g(\Lambda)$, an element $g\in J(\mathbb Q_p)$ defines a map $\mathcal L_i\rightarrow \mathcal L_{i+\alpha(g)}$. The following Proposition follows from \cite{vw1} Remark 2.3 and \cite{vw2} Remark 4.1.

\begin{prop}
Let $i$ be an integer such that $ni$ is even and let $\Lambda \in \mathcal L_i$. 
\setlist{nolistsep}
\begin{enumerate}[label={--},noitemsep]
\item The map $\mathcal L_i\rightarrow \mathcal L_{i+\alpha(g)}$ induced by an element $g\in J(\mathbb Q_p)$ is an inclusion preserving, type preserving bijection.
\item We have $1\leq t(\Lambda) \leq n$. Furthermore $t(\Lambda)$ is odd.
\item The sets $\mathcal L_i$'s for various $i$'s are pairwise disjoint.
\end{enumerate} 
Moreover, two lattices $\Lambda, \Lambda' \in \bigsqcup_{ni \in 2\mathbb Z}\mathcal L_i$ are in the same orbit under the action of $J(\mathbb Q_p)$ if and only if $t(\Lambda) = t(\Lambda')$.
\end{prop}
 
We write $\mathcal L := \bigsqcup_{ni \in 2\mathbb Z}\mathcal L_i$. For any odd number $t$ between $1$ and $n$, there exists a lattice $\Lambda\in \mathcal L_0$ of orbit type $t$. Write $t_{\mathrm{max}} := 2\thetamax+1$, so that the orbit type $t$ of any lattice in $\mathcal L$ satisfies $1\leq t \leq t_{\mathrm{max}}$. The following lemma will be useful later. 

\begin{lem}\label{LatticeAndDual}
Let $i \in \mathbb Z$ such that $ni$ is even, and let $\Lambda \in \mathcal L_i$. We have $\Lambda^{\vee} \in \mathcal L$ if and only if either $n$ is even, either $n$ is odd and $t(\Lambda) = t_{\mathrm{max}}$. If $\Lambda^{\vee} \in \mathcal L$ and $n$ is even, then $\Lambda^{\vee} \in \mathcal L_{-i-1}$ and $t(\Lambda^{\vee}) = n - t(\Lambda)$. If on the contrary $n$ is odd, then $\Lambda^{\vee} \in \mathcal L_{-i}$ and $t(\Lambda^{\vee}) = t(\Lambda)$.
\end{lem}

\begin{proof}
First we prove the converse. We have the following chain of inclusions 
$$p^{-i}\Lambda \overset{n-t(\Lambda)}{\subset} \Lambda^{\vee} \overset{t(\Lambda)}{\subset} p^{-i-1}\Lambda.$$
If $n$ is even, then $-n(i+1)$ is also even and $n-t(\Lambda) \not = 0$. Since $(\Lambda^{\vee})^{\vee} = \Lambda$, we deduce that $\Lambda^{\vee} \in \mathcal L_{-i-1}$ with orbit type $n-t(\Lambda)$. Assume now that $n$ is odd and that $t(\Lambda) = t_{\mathrm{max}} = n$. Then $\Lambda^{\vee} = p^{-i}\Lambda \in \mathcal L_{-i}$.\\
Let us now assume that $\Lambda^{\vee} \in \mathcal L$ and that $n$ is odd. Let $i' \in 2\mathbb Z$ such that $\Lambda^{\vee} \in \mathcal L_{i'}$. We have 
\begin{align*}
\Lambda^{\vee} \overset{n-t(\Lambda^{\vee})}{\subset} p^{i'}\Lambda \overset{n-t(\Lambda)}{\subset} p^{i'+i}\Lambda^{\vee}, & & \Lambda^{\vee} \overset{t(\Lambda)}{\subset} p^{-i-1}\Lambda \overset{t(\Lambda^{\vee})}{\subset} p^{-i-i'-2}\Lambda^{\vee},
\end{align*}
therefore $-2\leq i+i' \leq 0$. Since $i+i'$ is even it is either $-2$ or $0$. If it were $-2$, then we would have $t(\Lambda) = t(\Lambda^{\vee}) = 0$ which is absurd. Therefore $i+i' = 0$, and we have $n-t(\Lambda) = n - t(\Lambda^{\vee}) = 0$.
\end{proof}

With the help of $\mathcal L_i$, one may construct an abstract simplicial complex $\mathcal B_i$. For $s\geq 0$, an $s$-simplex of $\mathcal B_i$ is a subset $S\subset \mathcal L_i$ of cardinality $s+1$ such that for some ordering $\Lambda_0,\ldots ,\Lambda_s$ of its elements, we have a chain of inclusions $p^{i+1}\Lambda_s^{\vee}\subsetneq \Lambda_0 \subsetneq \Lambda_1 \subsetneq \ldots \subsetneq \Lambda_s$. We must have $0\leq s \leq m$ for such a simplex to exist. Let $\tilde{J} = \mathrm{SU}(\mathbf V)$ be the derived group of $J$. We consider the abstract simplicial complex $\mathrm{BT}(\tilde J,\mathbb Q_p)$ of the Bruhat-Tits building of $\tilde{J}$ over $\mathbb Q_p$. A concrete description of this complex is given in \cite{vw1} Theorem 3.5.

\begin{theo}
The Bruhat-Tits building $\mathrm{BT}(\tilde J,\mathbb Q_p)$ is naturally identified with $\mathcal B_i$ for any fixed integer $i$ such that $ni$ is even. The set $\mathcal L_i$ is identified with the set of vertices of $\mathrm{BT}(\tilde J,\mathbb Q_p)$. The identification is $\tilde J(\mathbb Q_p)$-equivariant.
\end{theo}

Apartments in the Bruhat-Tits building $\mathrm{BT}(\tilde J,\mathbb Q_p)$ are in $1$ to $1$ correspondence with Witt decompositions of $\mathbf V$. Let $L = \{L_i\}_{i\in I}$ be a Witt decomposition of $\mathbf V$ and let $f = (f_i)_{i\in I} \sqcup B^{\mathrm{an}}$ be a basis of $\mathbf V$ adapted to the decomposition, where $f_i \in L_i$ and $B^{\mathrm{an}}$ is an orthogonal basis of $L^{\mathrm{an}}$. Under the identification of $\mathrm{BT}(\tilde J,\mathbb Q_p)$ with $\mathcal B_i$, the vertices inside the apartment associated to $L$ correspond to the lattices $\Lambda \in \mathcal L_i$ which are equal to the direct sum of $\Lambda\cap L^{\mathrm{an}}$ and of the modules $p^{r_i}\mathbb Z_{p^2}f_i$ for some integers $(r_i)_{i\in I}$. The subset of $\mathcal L_i$ consisting of all such lattices will be denoted $\mathcal A_i^L$ or, with an abuse of notations, $\mathcal A_i^f$. We call such a set $\mathcal A_i^L$ the \textbf{apartment associated to $L$ in $\mathcal L_i$.} We also define $\mathcal A^L := \bigsqcup_{ni\in 2\mathbb Z} \mathcal A_i^L$. We recall a general result regarding Bruhat-Tits buildings. 

\begin{prop}\label{CommonApartment}
Let $i$ be an integer such that $ni$ is even. Any two lattices $\Lambda$ and $\Lambda'$ in $\mathcal L_i$ lie inside a common apartment $\mathcal A_i^{L}$ for some Witt decomposition $L$. Moreover, the action of $\tilde J(\mathbb Q_p)$ acts transitively on the set of apartments $\{\mathcal A_i^L\}_L$.
\end{prop}

Recall the basis $e$ of $\mathbf V$ that we have fixed earlier. We will denote by 
$$\Lambda(r_{-\thetamax},\ldots ,r_{-1},s,r_1,\ldots ,r_{\thetamax})$$
the $\mathbb Z_{p^2}$-lattice generated by the vectors $p^{r_j}e_j$ for all $j = \pm 1,\ldots, \pm \thetamax$, by $p^{s_0}e^{\mathrm{an}}_0$ and if $n$ is even, by $p^{s_1}e^{\mathrm{an}}_1$ too. Here, the $r_j$'s are integers and $s$ denotes either the integer $s_0$ if $n$ is odd or the pair of integers $(s_0,s_1)$ if $n$ is even.

\begin{prop}\label{ConditionsLi} 
Let $i$ be an integer such that $ni$ is even. Let $(r_j,s)$ be a family of integers as above. The corresponding lattice $\Lambda = \Lambda(r_{-\thetamax},\ldots ,r_{-1},s,r_1,\ldots ,r_{\thetamax})$ belongs to $\mathcal L_i$ if and only if the following conditions are satisfied 
\setlist{nolistsep}
\begin{enumerate}[label={--},noitemsep]
\item for all $1\leq j \leq \thetamax$, we have $r_{-j} + r_j \in \{i,i+1\}$,
\item $s_0 = \lfloor \frac{i+1}{2}\rfloor$,
\item if $n$ is even, then $s_1 = \lfloor \frac{i}{2}\rfloor$.
\end{enumerate}
Moreover, when $\Lambda \in \mathcal L_i$ then
$$t(\Lambda) = 1 + 2\#\{1\leq j \leq \thetamax \,|\, r_{-j} + r_j = i\}.$$
\end{prop}
 
\begin{proof}
The lattice $\Lambda$ belongs to $\mathcal L_i$ if and only if the following chain of inclusions holds
$$p^{i+1}\Lambda^{\vee} \subsetneq \Lambda \subset p^i\Lambda^{\vee}.$$
The dual lattice $\Lambda^{\vee}$ is equal to the lattice $\Lambda(-r_{\thetamax},\ldots ,-r_{1},s',-r_{-1},\ldots ,-r_{-\thetamax})$, where $s' = -s_0$ when $n$ is odd, and $s' = (-s_0,-s_1-1)$ when $n$ is even. Thus, the inclusions above are equivalent to the following inequalities:
\begin{align*}
& i - r_{-j} \leq r_j \leq i + 1 - r_{-j}, & i - s_0 \leq s_0 \leq i + 1 - s_0,\\
& i - 1 - s_1 \leq s_1 \leq i - s_1 \text{ (if }n\text{ is even)}. &
\end{align*}
This proves the desired condition on the integers $r_j$'s and on $s$. Let us now assume that $\Lambda\in \mathcal L_i$. Its orbit type is equal to the index $[\Lambda,p^{i+1}\Lambda^{\vee}]$. This corresponds to the number of times equality occurs with the left-hand side in all the inequalities above. Of course, if the equality $i - r_{-j} = r_j$ occurs for some $j$, then it occurs also for $-j$. Moreover, if $i$ is even then the equality $i - s_0 = s_0$ occurs whereas $i - 1 - s_1 \not = s_1$. On the contrary if $i$ is odd, then the equality $i-1-s_1 = s_1$ occurs whereas $i - s_0 \not = s_0$. Thus in all cases, only one of $s_0$ and $s_1$ contributes to the index. Putting things together, we deduce the desired formula.
\end{proof} 

We deduce the following corollary.

\begin{corol}\label{ApartmentDescription}
The apartment $A_i^e$ (resp. $A^e$) consists of all the lattices of the form 
$$\Lambda = \Lambda(r_{-\thetamax},\ldots ,r_{-1},s,r_1,\ldots ,r_{\thetamax})$$ 
which belong to $\mathcal L_i$ (resp. to $\mathcal L$).
\end{corol}

\begin{proof}
According to the previous proposition, it is clear that all lattices which belong to $\mathcal L_i$ and are of the form $\Lambda(r_{-\thetamax},\ldots ,r_{-1},s,r_1,\ldots ,r_{\thetamax})$ are elements of $\mathcal A_i^e$. We shall prove the converse. Let $\Lambda \in \mathcal A_i^e$. By definition, there exists integers $(r_j)_j$ such that 
$$\Lambda = \Lambda\cap\mathbf V^{\mathrm{an}} \oplus \bigoplus_{1\leq j \leq \thetamax}\left(p^{r_{-j}}\mathbb Z_{p^2}e_{-j} \oplus p^{r_j}\mathbb Z_{p^2}e_j\right).$$
Write $\Lambda' = \Lambda \cap \mathbf V^{\mathrm{an}}$. This is a lattice in $\mathbf V^{\mathrm{an}}$ which satisfies the chain of inclusions 
$$p^{i+1}\Lambda'\,^{\vee} \subset \Lambda' \subset p^i\Lambda'\,^{\vee},$$
where the duals are taken with respect to the restriction of $\{\cdot,\cdot\}$ to $\mathbf V^{\mathrm{an}}$. Since $\mathbf V^{\mathrm{an}}$ is anisotropic, there is only a single lattice satisfying the chain of inclusions above. If we write $a := \lfloor \frac{i+1}{2}\rfloor$ and $b := \lfloor \frac{i}{2}\rfloor$, it is given by $p^{a}\mathbb Z_{p^2}e_0^{\mathrm{an}}$ if $n$ is odd, and by $p^{a}\mathbb Z_{p^2}e_0^{\mathrm{an}} \oplus p^{b}\mathbb Z_{p^2}e_1^{\mathrm{an}}$ if $n$ is even. Thus, it must be equal to $\Lambda'$ and it concludes the proof.
\end{proof}

We fix a maximal simplex in $\mathcal L_0$ lying inside the apartment $\mathcal A_0^e$. For $0\leq \theta \leq \thetamax$ we define 
$$\Lambda_{\theta} := \Lambda(\underbrace{0,\ldots,0}_{\thetamax},0,\underbrace{0,\ldots,0}_{\theta}, \underbrace{1,\ldots ,1}_{\thetamax-\theta}).$$
Here, the $0$ in the middle stands for $(0,0)$ in case $n$ is even. We have $t(\Lambda_{\theta}) = 2\theta+1$ and 
$$p\Lambda_0^{\vee}\subsetneq \Lambda_0 \subset \ldots \subset \Lambda_{\thetamax}.$$
Thus, they form an $\thetamax$-simplex in $\mathcal L_0$. Given a lattice $\Lambda\in \mathcal L_i$, a subfunctor $\mathcal M_{\Lambda}$ of $\mathcal M_{i,\mathrm{red}}$ is defined in \cite{vw2}, classifying those $p$-divisible groups for which a certain quasi-isogeny, depending on $\Lambda$, is in fact an actual isogeny. In Lemma 4.2, the authors prove that it is representable by a projective scheme over $\mathbb F_{p^2}$, and that the natural morphism $\mathcal M_{\Lambda} \hookrightarrow \mathcal M_{i,\mathrm{red}}$ is a closed immersion. The schemes $\mathcal M_{\Lambda}$ are called the \textbf{closed Bruhat-Tits strata of $\mathcal M$}. Their rational points are described as follows, see Lemma 4.3 of loc. cit. 

\begin{prop}
Let $k$ be a perfect field extension of $\mathbb F_{p^2}$, and let $M\in \mathcal M_{i,\mathrm{red}}(k)$. Then 
$$M\in \mathcal M_{\Lambda}(k) \iff M \subset \Lambda_k := \Lambda \otimes_{\mathbb Z_{p^2}} W(k).$$ 
\end{prop}

The set of lattices satisfying the condition above was conjectured in \cite{vw1} to be the set of points of a subscheme of $\mathcal M_{i,\mathrm{red}}$, and it was proved in the special cases $n=2,3$. In \cite{vw2}, the general argument is given by the construction of $\mathcal M_{\Lambda}$. The action of an element $g\in J(\mathbb Q_p)$ on $\mathcal{M}_{\mathrm{red}}$ induces an isomorphism $\mathcal M_{\Lambda} \xrightarrow{\sim} \mathcal M_{g\cdot \Lambda}$. \\
Let $\Lambda \in \mathcal L$. We denote by $J_{\Lambda}$ the fixator of $\Lambda$ under the action of $J(\mathbb Q_p)$. If $\Lambda = \Lambda_{\theta}$ for some $0\leq \theta \leq \thetamax$, we will write $J_{\theta}$ instead. These are \textbf{maximal parahoric subgroups} of $J(\mathbb Q_p)$. In unramified unitary similitude groups, maximal parahoric subgroups and maximal compact subgroups are the same. A general \textbf{parahoric subgroup} is an intersection $J_{\Lambda_1}\cap \ldots \cap J_{\Lambda_s}$ where $\{\Lambda_1,\ldots ,\Lambda_s\}$ is an $s$-simplex in $\mathcal L_i$ for some $i$. Any parahoric subgroup is compact and open in $J(\mathbb Q_p)$. Let $i$ be the integer such that $\Lambda \in \mathcal L_i$. We define $V_{\Lambda}^0:= \Lambda/p^{i+1}\Lambda^{\vee}$ and $V_{\Lambda}^1 := p^{i}\Lambda^{\vee}/\Lambda$. Since $p\Lambda \subset p\cdot p^i\Lambda^{\vee}$ and $p\cdot p^{i}\Lambda^{\vee} \subset \Lambda$, these are both $\mathbb F_{p^2}$-vector space of dimensions respectively $t(\Lambda)$ and $n-t(\Lambda)$. Both spaces come together with a non-degenerate $\sigma$-hermitian form $(\cdot,\cdot)_0$ and $(\cdot,\cdot)_1$ with values in $\mathbb F_{p^2}$, respectively induced by $p^{-i}\{\cdot,\cdot\}$ and by $p^{-i+1}\{\cdot,\cdot\}$. If $k$ is a perfect field extension of $\mathbb F_{p^2}$ and if $\epsilon \in \{0,1\}$, we may extend the pairings to $(V_{\Lambda}^{\epsilon})_k = V_{\Lambda}^{\epsilon}\otimes_{\mathbb F_{p^2}} k$ by setting 
$$(v\otimes x,w\otimes y)_{\epsilon}:=xy^{\sigma}(v,w)_{\epsilon}\in k$$
for all $v,w\in V_{\Lambda}^{\epsilon}$ and $x,y\in k$. If $U$ is a subspace of $(V_{\Lambda}^{\epsilon})_k$ we denote by $U^{\perp}$ its orthogonal. \\
Denote by $J_{\Lambda}^+$ the pro-unipotent radical of $J_{\Lambda}$ and write $\mathcal J_{\Lambda} := J_{\Lambda}/J_{\Lambda}^+$. This is a finite group of Lie type, called the \textbf{maximal reductive quotient} of $J_{\Lambda}$. We have an identification $\mathcal J_{\Lambda} \simeq \mathrm{G}(\mathrm{U}(V_{\Lambda}^0) \times \mathrm{U}(V_{\Lambda}^1))$, that is the group of pairs $(g_0,g_1)$ where for $\epsilon\in\{0,1\}$ we have $g_{\epsilon} \in \mathrm{GU}(V_{\Lambda}^{\epsilon})$ and $c(g_0) = c(g_1)$. Here, $c(g_{\epsilon}) \in \mathbb F_{p}^{\times}$ denotes the multiplier of $g_{\epsilon}$. For $0\leq \theta \leq \thetamax$ and $\epsilon \in \{0,1\}$, we will write $V_{\theta}^{\epsilon}$ and $\mathcal J_{\theta}$ instead of $V_{\Lambda_{\theta}}^{\epsilon}$ and $\mathcal J_{\Lambda_{\theta}}$.\\

Let $\Lambda \in \mathcal L_i$ where $ni$ is even. We write $t(\Lambda) = 2\theta +1$. Let $k$ be a perfect field extension of $\mathbb F_{p^2}$. Let $T$ be any $W(k)$-lattice in $\mathbf V_k$ such that 
$$p^{i+1}T^{\vee} \overset{2\theta'+1}{\subset} T \subset \Lambda_k$$
where $0 \leq \theta' \leq \theta$. Then $T$ must contain $p^{i+1}\Lambda_k^{\vee}$ and $[\Lambda_k:T] = \theta - \theta'$. We may consider $\overline{T} := T/p^{i+1}\Lambda_k^{\vee}$ the image of $T$ in $V^{(0)}_{\Lambda}$. Then $\overline{T}$ is an $\mathbb F_{p^2}$-subspace of dimension $\theta + \theta' + 1$. Moreover, one may check that $\overline{p^{i+1}T^{\vee}} = \overline{T}^{\perp}$, therefore the subspace $\overline{T}$ contains its orthogonal. These observations lead to the following proposition, see \cite{vw1} Proposition 2.7.

\begin{prop}\label{PointsOfMlambda}
The mapping $T \mapsto \overline{T}$ defines a bijection between the set of $W(k)$-lattices $T$ in $\mathbf V_k$ such that $p^{i+1}T^{\vee} \overset{2\theta'+1}{\subset} T \subset \Lambda_k$ and the set
$$\{U\subset (V_{\Lambda}^0)_k \,|\, \dim U = \theta+\theta'+1 \text{ and }U^{\perp}\subset U\}.$$
\end{prop}

In particular taking $\theta' = 0$, this set is in bijection with $\mathcal M_{\Lambda}(k)$.

\begin{rk}
Similarly, the set of $W(k)$-lattices $T$ such that $\Lambda_k \subset T \overset{n-2\theta'-1}{\subset} p^{i}T^{\vee}$ for some $\theta \leq \theta' \leq \thetamax$ is in bijection with 
$$\{U\subset (V_{\Lambda}^1)_k \,|\, \dim U = n - \theta' - \theta - 1 \text{ and }U^{\perp}\subset U\}.$$
The bijection is given by $T \mapsto \overline{T}^{\perp}$ where $\overline{T} := T/\Lambda_k \subset V^{(1)}_k$. These sets can be seen as the $k$-rational points of some flag variety for $\mathrm{GU}(V^{(0)}_{\Lambda})$ and $\mathrm{GU}(V^{(1)}_{\Lambda})$, which are special instances of Deligne-Lusztig varieties. This is accounted for in the next paragraph.
\end{rk}

Let $\Lambda \in \mathcal L$. The action of $J(\mathbb Q_p)$ on the Rapoport-Zink space $\mathcal M$ restricts to an action of the parahoric subgroup $J_{\Lambda}$ on the closed Bruhat-Tits stratum $\mathcal M_{\Lambda}$. This action factors through the maximal reductive quotient $\mathcal J_{\Lambda} \simeq \mathrm{G}(\mathrm{U}(V_{\Lambda}^0) \times \mathrm{U}(V_{\Lambda}^1))$. This action is trivial on the normal subgroup $\{\mathrm{id}\}\times \mathrm U(V_{\Lambda}^1) \subset \mathcal J_{\Lambda}$, thus it factors again through the quotient which is isomorphic to $\mathrm{GU}(V_{\Lambda}^0)$. With respect to this action, the variety $\mathcal M_{\Lambda}$ is isomophic to a generalized Deligne-Lusztig variety, see \cite{vw2} Theorem 4.8.

\begin{theo}\label{IsomorphismWithDL-Variety}
There is an isomorphism between $\mathcal M_{\Lambda}$ and a certain generalized parabolic Deligne-Lusztig variety for the finite group of Lie type $\mathrm{GU}(V_{\Lambda}^{0})$, compatible with the actions. In particular, if $t(\Lambda) = 2\theta + 1$ then the variety $\mathcal M_{\Lambda}$ is projective, smooth, geometrically irreducible of dimension $\theta$.
\end{theo}

We refer to \cite{muller} Section 1 for the definition of Deligne-Lusztig varieties. In particular, the adjective \enquote{generalized} is understood according to loc. cit. The Deligne-Lusztig variety isomorphic to $\mathcal M_{\Lambda}$ is introduced in \cite{vw2} Section 4.5, and it is denoted by $Y_{\Lambda}$ there. Theorem 5.1 of loc. cit. describes the incidence relations between the different strata.

\begin{theo}\label{IncidenceStrata}
Let $i\in \mathbb Z$ such that $ni$ is even. Let $\Lambda, \Lambda' \in \mathcal L_i$. The following statements hold.
\begin{enumerate}[label=\upshape (\arabic*)]
		\item The inclusion $\Lambda\subset \Lambda'$ is equivalent to the scheme-theoretic inclusion $\mathcal M_{\Lambda}\subset \mathcal M_{\Lambda'}$. It also implies $t(\Lambda)\leq t(\Lambda')$ and there is equality if and only if $\Lambda = \Lambda'$.
		\item The three following assertions are equivalent.
		\begin{align*}
		\mathrm{(i)}\;\Lambda\cap \Lambda'\in \mathcal L_i. & & \mathrm{(ii)}\; \Lambda\cap \Lambda' \text{ contains a lattice of }\mathcal L_i. & & \mathrm{(iii)}\;\mathcal M_{\Lambda}\cap \mathcal M_{\Lambda'} \not = \emptyset.
		\end{align*}
		If these conditions are satisfied, then $\mathcal M_{\Lambda}\cap \mathcal M_{\Lambda'}=\mathcal M_{\Lambda\cap \Lambda'}$, where we understand the left hand side as the scheme theoretic intersection inside $\mathcal M_{i,\mathrm{red}}$.
		\item The three following assertions are equivalent
		\begin{align*}
		& \mathrm{(i)}\;\Lambda+ \Lambda'\in \mathcal L_i. & & \mathrm{(ii)}\; \Lambda+ \Lambda' \text{ is contained in a lattice of }\mathcal L_i. & & \\ & \mathrm{(iii)}\;\mathcal M_{\Lambda}, \mathcal M_{\Lambda'} \subset \mathcal M_{\widetilde{\Lambda}} \text{ for some }\widetilde{\Lambda} \text{ in } \mathcal L_i.
		\end{align*}
		If these conditions are satisfied, then $\mathcal M_{\Lambda+ \Lambda'}$ is the smallest subscheme of the form $\mathcal M_{\widetilde{\Lambda}}$ containing both $\mathcal M_{\Lambda}$ and $\mathcal M_{\Lambda'}$.
		\item If $k$ is a perfect field extension of $\mathbb F_{p^2}$ then $\mathcal M_i(k)=\bigcup_{\Lambda\in \mathcal L_i}\mathcal M_{\Lambda}(k)$.
	\end{enumerate}
\end{theo}

In essence, the previous statements explain how the stratification given by the $\mathcal M_{\Lambda}$ mimics the combinatorics of the Bruhat-Tits building of $\tilde J$, hence the name.

\subsection{Normalizers of maximal parahoric subgroups of $J(\mathbb Q_p)$}

In this section we compute the normalizer of the maximal parahoric subgroups $J_{\Lambda}$.

\begin{lem}\label{SameParahoric}
Let $\Lambda, \Lambda' \in \mathcal L$. 
\setlist{nolistsep}
\begin{enumerate}[label=(\roman*),noitemsep]
\item The parahoric subgroup $J_{\Lambda}$ acts transitively on the set of apartments containing $\Lambda$.
\item We have $J_{\Lambda} = J_{\Lambda'}$ if and only if there exists $k\in \mathbb Z$ such that $\Lambda = p^k\Lambda'$ or $\Lambda = p^k\Lambda'\,^{\vee}$.
\end{enumerate}
\end{lem}

\begin{proof}
The first point is a general fact from the theory of Bruhat-Tits buildings. For the second point, the converse is clear. Indeed, if $x\in \mathbb Q_{p^2}^{\times}$ then $J_{x\Lambda} = J_{\Lambda}$, and an element $g\in J(\mathbb Q_p)$ fixes a lattice $\Lambda$ if and only if it fixes its dual $\Lambda^{\vee}$. Now, let $\Lambda, \Lambda' \in \mathcal L$ such that $J_{\Lambda} = J_{\Lambda'}$. Up to replacing $\Lambda'$ with an appropriate lattice $g\cdot \Lambda'$, it is enough to treat the case $\Lambda' = \Lambda_{\theta}$ for some $0\leq \theta \leq \thetamax$. By Proposition \ref{CommonApartment}, we can find an apartment $\mathcal A^L$ containing both $\Lambda_{\theta}$ and $\Lambda$. By the first point, we can find $g\in J_{\theta} = J_{\Lambda}$ which sends $\mathcal A^L$ to $\mathcal A^e$. Therefore $g\cdot \Lambda = \Lambda$ belongs to $\mathcal A^e$. According to Proposition \ref{ApartmentDescription}, we may write 
$$\Lambda = \Lambda(r_{-\thetamax},\ldots ,r_{-1},s,r_1,\ldots ,r_{\thetamax})$$
for some integers $(r_j,s)$. Let $i$ be the integer such that $\Lambda \in \mathcal L_i$. Then according to Proposition \ref{ConditionsLi} we have
\setlist{nolistsep}
\begin{enumerate}[label={--},noitemsep]
	\item $\forall 1\leq j \leq \thetamax, r_{-j} + r_{j} \in \{i,i+1\}$.
	\item $s_0 = \lfloor \frac{i+1}{2} \rfloor$.
	\item if $n$ is even then $s_1 = \lfloor \frac{i}{2} \rfloor$.
\end{enumerate}
For $1\leq j\leq \theta$, let $g_j$ be the automorphism of $\mathbf V$ which exchanges $e_{-j}$ and $e_j$ while fixing all the other vectors in the basis $e$. Then, from the definition of $\Lambda_{\theta}$ we have $g_j \in J_{\theta}$. Therefore $g_j$ must fix $\Lambda$ too, which implies that $r_{-j} = r_j$. And for $\theta + 1 \leq j \leq \thetamax$, let $g_j$ be the automorphism sending $e_{j}$ to $p^{-1}e_{-j}$ and $e_{-j}$ to $pe_j$ while fixing all the other vectors in the basis $e$. Then again we have $g_j \in J_{\theta} = J_{\Lambda}$ which implies that $r_{-j} = r_j - 1$.\\
Assume first that $i = 2i'$ is even. Combining the previous observations, we have $r_j = i'$ for all $1\leq j \leq \theta$ and $r_j = i' + 1$ for all $\theta + 1 \leq j \leq \thetamax$. Moreover we have $s_0 = i'$ and if $n$ is even, we have $s_1 = i'$. In other words, we have $\Lambda = p^{i'}\Lambda_{\theta}$.\\
Assume now that $i = 2i' + 1$ is odd. This implies that $n$ is even. Combining the previous observations, we have $r_j = i' + 1$ for all $1\leq j \leq \thetamax$. Moreover we have $s_0 = i' + 1$ and if $n$ is even, we have $s_1 = i'$. In other words, we have $\Lambda = p^{i'+1}\Lambda_{\theta}^{\vee}$.
\end{proof}

\begin{prop}\label{Normalizers}
Let $\Lambda \in \mathcal L$. If $t(\Lambda) \not = n - t(\Lambda)$ then the normalizer of $J_{\Lambda}$ in $J(\mathbb Q_p)$ is $\mathrm N_J(J_{\Lambda}) = \mathrm Z(J(\mathbb Q_p))J_{\Lambda}$. Otherwise, $n$ is even and there exists an element $h_0 \in J(\mathbb Q_p)$ such that $h_0^2 = p\cdot\mathrm{id}$ and $N_J(J_{\lambda})$ is the subgroup generated by $J_{\Lambda}$ and $h_0$. In particular, $\mathrm Z(J(\mathbb Q_p))J_{\Lambda}$ is a subgroup of index $2$ in $N_J(J_{\Lambda})$.
\end{prop}

\begin{rk} The condition $t(\Lambda) \not = n - t(\Lambda)$ is automatically satisfied if $n$ is odd. If $n$ is even, it is satisfied when $t(\Lambda) \not = \thetamax + 1$, this is the case in particular when $\thetamax$ is odd. 
\end{rk}

\begin{proof}
It is clear that $\mathrm Z(J(\mathbb Q_p))J_{\Lambda} \subset \mathrm N_J(J_{\Lambda})$. Conversely, let $g \in N_J(J_{\Lambda})$, so that we have $J_{\Lambda} = {}^gJ_{\Lambda} = J_{g\cdot \Lambda}$. We apply Lemma \ref{SameParahoric} to deduce the existence of $k\in \mathbb Z$ such that $g\cdot \Lambda = p^{k} \Lambda$ (case $1$) or $g\cdot \Lambda = p^{k} \Lambda^{\vee}$ (case $2$). If we are in case $1$, then $g \in p^{k}J_{\Lambda} \subset \mathrm Z(J(\mathbb Q_p))J_{\Lambda}$ and we are done. If $n$ is even, the assumption that $t(\Lambda) \not = n - t(\Lambda)$ makes the case $2$ impossible. If $n$ is odd and we are in case $2$, then in particular $\Lambda^{\vee} \in \mathcal L$. By Lemma \ref{LatticeAndDual}, we must have $\Lambda = p^{i}\Lambda^{\vee}$ for some even $i\in\mathbb Z$. In particular, we are also in case $1$. Therefore, no matter the parity of $n$, we are always in case $1$.\\
Assume now that $t(\Lambda) = n - t(\Lambda)$, in particular $n$ and $\thetamax$ are both even. We write $\thetamax = 2\thetamax'$ so that $t(\Lambda) = 2\thetamax'+1$ and we solve the case $\Lambda = \Lambda_{\thetamax'}$ first. Recall the element $g_0$ that was defined earlier. By direct computation, we see that $g_0\cdot \Lambda_{\thetamax'} = p\Lambda_{\thetamax'}^{\vee}$. Therefore ${}^{g_0}J_{\thetamax'} = J_{p\Lambda_{\thetamax'}^{\vee}} = J_{\thetamax'}$ so that $g_0 \in \mathrm{N}_J(J_{\thetamax'})$. Now let $g$ be any element normalizing $J_{\thetamax}$, so that $J_{\thetamax'} = {}^gJ_{\thetamax'} = J_{g\cdot \Lambda_{\thetamax'}}$. According to \ref{SameParahoric} there exists $k\in \mathbb Z$ such that $g\cdot \Lambda_{\thetamax'} = p^k\Lambda_{\thetamax'}$ or $g\cdot \Lambda_{\thetamax'} = p^k\Lambda_{\thetamax'}^{\vee} = p^{k-1}g_0\cdot\Lambda_{\thetamax'}$. In the first case we have $g \in p^kJ_{\thetamax'}$ and in the second case we have $g \in p^{k-1}g_0J_{\thetamax'}$. Since $g_0^2 = p\cdot\mathrm{id}$, the claim is proved with $h_0 = g_0$.\\
In the general case, we have $t(\Lambda) = 2\thetamax'+1 = t(\Lambda_{\thetamax'})$. There exists some $g\in J(\mathbb Q_p)$ such that $\Lambda = g\cdot \Lambda_{\thetamax'}$. Then $\mathrm{N}_J(\Lambda) = {}^g\mathrm{N}_J(\Lambda_{\thetamax'})$ so that the claim follows with $h_0 := gg_0g^{-1}$.
\end{proof}

\subsection{Counting the closed Bruhat-Tits strata}

In this section we count the number of closed Bruhat-Tits strata which contain or which are contained in another given one. Let $d\geq 0$ and consider $V$ a $d$-dimensional $\mathbb F_{p^2}$-vector space equipped with a non degenerate hermitian form. This structure is uniquely determined up to isomorphism as we are working over a finite field. For $\left\lceil\frac{d}{2}\right\rceil \leq r \leq d$, we define 
\begin{align*}
N(r,V) & := \{U \,|\, U \text{ is an }r\text{-dimensional subspace of }V \text{ such that }U^{\perp}\subset U\},\\
\nu(r,d) & := \# N(r,V),
\end{align*}
where $U^{\perp}$ denotes the orthogonal of $U$ with respect to the hermitian form on $V$. By Proposition \ref{PointsOfMlambda} and the following Remark, we have the following statement, see also \cite{vw2} Corollary 5.7. 

\begin{prop}\label{NumberStrata}
Let $\Lambda\in \mathcal L$. Write $t(\Lambda) = 2\theta +1$ for some $0\leq \theta \leq \thetamax$.
\setlist{nolistsep}
\begin{enumerate}[label={--},noitemsep]
\item Let $\theta'$ be an integer such that $0\leq \theta' \leq \theta$. The number of closed Bruhat-Tits strata of dimension $\theta'$ contained in $\mathcal M_{\Lambda}$ is $\nu(\theta + \theta' + 1, 2\theta + 1)$.
\item Let $\theta'$ be an integer such that $\theta \leq \theta' \leq \thetamax$. The number of closed Bruhat-Tits strata of dimension $\theta'$ containing $\mathcal M_{\Lambda}$ is $\nu(n - \theta - \theta' - 1, n - 2\theta - 1)$.
\end{enumerate}
\end{prop}

Another way to formulate the proposition is to say that $\nu(\theta + \theta' + 1, 2\theta + 1)$ (resp. $\nu(n - \theta - \theta' - 1, n - 2\theta - 1)$) is the number of vertices of type $2\theta' + 1$ in the Bruhat-Tits building of $\tilde{J}$ which are neighbors of a given vertex of type $2\theta + 1$ for $\theta' \leq \theta$ (resp. $\theta' \geq \theta$). In \cite{vw2}, an explicit formula is given for $\nu(d-1,d)$. The next proposition gives a formula to compute $\nu(r,d)$ for general $r$ and $d$. 

\begin{prop}\label{ExplicitNumberStrata}
Let $d\geq 0$ and let $\left\lceil\frac{d}{2}\right\rceil \leq r \leq d$. We have 
$$\nu(r,d) = \frac{\prod_{j=1}^{2(d-r)}\left(p^{2r-d+j} - (-1)^{2r-d+j}\right)}{\prod_{j=1}^{d-r}\left(p^{2j} -1\right)}$$
\end{prop} 

\begin{proof}
We fix a basis $(e_1,\ldots ,e_d)$ of $V$ in which the hermitian form is represented by the matrix $A_d$. We denote by $U_0$ the subspace generated by the vectors $e_1,\ldots, e_r$. The orthogonal of $U_0$ is generated by $e_1,\ldots , e_{d-r}$. Since $\left\lceil\frac{d}{2}\right\rceil \leq r \leq d$, $U_0$ contains its orthogonal, thus $U_0 \in N(r,V)$. The unitary group $\mathrm{U}(V) \simeq \mathrm U_{d}(\mathbb F_p)$ acts transitively on the set $N(r,V)$ (since $p \not = 2$). The stabilizer of $U_0$ in $\mathrm U_{d}(\mathbb F_p)$ is the standard parabolic subgroup 
$$P_0 := \left\{ 
\begin{pmatrix}
B & * & *\\
0 & M & *\\
0 & 0 & F(B)
\end{pmatrix} \in \mathrm U_{d}(\mathbb F_p) \; \middle| \; B \in \mathrm{GL}_{d-r}(\mathbb F_{p^2}), M\in \mathrm U_{2r-d}(\mathbb F_p)\right\}.$$
Here, $F(B) = A_{d-r}(B^{(p)})^{-T}A_{d-r}$ where $B^{(p)}$ is the matrix $B$ with all coefficients raised to the power $p$. The order of $\mathrm U_d(\mathbb F_p)$ is well known and given by the formula 
$$\#\mathrm U_d(\mathbb F_p) = p^{\frac{d(d-1)}{2}}\prod_{j=1}^{d} \left(p^j - (-1)^j\right).$$
It remains to compute the order of $P_0$. We have a Levi decomposition $P_0 = L_0N_0$ with $L_0 \cap N_0 = \{1\}$ where 
\begin{align*}
L_0 & := \left\{ 
\begin{pmatrix}
B & 0 & 0\\
0 & M & 0\\
0 & 0 & F(B)
\end{pmatrix} \in \mathrm U_{d}(\mathbb F_p) \; \middle| \; B \in \mathrm{GL}_{d-r}(\mathbb F_{p^2}), M\in \mathrm U_{2r-d}(\mathbb F_p)\right\},  \\
N_0 & := \left\{ 
\begin{pmatrix}
1 & X & Z\\
0 & 1 & Y\\
0 & 0 & 1
\end{pmatrix} \in \mathrm U_{d}(\mathbb F_p) \; \middle| \; X \in \mathrm{M}_{d-r,2r-d}(\mathbb F_{p^2}), Y\in \mathrm M_{2r-d,d-r}(\mathbb F_{p^2}), Z\in \mathrm M_{d-r}(\mathbb F_{p^2})\right\}.
\end{align*}
The order of $L_0$ is given by 
$$\#L_0 = \#\mathrm{GL}_{d-r}(\mathbb F_{p^2})\#\mathrm U_{2r-d}(\mathbb F_p) = p^{(d-r)(d-r-1) + \frac{(2r-d)(2r-d-1)}{2}}\prod_{j=1}^{d-r}\left(p^{2j}-1\right)\prod_{j=1}^{2r-d} \left(p^j - (-1)^j\right).$$
As for $N_0$, we need some more conditions on the matrices $X, Y$ and $Z$. By direct computations, one may check that 
\begin{align*}
\begin{pmatrix}
1 & X & Z\\
0 & 1 & Y\\
0 & 0 & 1
\end{pmatrix} \in N_0 \iff  
Y = -A_{2r-d}(X^{(p)})^TA_{d-r} \text{ and } Z + A_{d-r}(Z^{(p)})^TA_{d-r} = XY \in \mathrm M_{d-r}(\mathbb F_{p^2}).
\end{align*}
Thus, $X$ is any matrix of size $(d-r) \times (2r-d)$ and $Y$ is determined by $X$. In the second equation, the matrix $A_{d-r}(Z^{(p)})^TA_{d-r}$ is the reflexion of $Z^{(p)}$ with respect to the antidiagonal. The equation implies that the coefficients below the antidiagonal of $Z$ determine those above the antidiagonal. Furthermore, if $z$ is a coefficient in the antidiagonal then the equation determines the value of $\mathrm{Tr}(z) = z + z^p$, where $\mathrm{Tr}: \mathbb F_{p^2} \rightarrow \mathbb F_p$ is the trace relative to the extension $\mathbb F_{p^2}/\mathbb F_p$. The trace is surjective and its kernel has order $p$. Thus, there are only $p$ possibilities for each antidiagonal coefficient. Putting things together, the order of $N_0$ is given by 
$$\#N_0 = p^{2(d-r)(2r-d)}\cdot p^{2\frac{(d-r)(d-r-1)}{2}}\cdot p^{d-r} = p^{(d-r)(3r-d)}$$
where the three terms take account respectively of the choice of $X$, the choice of the coefficients below the antidiagonal of $Z$ and the choice of the coefficients in the antidiagonal of $Z$. Hence the order of $P_0$ is given by 
$$\#P_0 = \#L_0\#N_0 = p^{\frac{d(d-1)}{2}}\prod_{j=1}^{d-r}\left(p^{2j}-1\right)\prod_{j=1}^{2r-d} \left(p^j - (-1)^j\right).$$
Upon taking the quotient $\nu(r,d) = \#\mathrm U_d(\mathbb F_p)/\#P_0$, the result follows.
\end{proof}

In particular with $r = d-1$, we obtain 
$$\nu(d-1,d) = \frac{(p^{d-1}-(-1)^{d-1})(p^d-(-1)^{d})}{p^2-1}.$$
If $d = 2\delta$ is even, it is equal to $(p^{d-1}+1)\sum_{j=0}^{\delta-1}p^{2j}$, and if $d = 2\delta + 1$ is odd, it is equal to $(p^{d}+1)\sum_{j=0}^{\delta-1}p^{2j}$. This coincides with the formula given in \cite{vw2} Example 5.6.

\section{The cohomology of a closed Bruhat-Tits stratum}

In \cite{muller}, we computed the cohomology groups $\mathrm{H}^{\bullet}_c(\mathcal M_{\Lambda}\otimes \mathbb F,\overline{\mathbb Q_{\ell}})$ of the closed Bruhat-Tits strata. The computation relies on the Ekedahl-Oort stratification on $\mathcal M_{\Lambda}$ which, in the language of Deligne-Lusztig varieties, translates into a stratification by Coxeter varieties for unitary groups of smaller sizes. The cohomology of Coxeter varieties is well known thanks to the work of Lusztig in \cite{cox}. In order to state our results, we recall the classification of unipotent representations of the finite unitary group.

Let $q$ be a power of prime number $p$, and let $\mathbf G$ be a reductive connected group over an algebraic closure $\mathbb F$ of $\mathbb F_p$. Assume that $\mathbf G$ is equipped with an $\mathbb F_q$-structure induced by a Frobenius morphism $F$. Let $G = \mathbf G^F$ be the associated finite group of Lie type. Let $(\mathbf T,\mathbf B)$ be a pair consisting of an $F$-stable maximal torus $\mathbf T$ and an $F$-stable Borel subgroup $\mathbf B$ containing $\mathbf T$. Let $\mathbf W = \mathbf W(\mathbf T)$ denote the Weyl group of $\mathbf G$. The Frobenius $F$ induces an action on $\mathbf W$. For $w\in \mathbf W$, let $\dot{w}$ be a representative of $w$ in the normalizer $\mathrm N_{\mathbf G}(\mathbf T)$ of $\mathbf T$. By the Lang-Steinberg theorem, one can find $g\in \mathbf G$ such that $\dot{w} = g^{-1}F(g)$. Then ${}^g\mathbf T := g\mathbf T g^{-1}$ is another $F$-stable maximal torus, and $w \in \mathbf W$ is said to be the \textbf{type} of $^g \mathbf T$ with respect to $\mathbf T$. Every $F$-stable maximal torus arises in this manner. According to \cite{dl} Corollary 1.14, the $G$-conjugacy class of $^g \mathbf T$ only depends on the $F$-conjugacy class of $w$ in the Weyl group $\mathbf W$. Here, two elements $w$ and $w'$ in $\mathbf W$ are said to be $F$-conjugate if there exists some element $\tau \in \mathbf W$ such that $w = \tau w' F(\tau)^{-1}$. For every $w\in \mathbf W$, we fix $\mathbf T_w$ an $F$-stable maximal torus of type $w$ with respect to $\mathbf T$. The Deligne-Lusztig induction of the trivial representation of $\mathbf T_w$ is the virtual representation of $G$ defined by the formula 
$$R_w := \sum_{i\geq 0} (-1)^i\mathrm H^i_c(X(w)\otimes \, \mathbb F,\overline{\mathbb Q_{\ell}}),$$
where $X(w)$ is the Deligne-Lusztig variety for $\mathbf G$ given by 
$$X(w) := \{g\mathbf B \in \mathbf G/\mathbf B \,|\, g^{-1}F(g)\in \mathbf B w \mathbf B\}.$$ 
According to \cite{dl} Theorem 1.6, the virtual representation $R_w$ only depends on the $F$-conjugacy class of $w$ in $\mathbf W$. An irreducible representation of $G$ is said to be \textbf{unipotent} if it occurs in $R_w$ for some $w\in \mathbf W$. The set of isomorphism classes of unipotent representations of $G$ is denoted by $\mathcal E(G,1)$. 

\begin{rk}
Since the center $\mathrm{Z}(G)$ acts trivially on the variety $X(w)$, any irreducible unipotent representation of $G$ has trivial central character.
\end{rk}

Let $\mathbf G$ and $\mathbf G'$ be two reductive connected group over $\mathbb F$ both equipped with an $\mathbb F_q$-structure. We denote by $F$ and $F'$ the respective Frobenius morphisms. Let $f: \mathbf G\rightarrow \mathbf G'$ be an $\mathbb F_q$-isotypy, that is a homomorphism defined over $\mathbb F_q$ whose kernel is contained in the center of $\mathbf G$ and whose image contains the derived subgroup of $\mathbf G'$. Then, according to \cite{dm} Proposition 11.3.8, we have an equality 
$$\mathcal E(G,1) = \{\rho\circ f \,|\, \rho \in \mathcal E(G',1)\}.$$
Thus, the irreducible unipotent representations of $G$ and of $G'$ can be identified. We will use this observation in the case $G = \mathrm{U}_k(\mathbb F_q)$ and $G' = \mathrm{GU}_k(\mathbb F_q)$. The corresponding reductive groups are $\mathbf G = \mathrm{GL}_k$ and $\mathbf G' = \mathrm{GL}_k \times \mathrm{GL}_1$. The Frobenius morphisms can be defined as 
\begin{align*}
F(M) = \dot{w_0}(M^{(q)})^{-T}\dot{w_0}, & & F'(M,c) = (c^q\dot{w_0}(M^{(q)})^{-T}\dot{w_0}, c^q).
\end{align*}
Here, $\dot{w_0}$ is the $k\times k$ matrix with only $1$'s in the antidiagonal and $M^{(q)}$ is the matrix $M$ whose entries are all raised to the power $q$. The isotypy $f: \mathbf G\rightarrow \mathbf G'$ is defined by $f(M) = (M,1)$. It satisfies $F'\circ f = f \circ F$, it is injective and its image contains the derived subgroup $\mathrm{SL}_n\times \{1\} \subset \mathbf G'$. Hence, we obtain the following result. 

\begin{prop}\label{SameUnipotent}
The irreducible unipotent representations of the finite groups of Lie type $\mathrm U_k(\mathbb F_q)$ and $\mathrm{GU}_k(\mathbb F_q)$ can be naturally identified.
\end{prop}

Assume that the Coxeter graph of the reductive group $\mathbf G$ is a union of subgraphs of type $A_m$ (for various $m$). Let $\widecheck{\mathbf W}$ be the set of isomorphism classes of irreducible representations of its Weyl group $\mathbf W$. The action of the Frobenius $F$ on $\mathbf W$ induces an action on $\widecheck{\mathbf W}$, and we consider the fixed point set $\widecheck{\mathbf W}^F$. The following theorem of \cite{ls} classifies the irreducible unipotent representations of $G$.

\begin{theo}\label{ClassificationUnipotent}
There is a bijection between $\widecheck{\mathbf W}^F$ and the set of isomorphism classes of irreducible unipotent representations of $G$. 
\end{theo}

We recall how the bijection is constructed. According to loc. cit. if $V\in \widecheck{\mathbf W}^F$ there is a unique automorphism $\widetilde{F}$ of $V$ of finite order such that 
$$R(V) := \frac{1}{|\mathbf W|}\sum_{w\in \mathbf W} \mathrm{Trace}(w\circ \widetilde{F} \,|\, V)R_w$$
is an irreducible representation of $G$. Then the map $V \mapsto R(V)$ is the desired bijection. In the case of $\mathrm U_k(\mathbb F_q)$ or $\mathrm{GU}_k(\mathbb F_q)$, the Weyl group $\mathbf W$ is identified with the symmetric group $\mathfrak S_k$ and we have an equality $\widecheck{\mathbf W}^F = \widecheck{\mathbf W}$. Moreover, the automorphism $\widetilde{F}$ is the multiplication by $w_0$, where $w_0$ is the element of maximal length in $\mathbf W$. Thus, in both cases the irreducible unipotent representations of $G$ are classified by the irreducible representations of the Weyl group $\mathbf W\simeq \mathfrak S_k$, which in turn are classified by partitions of $k$ or equivalently by Young diagrams, as we briefly recall in the next paragraph.
 
A partition of $k$ is a tuple of integers $\lambda = (\lambda_1 \geq \ldots \geq \lambda_r>0)$ with $r\geq 1$ such that $\lambda_1 + \ldots + \lambda_r = k$. The integer $k$ is called the length of the partition, and it is denoted by $|\lambda|$. A Young diagram of size $k$ is a top left justified collection of $k$ boxes, arranged in rows and columns. There is a correspondence between Young diagrams of size $k$ and partitions of $k$, by associating to a partition $\lambda = (\lambda_1, \ldots, \lambda_r)$ the Young diagram having $r$ rows consisting successively of $\lambda_1, \ldots, \lambda_r$ boxes. We will often identify a partition with its Young diagram, and conversely. For example, the Young diagram associated to $\lambda = (3^2,2^2,1)$ is the following one. 
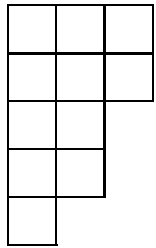
\begin{center}\ydiagram{3,3,2,2,1}\end{center}
To any partition $\lambda$ of $k$, one can naturally associate an irreducible character $\chi_{\lambda}$ of the symmetric group $\mathfrak S_k$. An explicit construction is given, for instance, by the notion of Specht modules as explained in \cite{james} 7.1. 

The irreducible unipotent representation of $\mathrm U_k(\mathbb F_q)$ (resp. $\mathrm{GU}_k(\mathbb F_q)$) associated to $\chi_{\lambda}$ by the bijection of Theorem \ref{ClassificationUnipotent} is denoted by $\rho_{\lambda}^{\mathrm U}$ (resp. $\rho_{\lambda}^{\mathrm{GU}}$). In virtue of Proposition \ref{SameUnipotent}, for every $\lambda$ we have $\rho_{\lambda}^{\mathrm U} = \rho_{\lambda}^{\mathrm{GU}} \circ f$, where $f: \mathrm U_k(\mathbb F_q) \rightarrow \mathrm{GU}_k(\mathbb F_q)$ is the inclusion. Thus, it is harmless to identify $\rho_{\lambda}^{\mathrm U}$ and $\rho_{\lambda}^{\mathrm{GU}}$ so that from now on, we will omit the superscript. The partition $(k)$ corresponds to the trivial representation and $(1^k)$ to the Steinberg representation.  Given a box $\msquare$ in the Young diagram of $\lambda$, its \textbf{hook length} $h(\msquare)$ is $1$ plus the number of boxes lying below it or on its right. For instance, in the following figure the hook length of every box of the Young diagram of $\lambda = (3^2,2^2,1)$ has been written inside it.
\ytableausetup{mathmode, centertableaux}
\begin{center}
\begin{ytableau}
7 & 5 & 2 \\
6 & 4 & 1 \\
4 & 2 \\
3 & 1 \\
1
\end{ytableau}
\end{center}

The degree of the representations $\rho_{\lambda}$ is given by expressions known as \textbf{hook formula}, see for instance \cite{geck2} Proposition 4.3.5.

\begin{prop} 
Let $\lambda = (\lambda_1 \geq \ldots \geq \lambda_r > 0)$ be a partition of $k$. The degree of the irreducible unipotent representation $\rho_{\lambda}$ is given by the following formula
$$\deg(\rho_{\lambda}) = q^{a(\lambda)}\frac{\prod_{i=1}^k q^i - (-1)^i}{\prod_{\msquare \in \lambda} q^{h(\msquare)} - (-1)^{h(\msquare)}}$$
where $a(\lambda) = \sum_{i=1}^r (i-1)\lambda_i$. 
\end{prop}

We may describe the cuspidal support of the unipotent representations $\rho_{\lambda}$. According to \cite{classical} Propositions 9.2 and 9.4 there exists an irreducible unipotent cuspidal representation of $\mathrm U_k(\mathbb F_q)$ (or $\mathrm{GU}_k(\mathbb F_q)$) if and only if $k$ is an integer of the form $k = \frac{t(t+1)}{2}$ for some $t\geq 0$. When $k$ is an integer of this form, the unique unipotent cuspidal representation is associated to the partition $\Delta_t := (t, t-1,\ldots,1)$, whose Young diagram has the distinctive shape of a staircase. Here, as a convention $U_0(\mathbb F_q) = \mathrm{GU}_0(\mathbb F_q)$ denotes the trivial group. For example, here are the Young diagrams of $\Delta_1,\Delta_2$ and $\Delta_3$. Of course, the one of $\Delta_0$ the empty diagram.
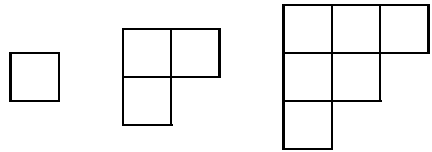
\begin{center}\ydiagram{1} \quad \quad \ydiagram{2,1} \quad \quad \ydiagram{3,2,1}\end{center}

We consider an integer $t\geq 0$ such that $k$ decomposes as $k = 2e + \frac{t(t+1)}{2}$ for some $e\geq 0$. Let $G$ denote $\mathrm{U}_k(\mathbb F_q)$ or $\mathrm{GU}_k(\mathbb F_q)$, and consider $L_t$ the subgroup consisting of block-diagonal matrices having one middle block of size $\frac{t(t+1)}{2}$ and all other blocks of size $1$. This is a standard Levi subgroup of $G$. For $\mathrm U_k(\mathbb F_q)$, it is isomorphic to $\mathrm{GL}_1(\mathbb F_{q^2})^e\times \mathrm U_{\frac{t(t+1)}{2}}(\mathbb F_q)$ whereas in the case of $\mathrm{GU}_k(\mathbb F_q)$ it is isomorphic to $\mathrm{G}\left(\mathrm{U}_1(\mathbb F_{q})^e\times \mathrm{U}_{\frac{t(t+1)}{2}}(\mathbb F_q)\right)$. In both cases, $L_t$ admits a quotient which is isomorphic to a group of the same type as $G$ but of size $\frac{t(t+1)}{2}$. We write $\rho_t$ for the inflation to $L_t$ of the unipotent cuspidal representation $\rho_{\Delta_t}$ of this quotient. If $\lambda$ is a partition of $k$, the cuspidal support of the representation $\rho_{\lambda}$ is given by exactly one of the pair $(L_t,\rho_t)$ up to conjugation, where $t\geq 0$ is an integer such that for some $e\geq 0$ we have $k = 2e + \frac{t(t+1)}{2}$. Note that in particular $k$ and $\frac{t(t+1)}{2}$ must have the same parity. With these notations, the irreducible unipotent representations belonging to the principal series (ie. those whose cuspidal support is supported on a minimal parabolic subgroup) are those with cuspidal support $(L_0,\rho_0)$ if $k$ is even and $(L_1,\rho_1)$ if $k$ is odd.

Given an irreducible unipotent representation $\rho_{\lambda}$, there is a combinatorical way to determine the Harish-Chandra series to which it belongs, as we recalled in \cite{muller} Section 2. We consider the Young diagram of $\lambda$. We call \textbf{domino} any pair of adjacent boxes in the diagram. It may be either vertical or horizontal. We remove dominoes from the diagram of $\lambda$ so that the resulting shape is again a Young diagram, until one can not proceed further. This process results in the Young diagram of the partition $\Delta_t$ for some $t\geq 0$, and it is called the \textbf{$2$-core} of $\lambda$. It does not depend on the successive choices for the dominoes. Then, the representation $\rho_{\lambda}$ has cuspidal support $(L_t,\rho_t)$ if and only if $\lambda$ has $2$-core $\Delta_t$. For instance, the diagram $\lambda = (3^2,2^2,1)$ has $2$-core $\Delta_1$, as it can be determined by the following steps. We put crosses inside the successive dominoes that we remove from the diagram. 
\begin{center}
\ydiagram{3,3,1,1,1}*[\times]{0,0,1+1,1+1} $\implies$ \ydiagram{3,1,1,1,1}*[\times]{0,1+2} $\implies$ \ydiagram{3,1,1}*[\times]{0,0,0,1,1} $\implies$ \ydiagram{1,1,1}*[\times]{1+2} $\implies$ \ydiagram{1}*[\times]{0,1,1} $\implies$ \ydiagram{1}
\end{center}
Thus, the unipotent representation $\rho_{\lambda}$ of $\mathrm U_{11}(\mathbb F_q)$ or $\mathrm{GU}_{11}(\mathbb F_q)$ has cuspidal support $(L_1,\rho_1)$, so in particular it is a principal series representation.

From now on, we take $q = p$. Let $\Lambda \in \mathcal L$ with orbit type $t(\Lambda) = 2\theta + 1$. Recall that the stratum $\mathcal M_{\Lambda}$ is equipped with an action of the finite group of Lie type $\mathrm{GU}(V_{\Lambda}^0)$. Upon choosing a basis, we identify this group with $\mathrm{GU}_{2\theta+1}(\mathbb F_p)$. Let $\mathrm{Frob} = \sigma^{-2} \in \mathrm{Gal}(\mathbb F/\mathbb F_{p^2})$ be the geometric Frobenius. Then $\mathrm{Frob}$ is a topological generator of $\mathrm{Gal}(\mathbb F/\mathbb F_{p^2})$. In \cite{muller}, we computed the cohomology groups $\mathrm{H}^{\bullet}(\mathcal M_{\Lambda}\otimes \mathbb F,\overline{\mathbb Q_{\ell}})$ in terms of a $\mathrm{GU}_{2\theta+1}(\mathbb F_p) \times \langle \mathrm{Frob} \rangle$-representations. The result is summed up in the following Theorem.

\begin{theo}\label{CohomologyBT-Stratum}
Let $\Lambda \in \mathcal L$ and write $t(\Lambda) = 2\theta +1$ for some $0\leq \theta \leq \thetamax$.
\begin{enumerate}[label=\upshape (\arabic*)]
		\item The cohomology group $\mathrm H^j(\mathcal M_{\Lambda}\otimes \mathbb F,\overline{\mathbb Q_{\ell}})$ is zero unless $0 \leq j \leq 2\theta$. 
		\item The Frobenius $\mathrm{Frob}$ acts like multiplication by $(-p)^j$ on $\mathrm H^j(\mathcal M_{\Lambda}\otimes \mathbb F,\overline{\mathbb Q_{\ell}})$. 
		\item For $0\leq j \leq \theta$ we have 
		$$\mathrm H^{2j}(\mathcal M_{\Lambda}\otimes \mathbb F,\overline{\mathbb Q_{\ell}}) = \bigoplus_{s=0}^{\min(j,\theta - j)} \rho_{(2\theta + 1 - 2s, 2s)}.$$
		For $0\leq j \leq \theta - 1$ we have 
		$$\mathrm H^{2j+1}(\mathcal M_{\Lambda}\otimes \mathbb F,\overline{\mathbb Q_{\ell}}) = \bigoplus_{s=0}^{\min(j,\theta - 1 - j)} \rho_{(2\theta - 2s, 2s + 1)}.$$
	\end{enumerate}
\end{theo} 

Thus, the cohomology of $\mathcal M_{\Lambda}$ consists only of unipotent representations whose associated Young diagram has at most two rows.

\begin{rks} Let us make a few comments.
\setlist{nolistsep}
\begin{enumerate}[label={--},noitemsep]
\item The cohomology groups of index $0$ and $2\theta$ are the trivial representation of $\mathrm{GU}_{2\theta+1}(\mathbb F_p)$.
\item All irreducible representations in the cohomology groups of even index belong to the unipotent principal series, whereas all the ones in the groups of odd index have cuspidal support $(L_2,\rho_2)$.
\item The cohomology group $\mathrm H^j(\mathcal M_{\Lambda}\otimes \mathbb F,\overline{\mathbb Q_{\ell}})$ contains no cuspidal representation unless $\theta = j = 0$ or $\theta = j = 1$. If $\theta = 0$ then $\mathrm H^0$ is the trivial representation of $\mathrm{GU}_1(\mathbb F_p) = \mathbb F_{p^2}^{\times}$, and if $\theta = 1$ then $\mathrm H^1$ is the representation $\rho_{\Delta_2}$ of $\mathrm{GU}_3(\mathbb F_p)$. Both of them are cuspidal.
\end{enumerate}
\end{rks}

\section{Shimura variety and $p$-adic uniformization of the supersingular locus}\label{Shimura}

In this section, we introduce the PEL unitary Shimura variety with signature $(1,n-1)$ as in \cite{vw2} Sections 6.1 and 6.2, and we recall the $p$-adic uniformization theorem of its basic (or supersingular) locus. The Shimura variety can be defined as a moduli problem classifying abelian varieties with additional structures, as follows. Let $\mathbb E$ be a quadratic imaginary extension of $\mathbb Q$ such that $\mathbb E_p \simeq E$. In particular \textbf{$p$ is inert} in $\mathbb E$. Let $B/\mathbb E$ be a simple central algebra of degree $d\geq 1$ which splits over $p$ and at infinity. Let $*$ be a positive involution of the second kind on $B$, and let $\mathbb V$ be a non-zero finitely generated left $B$-module equipped with a non-degenerate $*$-alternating form $\langle\cdot,\cdot\rangle$ taking values in $\mathbb Q$. Assume also that $\dim_{\mathbb E}(\mathbb V) = nd$. Let $\mathbb G$ be the connected reductive group over $\mathbb Q$ whose points over a $\mathbb Q$-algebra $R$ are given by 
$$\mathbb G(R) := \{g \in \mathrm{GL}_{\mathbb E\otimes R}(\mathbb V\otimes R) \,|\, \exists c\in R^{\times} \text{ such that for all } v,w \in \mathbb V\otimes R, \langle gv,gw\rangle = c\langle v,w\rangle \}.$$
We denote by $c:\mathbb G\rightarrow \mathbb G_{m}$ the \textbf{multiplier} character. The base change $\mathbb G_{\mathbb R}$ is isomorphic to a group of unitary similitudes $\mathrm{GU}(r,s)$ of a hermitian space with signature $(r,s)$ where $r+s=n$. We assume that $r=1$ and $s=n-1$. We consider a Shimura datum of the form $(\mathbb G,X)$, where $X$ denotes the unique $\mathbb G(\mathbb R)$-conjugacy class of homorphisms $h:\mathbb C^{\times} \rightarrow \mathbb G_{\mathbb R}$ such that for all $z\in \mathbb C^{\times}$ we have $\langle h(z)\cdot,\cdot \rangle = \langle \cdot, h(\overline{z})\cdot\rangle$, and such that the $\mathbb R$-pairing $\langle\cdot,h(i)\cdot\rangle$ is positive definite. Such a homomorphism $h$ induces a decomposition $\mathbb V\otimes \mathbb C = \mathbb V_1 \oplus \mathbb V_2$. Concretely, $\mathbb V_1$ (resp. $\mathbb V_2$) is the subspace where $h(z)$ acts like $z$ (resp. like $\overline{z}$). Let $F$ be the unique subfield of $\mathbb C$ isomorphic to $\mathbb E$. The reflex field associated to the PEL data, that is the field of definition of $\mathbb V_1$ as a complex representation of $B$, is equal to $F$ unless $n=2$, in which case it is $\mathbb Q$. Nonetheless, for simplicity we will consider the associated Shimura varieties over $F$ even in the case $n=2$. 

\begin{rk}
As remarked in \cite{vw1} Section 6, the group $G$ satisfies the Hasse principle, ie. $\mathrm{ker}^1(\mathbb Q,\mathbb G)$ is a singleton. Therefore, the Shimura variety associated to the Shimura datum $(G,X)$ coincides with the moduli space of abelian varieties that we are going to define.
\end{rk}

Let $\mathbb A_f$ denote the ring of finite adèles over $\mathbb Q$ and let $K\subset G(\mathbb A_f)$ be an open compact subgroup. We define a functor $\mathrm{Sh}_K$ by associating to an $F$-scheme $S$ the set of isomorphism classes of tuples $(A,\lambda_A,\iota_A,\overline{\eta}_A)$ where
\begin{enumerate}[label={--}]
\item $A$ is an abelian scheme over $S$.
\item $\lambda_A: A\rightarrow \widehat{A}$ is a polarization.
\item $\iota_A:B\rightarrow \mathrm{End}(A)\otimes \mathbb Q$ is a morphism of algebras such that $\iota_A(b^*) = \iota_A(b)^{\dagger}$ where $\cdot^{\dagger}$ denotes the Rosati involution associated to $\lambda_A$, and such that the Kottwitz determinant condition is satisfied:
$$\forall b \in B,\, \det(\iota_A(b)) = \det(b\,|\, \mathbb V_1).$$
\item $\overline{\eta}_A$ is a $K$-level structure, that is a $K$-orbit of isomorphisms of $B\otimes \mathbb A_f$-modules $\mathrm{H}_1(A,\mathbb A_f) \xrightarrow{\sim} \mathbb V\otimes \mathbb A_f$ that is compatible with the other data.
\end{enumerate}

The Kottwitz condition in the third point is independent on the choice of $h\in X$. If $K$ is sufficiently small, this moduli problem is represented by a smooth quasi-projective scheme $\mathrm{Sh}_K$ over $F$. When the level $K$ varies, the Shimura varieties form a projective system $(\mathrm{Sh}_K)_K$ equipped with an action of $\mathbb G(\mathbb A_f)$ by Hecke correspondences. 

We assume the existence of a $\mathbb Z_{(p)}$-order $\mathcal O_B$ in $B$, stable under the involution $*$, such that its $p$-adic completion is a maximal order in $B_{\mathbb Q_p}$. We also assume that there is a $\mathbb Z_p$-lattice $\Gamma$ in $\mathbb V\otimes \mathbb Q_p$, invariant under $\mathcal O_B$ and self-dual for $\langle\cdot,\cdot\rangle$. We may fix isomorphisms $\mathbb E_p \simeq E$ and $B_{\mathbb Q_p} \simeq \mathrm M_d(E)$ such that $\mathcal O_B\otimes \mathbb Z_{p}$ is identified with $\mathrm M_d(\mathcal O_E)$. \\
As a consequence of the existence of $\Gamma$, the group $G := \mathbb G_{\mathbb Q_p}$ is unramified. Let $K_0 := \mathrm{Fix}(\Gamma)$ be the subgroup of $G(\mathbb Q_p)$ consisting of all $g$ such that $g\cdot \Gamma = \Gamma$. It is a hyperspecial maximal compact subgroup of $G(\mathbb Q_p)$. We will consider levels of the form $K = K_0K^p$ where $K^p$ is an open compact subgroup of $\mathbb G(\mathbb A_f^p)$. Note that $K$ is sufficiently small as soon as $K^p$ is sufficiently small. By the work of Kottwitz in \cite{kottwitzpoints}, the Shimura varieties $\mathrm{Sh}_{K_0K^p}$ admit integral models over $\mathcal O_{F,(p)}$ which have the following moduli interpretation. We define a functor $\mathrm{S}_{K^p}$ by associating to an $\mathcal O_{F,(p)}$-scheme $S$ the set of isomorphism classes of tuples $(A,\lambda_A,\iota_A,\overline{\eta}^p_A)$ where
\begin{enumerate}[label={--}]
\item $A$ is an abelian scheme over $S$.
\item $\lambda_A: A\rightarrow \widehat{A}$ is a polarization whose order is prime to $p$.
\item $\iota_A:\mathcal O_B\rightarrow \mathrm{End}(A)\otimes \mathbb Z_{(p)}$ is a morphism of algebras such that $\iota_A(b^*) = \iota_A(b)^{\dagger}$ where $\cdot^{\dagger}$ denotes the Rosati involution associated to $\lambda_A$, and such that the Kottwitz determinant condition is satisfied:
$$\forall b \in \mathcal O_B,\, \det(\iota_A(b)) = \det(b\,|\, \mathbb V_1).$$
\item $\overline{\eta}^p_A$ is a $K^p$-level structure, that is a $K^p$-orbit of isomorphisms of $B\otimes \mathbb A_f^p$-modules $\mathrm{H}_1(A,\mathbb A_f^p) \xrightarrow{\sim} \mathbb V\otimes \mathbb A_f^p$ that is compatible with the other data.
\end{enumerate}

If $K^p$ is sufficiently small, this moduli problem is also representable by a smooth quasi-projective scheme over $\mathcal O_{F,(p)}$. When the level $K^p$ varies, these integral Shimura varieties form a projective system $(\mathrm{S}_{K^p})_{K^p}$ equipped with an action of $\mathbb G(\mathbb A_f^p)$ by Hecke correspondences. We have a family of isomorphisms 
$$\mathrm{Sh}_{K_0K^p} \simeq \mathrm{S}_{K^p}\otimes_{\mathcal O_{F,(p)}} F$$
which are compatible as the level $K^p$ varies. 

\begin{notation}
From now on, we identify $F_p$ with $\mathbb Q_{p^2}$ and $\mathcal O_{F_p}$ with $\mathbb Z_{p^2}$. Moreover, the notation $\mathrm S_{K^p}$ will refer to the base change $\mathrm S_{K^p} \otimes_{\mathcal O_{F,(p)}} \mathbb Z_{p^2}$.
\end{notation}

Therefore, under this convention we have isomorphisms $\mathrm{Sh}_{K_0K^p}\otimes_F \mathbb Q_{p^2} \simeq \mathrm{S_{K^p}}\otimes_{\mathbb Z_{p^2}} \mathbb Q_{p^2}$ compatible as the level $K^p$ varies. Let $\overline{\mathrm S}_{K^p} := \mathrm S_{K^p}\otimes_{\mathbb Z_{p^2}} \mathbb F_{p^2}$ denote the special fiber of the Shimura variety. Let $\overline{\mathrm S}_{K^p}^{\mathrm{ss}}$ denote the \textbf{supersingular locus} of the Shimura variety, ie. the locus of points $x \in \overline{\mathrm S}_{K^p}$ such that the universal abelian scheme is supersingular at $x$. Then $\overline{\mathrm S}_{K^p}^{\mathrm{ss}}$ is a closed subvariety of $\overline{\mathrm S}_{K^p}$, and its geometry can be described using the Rapoport-Zink space $\mathcal M$ in a process called $p$-adic uniformization, see \cite{RZ} and \cite{fargues}.\\
Let $x = [\mathcal A_x,\lambda_x,\iota_x,\overline{\eta}^p_x]$ be a geometric point of $\overline{\mathrm S}_{K^p}^{\mathrm{ss}}$. Since $\mathbb G$ satisfies the Hasse principle, according to \cite{fargues} Proposition 3.1.8 the isogeny class of $(\mathcal A_x, \lambda_x, \iota_x)$ does not depend on the choice of $x$. The $p$-divisible group $\mathcal A_x[p^{\infty}]$ inherits an $\mathcal O_B \otimes \mathbb Z_p \simeq \mathrm M_d(\mathcal O_E)$-action from $\iota_A$. Let $\mathbb X_x := \mathcal O_E^d \otimes_{\mathrm M_d(\mathcal O_E)} \mathcal A_x[p^{\infty}]$ with $\mathcal O_E$-action induced by the diagonal inclusion $\mathcal O_E \hookrightarrow \mathcal O_E^d$. According to \cite{vw2} Section 6.3, $\mathbb X_x$ is a unitary $p$-divisible group of signature $(1,n-1)$ over $\mathbb F$ in the sense of Section 1. Let also $\mathcal M_x$ be the Rapoport-Zink space defined as in Section 1, but using $\mathbb X_x$ as a framing object. In particular $\mathcal M_x$ is a formal scheme over $\mathrm{Spf}(W(\mathbb F))$. There exists an isogeny $\mathbb X \otimes \mathbb F \to \mathbb X_x$ of unitary $p$-divisible group, inducing an isomorphism $\mathcal M_{W(\mathbb F)} := \mathcal M \otimes_{\mathbb Z_{p^2}} W(\mathbb F) \xrightarrow{\sim} \mathcal M_x$, see \cite{vw2} Section 6.4. The Rapoport-Zink space $\mathcal M_x$ is equipped with an action of the group $J_x(\mathbb Q_p)$ where $J_x$ is the group of quasi-isogenies of the unitary $p$-divisible group $\mathbb X_x$. The quasi-isogeny $\mathbb X \otimes \mathbb F \to \mathbb X_x$ identifies $J_x$ with $J$ and makes the isomorphism between the Rapoport-Zink spaces $J(\mathbb Q_p)$-equivariant. We define $I := \mathrm{Aut}(\mathcal A_x,\lambda_x,\iota_x)$ as a reductive group over $\mathbb Q$. Since $x$ is in the supersingular locus, the group $I$ is the inner form of $\mathbb G$ such that $I_{\mathbb Q_p} = J$ (in fact $J_x$, which is identified with $J$), $I_{\mathbb A_f} = \mathbb G_{\mathbb A_f^p}$ and $I(\mathbb R) \simeq \mathrm{GU}(0,n)$, which is the unique inner form of $G(\mathbb R)$ that is compact modulo center. In particular, one can think of $I(\mathbb Q)$ as a subgroup both of $J(\mathbb Q_p)$ and of $G(\mathbb A_f^p)$. Let $(\widehat{\mathrm S}_{K^p})^{\mathrm{ss}}$ denote the formal completion of $\mathrm S_{K^p}$ along the supersingular locus. The $p$-adic uniformization theorem relates $(\widehat{\mathrm S}_{K^p})^{\mathrm{ss}}$ with a certain quotient of $\mathcal M_x$, see \cite{RZ} Theorem 6.23. Using the isomorphism above, we may replace $\mathcal M_x$ with $\mathcal M_{W(\mathbb F)}$ and obtain the following statement.

\begin{theo}\label{Uniformization}
There is an isomorphism of formal schemes over $\mathrm{Spf}(W(\mathbb F))$
$$\Theta_{K^p}: I(\mathbb Q)\backslash \left( \mathcal M_{W(\mathbb F)} \times \mathbb G(\mathbb A_f^p) / K^p \right) \xrightarrow{\sim} (\widehat{\mathrm S}_{K^p})^{\mathrm{ss}}\otimes_{\mathbb Z_{p^2}} W(\mathbb F)$$
which is compatible with the $\mathbb G(\mathbb A_f^p)$-action by Hecke correspondences as the level $K^p$ varies.
\end{theo}

This isomorphism is known as the \textbf{$p$-adic uniformization} of the supersingular locus. The induced map on the special fiber is an isomorphism 
$$(\Theta_{K^p})_s : I(\mathbb Q)\backslash \left( \mathcal M_{\mathrm{red}}\otimes_{\mathbb F_{p^2}} \mathbb F \times G(\mathbb A_f^p) / K^p \right) \xrightarrow{\sim} \overline{\mathrm S}_{K^p}^{\mathrm{ss}}\otimes_{\mathbb F_{p^2}} \mathbb F$$
of schemes over $\mathbb F$. The double coset space $I(\mathbb Q)\backslash \mathbb G(\mathbb A_f^p) / K^p$ is finite, so that we may fix a system of representatives $g_1,\ldots ,g_s \in \mathbb G(\mathbb A_f^p)$. For every $1 \leq k \leq s$, we define $\Gamma_k := I(\mathbb Q) \cap g_k K^p g_k^{-1}$, which we see as a discrete subgroup of $J(\mathbb Q_p)$ that is cocompact modulo the center. The left hand side of the $p$-adic uniformization theorem is isomorphic to the disjoint union of the quotients $\Gamma_k \backslash \mathcal M_{W(\mathbb F)}$. In particular for the special fiber, it is an isomorphism
$$(\Theta_{K^p})_s:\bigsqcup_{k=1}^s \Gamma_k \backslash (\mathcal M_{\mathrm{red}} \otimes \mathbb F) \xrightarrow{\sim} \overline{\mathrm S}_{K^p}^{\mathrm{ss}} \otimes \mathbb F.$$
Let $\Phi_{K^p}^k$ be the composition $\mathcal M_{\mathrm{red}} \otimes \mathbb F \rightarrow \Gamma_k \backslash (\mathcal M_{\mathrm{red}} \otimes \mathbb F) \rightarrow \overline{\mathrm{Sh}}_{C^p}^{\,\mathrm{ss}} \otimes \mathbb F$ and let $\Phi_{K^p}$ be the disjoint union of the $\Phi_{K^p}^k$. The map $\Phi_{K^p}$ is surjective. According to \cite{vw2} Section 6.4, it is a local isomorphism which can be used to transport the Bruhat-Tits stratification from $\mathcal M_{\mathrm{red}}$ to $\overline{\mathrm S}_{K^p}^{\mathrm{ss}}$.

\begin{prop}\label{ShimuraBT}
Let $\Lambda \in \mathcal L$. For any $1\leq k \leq s$, the restriction of $\Phi_{K^p}^k$ to $\mathcal M_{\Lambda} \otimes \mathbb F$ is an isomorphism onto its image.
\end{prop}

We will denote by $\overline{\mathrm{S}}_{K^p,\Lambda,k}$ the scheme theoretic image of $\mathcal M_{\Lambda} \otimes \mathbb F$ through $\Phi^k$. A subscheme of the form $\overline{\mathrm{S}}_{K^p,\Lambda,k}$ is called a \textbf{closed Bruhat-Tits stratum} of the Shimura variety. Together, they form the Bruhat-Tits stratification of the supersingular locus, whose combinatorics is described by the union of the complexes $\Gamma_k \backslash \mathcal L$.

\section{The cohomology of the Rapoport-Zink space at maximal level}

\subsection{The spectral sequence associated to an open cover of $\mathcal M^{\mathrm{an}}$}

The formal scheme $\mathcal M$ is special in the sense of \cite{berk2} since it is formally locally of finite type. Thus, we may consider the associated analytic space $\mathcal M^{\mathrm{an}}$ over $\mathbb Q_{p^2}$ in the sense of loc. cit. We note that $\mathcal M^{\mathrm{an}}$ is smooth, as follows from \cite{RZ} Proposition 5.17 (to be precise, this statement is about the rigid space $\mathcal M^{\mathrm{rig}}$ in the sense of Berthelot, but it is equivalent to the corresponding statement for $\mathcal M^{\mathrm{an}}$, see for instance \cite{fargues} Lemme 2.3.24, or Appendice D for a brief summary of various comparisons between analytic, rigid and adic spaces). We refer to $\mathcal M^{\mathrm{an}}$ as the generic fiber of $\mathcal M$. It is equipped with a reduction (or specialization) map $\mathrm{red}: \mathcal M^{\mathrm{an}} \to \mathcal M_{\mathrm{red}}$ which is anticontinuous, ie. the preimage of a closed (resp. open) subset is open (resp. closed). If $Z$ is a locally closed subset of $\mathcal M_{\mathrm{red}}$, then the preimage $\mathrm{red}^{-1}(Z)$ is called the \textbf{analytical tube over $Z$}. It is an analytic domain in $\mathcal M^{\mathrm{an}}$ and it coincides with the generic fiber of the formal completion of $\mathcal M_{\mathrm{red}}$ along $Z$. If $i\in \mathbb Z$ such that $ni$ is even, then the tube $\mathrm{red}^{-1}(\mathcal M_i) = \mathcal M_i^{\mathrm{an}}$ is open and closed in $\mathcal M^{\mathrm{an}}$ and we have $\mathcal M^{\mathrm{an}} = \bigsqcup_{ni\in 2\mathbb Z} \mathcal M_i^{\mathrm{an}}.$ If $\Lambda \in \mathcal L$, we define
$$U_{\Lambda} := \mathrm{red}^{-1}(\mathcal M_{\Lambda}),$$ 
the tube over $\mathcal M_{\Lambda}$. The action of $J(\mathbb Q_p)$ on $\mathcal M$ induces an action on the generic fiber $\mathcal M^{\mathrm{an}}$ such that $\mathrm{red}$ is $J(\mathbb Q_p)$-equivariant. By restriction it induces an action of $J_{\Lambda}$ on $U_{\Lambda}$. The analytic space $\mathcal M^{\mathrm{an}}$ and each of the open subspaces $U_{\Lambda}$ have dimension $n-1$.

We fix a prime number $\ell \not = p$. In \cite{berk}, Berkovich developped a theory of étale cohomology for his analytic spaces. Using it we may define the cohomology of the Rapoport-Zink space $\mathcal{M}^{\mathrm{an}}$ by the formula
\begin{align*}
\mathrm H^{\bullet}_c(\mathcal M^{\mathrm{an}}\widehat{\otimes}\, \mathbb C_p,\overline{\mathbb Q_{\ell}}) & := \varinjlim_U \mathrm H_c^{\bullet}(U\widehat{\otimes} \, \mathbb C_p,\overline{\mathbb Q_{\ell}}) \\
& = \varinjlim_U\varprojlim_n \mathrm H_c^{\bullet}(U\widehat{\otimes} \, \mathbb C_p,\mathbb Z/\ell^n\mathbb Z)\otimes \overline{\mathbb Q_{\ell}}
\end{align*}
where $U$ goes over all relatively compact open of $\mathcal M^{\mathrm{an}}$. These cohomology groups are equipped with commuting actions of $J(\mathbb Q_p)$ and of $W$, the absolute Weil group of $\mathbb Q_{p^2}$. The $J(\mathbb Q_p)$-action causes no problem of interpretation, but the $W$-action requires some explanations, see \cite{fargues} Section 4.4.1. Let $\mathrm{Frob} = \sigma^{-2}$ be the geometric Frobenius in $W$. The inertia subgroup $I\subset W$ acts on $\mathrm H^{\bullet}_c(\mathcal M^{\mathrm{an}}\widehat{\otimes}\, \mathbb C_p,\overline{\mathbb Q_{\ell}})$ via the coefficients $\mathbb C_p$, whereas $\mathrm{Frob}$ acts via the \textbf{Weil descent datum} defined by Rapoport and Zink in \cite{RZ} 3.48. Let 
$$F_{\mathbb X}: \mathbb X \otimes \mathbb F \to (\mathbb X \otimes \mathbb F)^{(p^2)}$$
denote the Frobenius morphism relative to $\mathbb F_{p^2}$. Let $(\mathcal M \otimes W(\mathbb F))^{(p^2)}$ be the functor defined by 
$$(\mathcal M \otimes W(\mathbb F))^{(p^2)}(S) := \mathcal M (S^{(p^2)}),$$ 
for all $W(\mathbb F)$-scheme $S$ where $p$ is locally nilpotent. The Weil descent datum is the isomorphism $\alpha_{\mathrm{RZ}}: \mathcal M \otimes W(\mathbb F)  \xrightarrow{\sim} (\mathcal M \otimes W(\mathbb F))^{(p^2)}$ given by $(X,\iota,\lambda,\rho) \in \mathcal M(S) \mapsto (X,\iota,\lambda,F_{\mathbb X} \circ \rho)$. We may describe this in terms of rational points and Dieudonné modules. If $k/\mathbb F$ is a perfect field extension, let $\tau := \mathrm{id}\otimes \sigma^2$ on $\mathbf V_k = \mathbf V \otimes_{\mathbb Q_{p^2}} W(k)_{\mathbb Q}$. Since we use covariant Dieudonné theory, the relative Frobenius $F_{\mathbb X}$ corresponds to the Verschiebung $\mathbf V^2$. By construction of $\mathbb X$, we have $\mathbf V^2 = p\tau^{-1}$ in $\mathbf V_k$. Therefore, $\alpha_{\mathrm{RZ}}$ sends a Dieudonné module $M\in \mathcal M(k)$ to $p\tau^{-1}(M)$.

\begin{rk}
We stress that the Weil descent datum $\alpha_{\mathrm{RZ}}$ is not effective, however the Rapoport-Zink space is defined over $\mathbb Z_{p^2}$, and this rational structure is induced by the effective descent datum $p^{-1}\alpha_{\mathrm{RZ}}$, with $p = p\cdot\mathrm{id} \in \mathrm Z(J(\mathbb Q_p))$.
\end{rk}

We define 
$$\varphi = (p^{-1}\cdot\mathrm{id},\mathrm{Frob}) \in J(\mathbb Q_p)\times W.$$ 
The action of $\varphi$ on the cohomology of $\mathcal M^{\mathrm{an}}$ coincides with the action of a geometric Frobenius induced by the effective descent datum $p^{-1}\alpha_{\mathrm{RZ}}$. Thus, we refer to $\varphi$ as the \textbf{rational Frobenius element}.

\begin{notation}
To alleviate the notations, we will omit the coefficients $\mathbb C_p$. Thefore we write $\mathrm H^{\bullet}_c(\mathcal M^{\mathrm{an}},\overline{\mathbb Q_{\ell}})$ and similarly for subspaces of $\mathcal M^{\mathrm{an}}$.
\end{notation}

The cohomology groups $\mathrm H^{\bullet}_c(\mathcal M^{\mathrm{an}},\overline{\mathbb Q_{\ell}})$ are concentrated in degrees $0$ to $2(n-1)$. According to \cite{fargues} Corollaire 4.4.7, these  groups are smooth for the $J(\mathbb Q_p)$-action and continous for the $I$-action. For $g\in J(\mathbb Q_p)$, we have an isomorphism
$$g : \mathrm H^{\bullet}_c(\mathcal M_i^{\mathrm{an}},\overline{\mathbb Q_{\ell}}) \xrightarrow{\sim} \mathrm H^{\bullet}_c(\mathcal M_{i+\alpha(g)}^{\mathrm{an}},\overline{\mathbb Q_{\ell}}),$$
which is induced by $g^{-1}$ and contravariance of cohomology. In particular, the action of $\mathrm{Frob}$ gives an isomorphism $\mathrm H^{\bullet}(\mathcal M_i,\overline{\mathbb Q_{\ell}}) \xrightarrow{\sim} \mathrm H^{\bullet}(\mathcal M_{i+2},\overline{\mathbb Q_{\ell}})$. Let $(J(\mathbb Q_p)\times W)^{\circ}$ be the subgroup of $J(\mathbb Q_p)\times W$ consisting of all elements of the form $(g,u\mathrm{Frob}^j)$ with $u\in I$ and $\alpha(g) = -2j$. In fact, we have $(J(\mathbb Q_p)\times W)^{\circ} = (J^{\circ}\times I)\varphi^{\mathbb Z}$ where $J^{\circ} := \mathrm{Ker}(\alpha) \subset J(\mathbb Q_p)$, and $\alpha = v_p \circ c$ was introduced in Section 1.1. Each group $\mathrm H^{\bullet}_c(\mathcal M_i^{\mathrm{an}},\overline{\mathbb Q_{\ell}})$ is a $(J(\mathbb Q_p)\times W)^{\circ}$-representation, and we have an isomorphism
$$\mathrm H^{\bullet}_c(\mathcal M^{\mathrm{an}},\overline{\mathbb Q_{\ell}}) \simeq \mathrm{c-Ind}_{(J(\mathbb Q_p)\times W)^{\circ}}^{J(\mathbb Q_p)\times W} \, \mathrm H^{\bullet}_c(\mathcal M_0^{\mathrm{an}},\overline{\mathbb Q_{\ell}}).$$
In particular, when $\mathrm H^{k}_c(\mathcal M^{\mathrm{an}},\overline{\mathbb Q_{\ell}})$ is non-zero it is infinite dimensional. However, by \cite{fargues} Proposition 4.4.13, these cohomology groups are always of finite type as $J(\mathbb Q_p)$-modules. 

We introduce the \v{C}ech spectral sequence associated to the locally finite covering of $\mathcal M^{\mathrm{an}}$ by the $U_{\Lambda}$'s. For $i\in \mathbb Z$ such that $ni$ is even and for $0\leq \theta \leq \thetamax$, we denote by $\mathcal L_i^{(\theta)}$ the subset of $\mathcal L_i$ whose elements are those lattices of orbit type $2\theta+1$. We also write $\mathcal L^{(\theta)}$ for the union of the $\mathcal L_i^{(\theta)}$. Then $\{U_{\Lambda}\}_{\Lambda \in \mathcal L^{(\thetamax)}}$ is an open cover of $\mathcal M^{\mathrm{an}}$. We may apply \cite{fargues} Proposition 4.2.2 to deduce the existence of the following \v{C}ech spectral sequence computing the cohomology of the Rapoport-Zink space, concentrated in degrees $a\leq 0$ and $0\leq b \leq 2(n-1)$,
\begin{equation}\label{SpectralSequence}
E_{1}^{a,b}: \bigoplus_{\gamma \in I_{-a+1}} \mathrm H^b_c(U(\gamma),\overline{\mathbb Q_{\ell}}) \implies \mathrm H^{a+b}_c(\mathcal M^{\mathrm{an}},\overline{\mathbb Q_{\ell}}).
\tag{$E$}
\end{equation}
Here, for $s\geq 1$ the set $I_s$ is defined by 
$$I_s := \left\{\gamma = (\Lambda^1,\ldots ,\Lambda^s) \, \middle | \, \forall 1\leq j \leq s, \Lambda^j \in \mathcal L^{(\thetamax)} \text{ and } U(\gamma):= \bigcap_{j=1}^s U_{\Lambda^j} \not = \emptyset \right\}.$$
Necessarily, if $\gamma = (\Lambda^1,\ldots ,\Lambda^s)\in I_s$ then there exists a unique $i$ such that $ni$ is even and $\Lambda^j \in \mathcal L_i^{(\thetamax)}$ for all $J(\mathbb Q_p)$. We then define $\Lambda(\gamma) := \bigcap_{j=1}^s \Lambda^j \in \mathcal L_i$ so that $U(\gamma) = U_{\Lambda(\gamma)}$. In particular, the open subspace $U(\gamma)$ depends only on the intersection $\Lambda(\gamma)$ of the elements in the $s$-tuple $\gamma$. \\
For $s\geq 2$ and $\gamma = (\Lambda^1,\ldots,\Lambda^{s}) \in I_s$, define $\gamma_j := (\Lambda^1,\ldots ,\widehat{\Lambda^j},\ldots ,\Lambda^s) \in I_{s-1}$ for the $(s-1)$-tuple obtained from $\gamma$ by removing the $j$-th term. Besides, for $\Lambda,\Lambda' \in \mathcal L_i$ with $\Lambda' \subset \Lambda$, we write $f_{\Lambda',\Lambda}^b$ for the natural map $\mathrm H^b_c(U_{\Lambda'},\overline{\mathbb Q_{\ell}}) \to \mathrm H^b_c(U_{\Lambda},\overline{\mathbb Q_{\ell}})$ induced by the open immersion $U_{\Lambda'} \subset U_{\Lambda}$. For $a\leq -1$, the differential $E_1^{a,b}\to E_1^{a+1,b}$ is denoted by $\varphi^{a,b}$. It is the direct sum over all $\gamma \in I_{-a+1}$ of the maps 
\begin{align*}
\mathrm H^b_c(U(\gamma),\overline{\mathbb Q_{\ell}}) & \to \bigoplus_{\delta \in \{\gamma_1,\ldots \gamma_{-a+1}\}} \mathrm H^b_c(U(\delta),\overline{\mathbb Q_{\ell}})\\
v & \mapsto \left(\sum_{\substack{j=1\\ \gamma_j = \delta}}^{-a+1}(-1)^{j+1}f_{\Lambda(\gamma),\Lambda(\gamma_j)}^b(v)\right)_{\delta \in \{\gamma_1,\ldots, \gamma_{-a+1}\}}.
\end{align*}

An element $g\in J(\mathbb Q_p)$ acts on the set $I_s$ by sending $\gamma$ to $g\cdot\gamma := (g\Lambda^1,\ldots,g\Lambda^s)$. The action of $g^{-1}$ induces an isomorphism
$$\mathrm H^{\bullet}_c(U(\gamma),\overline{\mathbb Q_{\ell}}) \xrightarrow{\sim} \mathrm H^{\bullet}_c(U(g\cdot \gamma),\overline{\mathbb Q_{\ell}}).$$
Likewise, $\mathrm{Frob} \in W$ induces an isomorphism $\mathrm H_c^{\bullet}(U(\gamma),\overline{\mathbb Q_{\ell}}) \xrightarrow{\sim} \mathrm H_c^{\bullet}(U(p\cdot\gamma),\overline{\mathbb Q_{\ell}})$. This defines a natural $J(\mathbb Q_p)\times W$-action on the terms $E_1^{a,b}$, with respect to which the spectral sequence is equivariant. 


In order to analyze the spectral sequence \eqref{SpectralSequence}, we begin by relating the cohomology of a tube $U_{\Lambda}$ to the cohomology of the corresponding closed Bruhat-Tits stratum $\mathcal M_{\Lambda}$. Note that by restriction, $\mathrm H_c^{\bullet}(U_{\Lambda},\overline{\mathbb Q_{\ell}})$ is naturally a representation of the subgroup $(J_{\Lambda}\times I)\varphi^{\mathbb Z} \subset J(\mathbb Q_p)\times W$. 

\begin{prop} \label{NearbyCycles}
Let $\Lambda \in \mathcal L$ and let $0\leq b \leq 2(n-1)$. There is a $(J_{\Lambda}\times I)\varphi^{\mathbb Z}$-equivariant isomorphism 
$$\mathrm H^{b}(\mathcal M_{\Lambda}\otimes \mathbb F, \overline{\mathbb Q_{\ell}}) \simeq \mathrm H^{b}(U_{\Lambda},\overline{\mathbb Q_{\ell}})$$
where, on the left-hand side, the inertia $I$ acts trivially and $\varphi$ acts like the geometric Frobenius $\mathrm{Frob}$.
\end{prop}

In particular, the inertia acts trivially on the cohomology of $U_{\Lambda}$. 

\begin{proof}
The closed subvariety $\mathcal M_{\Lambda} \subset \mathcal M_{\mathrm{red}}$ is bounded in the sense of \cite{RZ} Paragraph 2.30. Indeed, it is irreducible and all irreducible components of $\mathcal M_{\mathrm{red}}$ are bounded by the proof of loc. cit. Proposition 2.32. Thus, there exists a quasi-compact open formal subscheme $\mathcal U$ of $\mathcal M$ containing $\mathcal M_{\Lambda}$ (these are denoted by $U^f$ and are introduced in the proof of Theorem 2.16 in loc. cit.). The formal scheme $\mathcal U$ is of finite type, in particular the structure morphism $\mathcal U \to \mathrm{Spf}(\mathbb Z_{p^2})$ is adic. Since $\mathcal M$ is formally smooth, $\mathcal U$ is actually a smooth formal scheme. Replacing $\mathcal U$ by $J_{\Lambda} \cdot \mathcal U$, we may assume that $\mathcal U$ is stable under the action of $J_{\Lambda}$.\\
Let $\mathrm R\Psi_{\eta}\overline{\mathbb Q_{\ell}}$ denote Berkovich's nearby cycles on $\mathcal U_{\mathrm{red}}$ as defined in \cite{berk1}. Since $\mathcal U$ is smooth, by Corollary 5.4 of loc. cit. we actually have $\mathrm R\Psi_{\eta}\overline{\mathbb Q_{\ell}} \simeq \overline{\mathbb Q_{\ell}}$. Besides, let $\mathrm R\widetilde{\lambda}_*\overline{\mathbb Q_{\ell}}$ denote Huber's nearby cycles as defined in \cite{huber} Paragraph 3.12, where $\widetilde{\lambda}: \widetilde d(\mathcal U) \to \mathcal U$ is the natural reduction map attached to the adic space $\widetilde d(\mathcal U)$ associated to the formal scheme $\mathcal U$. Since the etale sites of $\mathcal U$ and of $\mathcal U_{\mathrm{red}}$ are naturally identified, we can think of $\mathrm R\widetilde{\lambda}_*\overline{\mathbb Q_{\ell}}$ as an object of the derived category of $\ell$-adic sheaves on $\mathcal U_{\mathrm{red}}$. According to \cite{fargues} Section 5.4.2, both notions of nearby cycles coincide, ie.
$$\mathrm R\widetilde{\lambda}_*\overline{\mathbb Q_{\ell}} \simeq \mathrm R\Psi_{\eta}\overline{\mathbb Q_{\ell}} \simeq \overline{\mathbb Q_{\ell}}.$$
In particular, the inertia acts trivially on the nearby cycles. Let $\mathcal U_{|\mathcal M_{\Lambda}}^{\wedge}$ denote the formal completion of $\mathcal U$ along $\mathcal M_{\Lambda}$. Since $\mathcal U$ is open in $\mathcal M$, it coincides with the formal completion of $\mathcal M$ along $\mathcal M_{\Lambda}$. Thus, we have $(\mathcal U_{|\mathcal M_{\Lambda}}^{\wedge})^{\mathrm{an}} = U_{\Lambda}$. Moreover, $\widetilde d(\mathcal U_{|\mathcal M_{\Lambda}}^{\wedge}) = U_{\Lambda}^{\mathrm{rig}}$ according to \cite{fargues} Appendice D, where $(\,\cdot\,)^{\mathrm{rig}}$ is the natural functor from the category of Hausdorff analytic spaces to the category of quasiseparated adic spaces. Therefore, by \cite{huberbook} Theorem 8.3.5.iii) we have an isomorphism $\mathrm H^b(U_{\Lambda},\overline{\mathbb Q_{\ell}}) \simeq \mathrm H^b(\widetilde d(\mathcal U_{|\mathcal M_{\Lambda}}^{\wedge}) \otimes \mathbb C_p,\overline{\mathbb Q_{\ell}})$. Moreover, by \cite{huber} Proposition 3.15 applied to the smooth formal scheme $\mathcal U$, we have 
$$\mathrm H^b(\widetilde d(\mathcal U_{|\mathcal M_{\Lambda}}^{\wedge}) \otimes \mathbb C_p,\overline{\mathbb Q_{\ell}}) \simeq \mathrm H^b(\mathcal M_{\Lambda}\otimes \mathbb F, (\mathrm R\widetilde{\lambda}_*\overline{\mathbb Q_{\ell}})_{|\mathcal M_{\Lambda}}) = \mathrm H^b(\mathcal M_{\Lambda}\otimes \mathbb F, \overline{\mathbb Q_{\ell}}).$$
The isomorphisms are compatible with the actions of $J_{\Lambda}$ and of the Frobenius.
\end{proof}


\begin{corol}
Let $\Lambda \in \mathcal L$ and let $0 \leq b \leq 2(n-1)$. There is a $(J_{\Lambda}\times I)\varphi^{\mathbb Z}$-equivariant isomorphism
$$\mathrm H^{b}_c(U_{\Lambda},\overline{\mathbb Q_{\ell}}) \xrightarrow{\sim} \mathrm H^{b - 2(n-1-\theta)}(\mathcal M_{\Lambda}\otimes\mathbb F,\overline{\mathbb Q_{\ell}})(n-1-\theta)$$
where $t(\Lambda) = 2\theta + 1$.
\end{corol}

\begin{proof}
This is a consequence of algebraic and analytic Poincaré duality, respectively for $U_{\Lambda}$ and for $\mathcal M_{\Lambda}$. Indeed, we have
\begin{align*}
\mathrm H^{b}_c(U_{\Lambda},\overline{\mathbb Q_{\ell}}) & \simeq \mathrm H^{2(n-1) - b}(U_{\Lambda},\overline{\mathbb Q_{\ell}})^{\vee}(n-1)\\
& \simeq \mathrm H^{2(n-1) - b}(\mathcal M_{\Lambda}\otimes\mathbb F,\overline{\mathbb Q_{\ell}})^{\vee}(n-1) \\
& \simeq  \mathrm H^{b - 2(n-1-\theta)}(\mathcal M_{\Lambda}\otimes\mathbb F,\overline{\mathbb Q_{\ell}})(n-1-\theta).
\end{align*}
\end{proof}

Let $\Lambda \in \mathcal L$ and write $t(\Lambda) = 2\theta + 1$. If $\lambda$ is a partition of $2\theta + 1$, recall the unipotent irreducible representation $\rho_{\lambda}$ of $\mathrm{GU}(V_{\Lambda}^{0}) \simeq \mathrm{GU}_{2\theta+1}(\mathbb F_p)$ that we introduced in Section 2. It can be inflated to the maximal reductive quotient $\mathcal J_{\Lambda} \simeq \mathrm{G}(\mathrm{U}(V^0_{\Lambda})\times \mathrm{U}(V^1_{\Lambda}))$, and then to the maximal parahoric subgroup $J_{\Lambda}$. With an abuse of notation, we still denote this inflated representation by $\rho_{\lambda}$. In virtue of Theorem \ref{CohomologyBT-Stratum}, the isomorphism in the last paragraph translates into the following result.

\begin{prop}\label{CohomologyOpenBT}
Let $\Lambda \in \mathcal L$ and write $t(\Lambda) = 2\theta + 1$. The following statements hold. 
\begin{enumerate}[label=\upshape (\arabic*)]
		\item The cohomology group $\mathrm H_c^b(U_{\Lambda},\overline{\mathbb Q_{\ell}})$ is zero unless $2(n-1-\theta) \leq b \leq 2(n-1)$.
		\item The action of $J_{\Lambda}$ on the cohomology factors through an action of the finite group of Lie type $\mathrm{GU}(V^{0}_{\Lambda})$. The rational Frobenius $\varphi$ acts like multiplication by $(-p)^{b}$ on $\mathrm H_c^b(U_{\Lambda},\overline{\mathbb Q_{\ell}})$. 
		\item For $0\leq b \leq \theta$ we have 
		$$\mathrm H_c^{2b + 2(n-1-\theta)}(U_{\Lambda},\overline{\mathbb Q_{\ell}}) = \bigoplus_{s=0}^{\min(b,\theta - b)} \rho_{(2\theta + 1 - 2s, 2s)}.$$
		For $0\leq b \leq \theta - 1$ we have 
		$$\mathrm H_c^{2b+1 + 2(n-1-\theta)}(U_{\Lambda},\overline{\mathbb Q_{\ell}}) = \bigoplus_{s=0}^{\min(b,\theta - 1 - b)} \rho_{(2\delta - 2s, 2s + 1)}.$$
	\end{enumerate}
\end{prop}

The description of the rational Frobenius action yields the following corollary.

\begin{corol}\label{SecondPage}
The spectral sequence degenerates on the second page $E_2$. For $0 \leq b \leq 2(n-1)$, the induced filtration on $\mathrm H_c^b(\mathcal M^{\mathrm{an}},\overline{\mathbb Q_{\ell}})$ splits, ie. we have an isomorphism 
$$\mathrm H_c^b(\mathcal M^{\mathrm{an}},\overline{\mathbb Q_{\ell}}) \simeq \bigoplus_{b \leq b' \leq 2(n-1)} E_2^{b-b',b'}.$$
The action of $W$ on $\mathrm H_c^b(\mathcal M^{\mathrm{an}},\overline{\mathbb Q_{\ell}})$ is trivial on the inertia subgroup and the action of the rational Frobenius element $\varphi$ is semisimple. The subspace $E_2^{b-b',b'}$ is identified with the eigenspace of $\varphi$ associated to the eigenvalue $(-p)^{b'}$.
\end{corol}

We stress that in the previous statement, the terms $E_2^{b-b',b'}$ may be zero.

\begin{proof}
The $(a,b)$-term in the first page of the spectral sequence is the direct sum of the cohomology groups $H^{b}_c(U(\gamma),\overline{\mathbb Q_{\ell}})$ for all $\gamma \in I_{-a+1}$. On each of these cohomology groups, the rational Frobenius $\varphi$ acts via multiplication by $(-p)^b$. This action is in particular independant of $\gamma$ and of $a$. Thus, on the $b$-th row of the first page of the sequence, the Frobenius acts everywhere as multiplication by $(-p)^b$. Starting from the second page, the differentials in the sequence connect two terms lying in different rows. Since the differentials are equivariant for the $\varphi$-action, they must all be zero. Thus, the sequence degenerates on the second page. By the machinery of spectral sequences, there is a filtration on $\mathrm H_c^b(\mathcal M^{\mathrm{an}},\overline{\mathbb Q_{\ell}})$ whose graded factors are given by the terms $E_2^{b-b',b'}$ of the second page. Only a finite number of these terms are non-zero, and since they all lie on different rows, the Frobenius $\varphi$ acts via multiplication by a different scalar on each graded factor of the filtration. It follows that the filtration splits, ie. the abutment is the direct sum of the graded pieces of the filtration, as they correspond to the eigenspaces of $\varphi$. Consequently, its action is semisimple.
\end{proof}

The spectral sequence $E_1^{a,b}$ has non-zero terms extending indefinitely in the range $a\leq 0$. For instance, if $\Lambda \in \mathcal L^{(\thetamax)}$ then $(\Lambda,\ldots ,\Lambda) \in I_{-a+1}$ so that $E_1^{a,b} \not = 0$ for all $a\leq 0$ and $2(n-1-\thetamax)\leq b \leq 2(n-1)$. To rectify this, we introduce the alternating \v{C}ech spectral sequence. If $v \in E_1^{a,b}$ and $\gamma \in I_{-a+1}$, we denote by $v_{\gamma} \in \mathrm H_c^{b}(U(\gamma),\overline{\mathbb Q_{\ell}})$ the component of $v$ in the summand of $E_1^{a,b}$ indexed by $\gamma$. Besides, if $\gamma = (\Lambda^1,\ldots ,\Lambda^{-a+1}) \in I_{-a+1}$ and if $\sigma \in \mathfrak S_{-a+1}$ then we write $\sigma(\gamma) := (\Lambda^{\sigma(1)},\ldots ,\Lambda^{\sigma(-a+1)}) \in I_{-a+1}$. For all $a,b$ we define 
$$E_{1,\mathrm{alt}}^{a,b} := \{v\in E_1^{a,b} \,|\, \forall \gamma \in I_{-a+1}, \forall \sigma \in \mathfrak S_{-a+1}, v_{\sigma(\gamma)} = \mathrm{sgn}(\sigma)v_{\gamma}\}.$$ 
In particular, if $\gamma = (\Lambda^1,\ldots ,\Lambda^{-a+1})$ with $\Lambda^{j} = \Lambda^{j'}$ for some $j\not = j'$ then $v \in E_{1,\mathrm{alt}}^{a,b} \implies v_{\gamma} = 0$. The subspace $E_{1,\mathrm{alt}}^{a,b} \subset E_1^{a,b}$ is stable under the action of $J(\mathbb Q_p)\times W$, and the differential $\varphi^{a,b}:E_1^{a,b} \to E_1^{a+1,b}$ sends $E_{1,\mathrm{alt}}^{a,b}$ to $E_{1,\mathrm{alt}}^{a+1,b}$. Thus, for all $b$ we have a chain complex $E_{1,\mathrm{alt}}^{\bullet,b}$ and the following proposition is well-known, see eg. \cite{stacks-project} Lemma 01FM.

\begin{prop}\label{AlternatingCech}
The inclusion map $E_{1,\mathrm{alt}}^{\bullet,b} \hookrightarrow E_{1}^{\bullet,b}$ is a homotopy equivalence. In particular we have canonical isomorphisms $E_{2,\mathrm{alt}}^{a,b} \simeq E_2^{a,b}$ for all $a,b$.
\end{prop}

The advantage of the alternating \v{C}ech spectral sequence is that it is concentrated in a finite strip. Indeed, if $\gamma = (\Lambda^1,\ldots ,\Lambda^{-a+1}) \in I_{-a+1}$, let $i \in \mathbb Z$ such that $\Lambda(\gamma) \in \mathcal L_i$. Then all the $\Lambda^j$'s belong to the set of lattices in $\mathcal L^{(\thetamax)}_i$ containing $\Lambda(\gamma)$. This set is finite of cardinality $\nu(n-\theta-\thetamax-1,n-2\theta-1)$ where $t(\Lambda(\gamma)) = 2\theta+1$ according to Proposition \ref{NumberStrata}. Thus, if $-a+1$ is big enough then all the $\gamma$'s in $I_{-a+1}$ will have some repetition, so that $E_{1,\mathrm{alt}}^{a,b} = 0$.

\begin{rk}
The Lemma 01FM of \cite{stacks-project} is stated in the context of \v{C}ech cohomology of an abelian presheaf $\mathcal F$ on a topological space $X$. However, the proof may be adapted to \v{C}ech homology of precosheaves such as $U \mapsto \mathrm H_c^b(U,\overline{\mathbb Q_{\ell}})$. 
\end{rk}

For $a = 0$, we have $E_{1,\mathrm{alt}}^{0,b} = E_1^{0,b}$ by definition. Let us consider the cases $b = 2(n-1-\thetamax)$ and $b = 2(n-1-\thetamax)+1$. For such $b$, it follows from \ref{CohomologyOpenBT} that $\mathrm H_c^b(U_{\Lambda},\overline{\mathbb Q_{\ell}}) = 0$ if $t(\Lambda) < t_{\mathrm{max}}$. If $a \leq -1$, we have $-a+1 \geq 2$ so that for all $\gamma = (\Lambda^1,\ldots ,\Lambda^{-a+1}) \in I_{-a+1}$, if there exists $j \not = j'$ such that $\Lambda^{j} \not = \Lambda^{j'}$, then $t(\Lambda(\gamma)) < t_{\mathrm{max}}$ so that $\mathrm H_c^b(U(\gamma),\overline{\mathbb Q_{\ell}})=0$. It follows that $E_{1,\mathrm{alt}}^{a,b} = 0$ for all $a\leq -1$ and $b$ as above. This observation, along with the previous paragraph, yields the following proposition. 

\begin{prop}\label{MiddleCohomologyFirst}
We have $E_2^{0,2(n-1-\thetamax)} \simeq E_1^{0,2(n-1-\thetamax)}$. If moreover $\thetamax\geq 1$ (ie. $n\geq 3$), then we have $E_2^{0,2(n-1-\thetamax)+1} \simeq E_1^{0,2(n-1-\thetamax)+1}$ as well.
\end{prop} 

In order to study the action of $J(\mathbb Q_p)$, we may rewrite $E_1^{a,b}$ conveniently in terms of compactly induced representations. To do this, let us introduce a few more notations. For $0\leq \theta \leq \thetamax$ and $s\geq 1$, we define
$$I_s^{(\theta)} := \{\gamma \in I_s \,|\, t(\Lambda(\gamma)) = 2\theta +1\}.$$
The subset $I_s^{(\theta)}\subset I_s$ is stable under the action of $J(\mathbb Q_p)$. We denote by $N(\Lambda_{\theta})$ the set of lattices $\Lambda\in \mathcal L_0$ of maximal orbit type containing $\Lambda_{\theta}$. For $s\geq 1$ we define
$$K_s^{(\theta)} := \{\delta = (\Lambda^1,\ldots ,\Lambda^s) \,|\, \forall 1\leq j \leq s, \Lambda^j \in N(\Lambda_{\theta}) \text{ and } \Lambda(\delta) = \Lambda_{\theta}\}.$$
Then $K_s^{(\theta)}$ is a finite subset of $I_s^{(\theta)}$ and it is stable under the action of $J_{\theta}$. If $\gamma \in I_s^{(\theta)}$, there exists some $g\in J(\mathbb Q_p)$ such that $g \cdot \Lambda(\gamma) = \Lambda_{\theta}$ since both lattices share the same orbit type. Moreover, the coset $J_{\theta}\cdot g$ is uniquely determined, and $g\cdot\gamma$ is an element of $K_s^{(\theta)}$. This mapping results in a natural bijection between the orbit sets 
$$J \backslash I_s^{(\theta)} \xrightarrow{\sim} J_{\theta} \backslash K_s^{(\theta)}.$$
The bijection sends the orbit $J\cdot \alpha$ to the orbit $J_{\theta} \cdot (g\cdot \alpha)$ where $g$ is chosen as above. The inverse sends an orbit $J_{\theta}\cdot \beta$ to $J\cdot \beta$. We note that both orbit sets are finite. We may now rearrange the terms in the spectral sequence.

\begin{prop}\label{RearrangementNotations}
We have an isomorphism 
\begin{align*}
E_1^{a,b} & \simeq \bigoplus_{\theta = 0}^{\thetamax} \bigoplus_{[\delta]\in J_{\theta}\backslash K_{-a+1}^{(\theta)}} \mathrm{c-Ind}_{\mathrm{Fix}(\delta)}^J\,\mathrm H_c^b(U_{\Lambda_{\theta}},\overline{\mathbb Q_{\ell}})_{|\mathrm{Fix}(\delta)}\\
& \simeq \bigoplus_{\theta = 0}^{\thetamax} \mathrm{c-Ind}_{J_{\theta}}^J \, \left( \mathrm H_c^b(U_{\Lambda_{\theta}},\overline{\mathbb Q_{\ell}}) \otimes \overline{\mathbb Q_{\ell}}[K_{-a+1}^{(\theta)}]\right),
\end{align*}
where $\overline{\mathbb Q_{\ell}}[K_{-a+1}^{(\theta)}]$ is the permutation representation associated to the action of $J_{\theta}$ on the finite set $K_{-a+1}^{(\theta)}$.
\end{prop}

\begin{rk}
For $\delta \in K_s^{(\theta)}$, the group $\mathrm{Fix}(\delta)$ consists of the elements $g\in J(\mathbb Q_p)$ such that $g\cdot\delta = \delta$. Any such $g$ satisfies $g\Lambda(\delta) = \Lambda(\delta)$, and since $\Lambda(\delta) = \Lambda_{\theta}$ we have $\mathrm{Fix}(\delta) \subset J_{\theta}$. If $\delta = (\Lambda^1,\ldots ,\Lambda^s)$ then $\mathrm{Fix}(\delta)$ is the intersection of the maximal parahoric subgroups $J_{\Lambda^1},\ldots ,J_{\Lambda^s}$. We note that in general, $\mathrm{Fix}(\delta)$ is itself not a parahoric subgroup of $J(\mathbb Q_p)$ since the lattices $\Lambda^1,\ldots,\Lambda^s$ need not form a simplex in $\mathcal L$, as they all share the same orbit type. If however $\Lambda^1 = \ldots = \Lambda^s$ then $\mathrm{Fix}(\delta) = J_{\Lambda^1}$ is a conjugate of the maximal parahoric subgroup $J_{\thetamax}$.
\end{rk}

\begin{proof}
First, by decomposing $I_{-a+1}$ as the disjoint union of the $I_{-a+1}^{(\theta)}$ for $0\leq \theta \leq \thetamax$, we may write 
$$E_1^{a,b} = \bigoplus_{\theta = 0}^{\thetamax} \, \bigoplus_{\gamma\in I_{-a+1}^{(\theta)}} \mathrm H^b_c(U(\gamma),\overline{\mathbb Q_{\ell}}).$$
For each orbit $X\in J \backslash I_{-a+1}^{(\theta)}$, we fix a representative $\delta_X$ which lies in $K_{-a+1}^{(\theta)}$. We may write 
$$E_1^{a,b} = \bigoplus_{\theta = 0}^{\thetamax} \, \bigoplus_{X\in J\backslash I_{-a+1}^{(\theta)}} \, \bigoplus_{\gamma\in X} \mathrm H^b_c(U(\gamma),\overline{\mathbb Q_{\ell}}) = \bigoplus_{\theta = 0}^{\thetamax} \, \bigoplus_{X\in J\backslash I_{-a+1}^{(\theta)}} \, \bigoplus_{g\in J/\mathrm{Fix}(\delta_X)} g\cdot \mathrm H^b_c(U(\delta_X),\overline{\mathbb Q_{\ell}}).$$
The rightmost sum can be identified with a compact induction from $\mathrm{Fix}(\delta_X)$ to $J(\mathbb Q_p)$. Identifying the orbit sets $J \backslash I_{-a+1}^{(\theta)} \xrightarrow{\sim} J_{\theta} \backslash K_{-a+1}^{(\theta)}$, we have
$$E_1^{a,b} \simeq \bigoplus_{\theta = 0}^{\thetamax} \bigoplus_{[\delta]\in J_{\theta}\backslash K_{-a+1}^{(\theta)}} \mathrm{c-Ind}_{\mathrm{Fix}(\delta)}^J\,\mathrm H_c^b(U_{\Lambda_{\theta}},\overline{\mathbb Q_{\ell}})_{|\mathrm{Fix}(\delta)}.$$
By transitivity of compact induction, we have 
$$\mathrm{c-Ind}_{\mathrm{Fix}(\delta)}^J\,\mathrm H_c^b(U_{\Lambda_{\theta}},\overline{\mathbb Q_{\ell}})_{|\mathrm{Fix}(\delta)} = \mathrm{c-Ind}_{J_{\theta}}^J \, \mathrm{c-Ind}_{\mathrm{Fix}(\delta)}^{J_{\theta}} \, \mathrm H_c^b(U_{\Lambda_{\theta}},\overline{\mathbb Q_{\ell}})_{|\mathrm{Fix}(\delta)}.$$
Since $H_c^b(U_{\Lambda_{\theta}},\overline{\mathbb Q_{\ell}})_{|\mathrm{Fix}(\delta)}$ is the restriction of a representation of $J_{\theta}$ to $\mathrm{Fix}(\delta)$, applying compact induction from $\mathrm{Fix}(\delta)$ to $J_{\theta}$ results in tensoring with the permutation representation of $J_{\theta}/\mathrm{Fix}(\delta)$. Thus 
\begin{align*} 
E_1^{a,b} & \simeq \bigoplus_{\theta = 0}^{\thetamax} \bigoplus_{[\delta]\in J_{\theta}\backslash K_{-a+1}^{(\theta)}} \mathrm{c-Ind}_{J_{\theta}}^J \, \left( \mathrm H_c^b(U_{\Lambda_{\theta}},\overline{\mathbb Q_{\ell}}) \otimes \overline{\mathbb Q_{\ell}}[J_{\theta}/\mathrm{Fix}(\delta)]\right)\\
& \simeq \bigoplus_{\theta = 0}^{\thetamax} \mathrm{c-Ind}_{J_{\theta}}^J \, \left( \mathrm H_c^b(U_{\Lambda_{\theta}},\overline{\mathbb Q_{\ell}}) \otimes \bigoplus_{[\delta]\in J_{\theta}\backslash K_{-a+1}^{(\theta)}} \overline{\mathbb Q_{\ell}}[J_{\theta}/\mathrm{Fix}(\delta)]\right),
\end{align*}
where on the second line we used additivity of compact induction. Now, $J_{\theta}/\mathrm{Fix}(\delta)$ is identified with the $J_{\theta}$-orbit $J_{\theta}\cdot\delta$ of $\delta$ in $K_{-a+1}^{(\theta)}$, so that 
$$\bigoplus_{[\delta]\in J_{\theta}\backslash K_{-a+1}^{(\theta)}} \overline{\mathbb Q_{\ell}}[J_{\theta}/\mathrm{Fix}(\delta)] \simeq \overline{\mathbb Q_{\ell}}[ \bigsqcup_{[\delta]\in J_{\theta}\backslash K_{-a+1}^{(\theta)}} J_{\theta}\cdot\delta] \simeq \overline{\mathbb Q_{\ell}}[K_{-a+1}^{(\theta)}],$$
which concludes the proof.
\end{proof}

By Proposition \ref{PointsOfMlambda}, we may identify $N(\Lambda_{\theta})$ with the set $N(n-\theta-\thetamax-1,V_{\theta}^1)$ as defined in Section 1.4. Thus, for $s\geq 1$, $K_s^{(\theta)}$ is naturally identified with
$$\overline{K}_s^{(\theta)} \simeq \left\{\overline{\delta} = (U^1,\ldots,U^s) \,\middle |\, \forall 1\leq j \leq s, U^j \in N(n-\theta-\thetamax-1,V_{\theta}^1) \text{ and } \sum_{j=1}^s U^j = V_{\theta}^1 \right\}.$$
The action of $J_{\theta}$ on $K_s^{(\theta)}$ corresponds to the natural action of $\mathrm{GU}(V_{\theta}^1)$ on $\overline{K}_s^{(\theta)}$, which factors through an action of the finite projective unitary group $\mathrm{PU}(V_{\theta}^1) := \mathrm U(V_{\theta}^1)/\mathrm Z(\mathrm U(V_{\theta}^1)) \simeq \mathrm{GU}(V_{\theta}^1)/\mathrm Z(\mathrm{GU}(V_{\theta}^1))$. Thus, the representation $\overline{\mathbb Q_{\ell}}[K_{-a+1}^{(\theta)}]$ is the inflation to $J_{\theta}$ of the representation $\overline{\mathbb Q_{\ell}}[\overline{K}_{-a+1}^{(\theta)}]$ of the finite projective unitary group $\mathrm{PU}(V_{\theta}^1)$. When $\theta = \thetamax$ or when $s=1$, we trivially have the following proposition.

\begin{prop}
For $s\geq 1$, we have $\overline{\mathbb Q_{\ell}}[K_s^{(\thetamax)}] = \mathbf 1$. For $0\leq \theta \leq \thetamax-1$, we have $\overline{\mathbb Q_{\ell}}[K_1^{(\theta)}] = 0$.
\end{prop}

\begin{proof}
If $\delta = (\Lambda^1,\ldots ,\Lambda^s) \in K_s^{(\thetamax)}$ then $\Lambda(\delta) = \Lambda_{\thetamax}$ has maximal orbit type $t_{\mathrm{max}} = 2\thetamax+1$. For any $1\leq j \leq s$ we have $\Lambda_{\thetamax} \subset \Lambda^j$, therefore $\Lambda^1 = \ldots = \Lambda^s = \Lambda_{\thetamax}$. Thus $K_s^{(\thetamax)}$ is a singleton and so $\overline{\mathbb Q_{\ell}}[K_s^{(\thetamax)}]$ is trivial. Besides, if $\theta < \thetamax$ then $K_1^{(\theta)}$ is clearly empty.
\end{proof}

Recall Proposition \ref{MiddleCohomologyFirst}. We obtain the following corollary.

\begin{corol}\label{MiddleCohomologyGroups}
We have 
$$E_1^{0,b} \simeq \mathrm{c-Ind}_{J_{\thetamax}}^J \, \mathrm H_c^b(U_{\Lambda_{\thetamax}},\overline{\mathbb Q_{\ell}}).$$
In particular, we have 
$$E_2^{0,b} \simeq 
\begin{cases}
\mathrm{c-Ind}_{J_{\thetamax}}^J \, \rho_{(2\thetamax+1)} & \text{if } b = 2(n-1-\thetamax),\\
\mathrm{c-Ind}_{J_{\thetamax}}^J \, \rho_{(2\thetamax,1)} & \text{if } m\geq 1 \text{ and } b = 2(n-1-\thetamax)+1.
\end{cases}$$
\end{corol}

\begin{rk}
The representation $\rho_{(2\thetamax+1)} = \mathbf 1$ is the trivial representation of $J_{\thetamax}$.
\end{rk}

Let us now consider the top row of the spectral sequence, corresponding to $b=2(n-1)$. For $\Lambda' \subset \Lambda$, recall the map $f_{\Lambda',\Lambda}^{2(n-1)} : \mathrm H_c^{2(n-1)}(U_{\Lambda'},\overline{\mathbb Q_{\ell}}) \to \mathrm H_c^{2(n-1)}(U_{\Lambda},\overline{\mathbb Q_{\ell}})$. By Poincaré duality, it is the dual map of the restriction morphism $\mathrm H^{0}(U_{\Lambda},\overline{\mathbb Q_{\ell}}) \to \mathrm H^{0}(U_{\Lambda'},\overline{\mathbb Q_{\ell}})$. Both spaces are one-dimensional by Proposition \ref{NearbyCycles}, and the restriction morphism is the identity. Thus, $E_1^{a,2(n-1)}$ is the $\overline{\mathbb Q_{\ell}}$-vector space generated by $I_{-a+1}$, and the differential $\varphi^{a,2(n-1)}$ is given by
\begin{align*}
\gamma \in I_{-a+1} \mapsto \sum_{j=1}^{-a+1}(-1)^{j+1}\gamma_j.
\end{align*}

Using this description, we may compute the highest cohomology group $\mathrm H_c^{2(n-1)}(\mathcal M^{\mathrm{an}},\overline{\mathbb Q_{\ell}})$ explicitely. 

\begin{prop}\label{HighestDegreeCohomologyGroup}
There is an isomorphism 
$$\mathrm H_c^{2(n-1)}(\mathcal M^{\mathrm{an}},\overline{\mathbb Q_{\ell}}) \simeq \mathrm{c-Ind}_{J^{\circ}}^J \, \mathbf 1,$$
and the rational Frobenius $\varphi$ acts via multiplication by $p^{2(n-1)}$. 
\end{prop}

\begin{proof}
The statement on the Frobenius action is already known by Corollary \ref{SecondPage}. Besides, we have $\mathrm H_c^{2(n-1)}(\mathcal M^{\mathrm{an}},\overline{\mathbb Q_{\ell}}) \simeq E_2^{0,2(n-1)} = \mathrm{Coker}(\varphi^{-1,2(n-1)})$. The differential $\varphi^{-1,2(n-1)}$ is described by
\begin{align*}
(\Lambda,\Lambda) & \mapsto 0, & & \forall \Lambda \in \mathcal L^{(\thetamax)},\\
(\Lambda,\Lambda') & \mapsto (\Lambda') - (\Lambda), & & \forall \Lambda,\Lambda' \in \mathcal L^{(\thetamax)} \text{ such that } U_{\Lambda} \cap U_{\Lambda'} \not = \emptyset. 
\end{align*}
Let $i \in \mathbb Z$ such that $ni$ is even, and let $\Lambda,\Lambda' \in \mathcal L_i^{(\thetamax)}$. Since the Bruhat-Tits building $\mathrm{BT}(\widetilde{J},\mathbb Q_p) \simeq \mathcal L_i$ is connected, there exists a sequence $\Lambda = \Lambda^0,\ldots ,\Lambda^{d} = \Lambda'$ of lattices in $\mathcal L_i$ such that for all $0\leq j \leq d-1$, $\{\Lambda^j,\Lambda^{j+1}\}$ is an edge in $\mathcal L_i$. Assume that $d\geq 0$ is minimal satisfying this property. Since $t(\Lambda) = t(\Lambda') = t_{\mathrm{max}}$, the integer $d$ is even and we may assume that $t(\Lambda^j)$ is equal to $t_{\mathrm{max}}$ when $j$ is even, and equal to $1$ when $j$ is odd. In particular, for all $0 \leq j \leq \frac{d}{2}-1$ we have $\Lambda^{2j},\Lambda^{2j+2} \in \mathcal L_i^{(\thetamax)}$ and $U_{\Lambda^{2j}}\cap U_{\Lambda^{2j+2}} \not = \emptyset$. Consider the vector 
$$w := \sum_{j=0}^{\frac{d}{2}-1} (\Lambda^{2j},\Lambda^{2j+2}) \in E_1^{-1,2(n-1)}.$$
Then we compute $\varphi^{-1,2(n-1)}(w) = (\Lambda') - (\Lambda)$. It follows that for all $\Lambda,\Lambda' \in \mathcal L_i$, we have $(\Lambda) \cong (\Lambda')$ in $\mathrm{Coker}(\varphi^{-1,2(n-1)})$. Thus, $\mathrm{Coker}(\varphi_1^{2(n-1)})$ consists of one copy of $\overline{\mathbb Q_{\ell}}$ for each $i\in \mathbb Z$ such that $ni$ is even. Considering the action of $J(\mathbb Q_p)$ as well, it readily follows that $\mathrm{Coker}(\varphi^{-1,2(n-1)}) \simeq \mathrm{c-Ind}_{J^{\circ}}^J\,\mathbf 1$.
\end{proof}

\begin{rk} The cohomology group $\mathrm H_c^{2(n-1)}(\mathcal M^{\mathrm{an}},\overline{\mathbb Q_{\ell}})$ can also be computed in another way which does not require the spectral sequence. Indeed, we have an isomorphism
$$\mathrm H_c^{2(n-1)}(\mathcal M^{\mathrm{an}},\overline{\mathbb Q_{\ell}}) \simeq \mathrm{c-Ind}_{J^{\circ}}^J \, \mathrm H_c^{2(n-1)}(\mathcal M^{\mathrm{an}}_0,\overline{\mathbb Q_{\ell}}).$$
By definition, we have 
$$\mathrm H_c^{2(n-1)}(\mathcal M^{\mathrm{an}}_0,\overline{\mathbb Q_{\ell}}) = \varinjlim_U \mathrm H_c^{2(n-1)}(U\widehat{\otimes} \, \mathbb C_p,\overline{\mathbb Q_{\ell}}),$$
where $U$ runs over the relatively compact open subspaces of $\mathcal M^{\mathrm{an}}_0$. Since $U$ is smooth, Poincaré duality gives 
$$\mathrm H_c^{2(n-1)}(U\widehat{\otimes} \, \mathbb C_p,\overline{\mathbb Q_{\ell}}) \simeq \mathrm H^0(U\widehat{\otimes} \, \mathbb C_p,\overline{\mathbb Q_{\ell}})^{\vee}.$$
Using the connectedness of the Bruhat-Tits building $\mathrm{BT}(\widetilde{J},\mathbb Q_p) \simeq \mathcal L_0$, one may prove that $\mathcal M_0^{\mathrm{an}}$ is connected. Thus we can insure that all the $U$'s involved in the limit are connected as well. Therefore $\mathrm H^0(U\widehat{\otimes} \, \mathbb C_p,\overline{\mathbb Q_{\ell}}) \simeq \overline{\mathbb Q_{\ell}}$, and all the transition maps in the direct limit are identity. It follows that $\mathrm H_c^{2(n-1)}(\mathcal M^{\mathrm{an}}_0,\overline{\mathbb Q_{\ell}})$ is trivial.
\end{rk}

\subsection{Compactly induced representations and type theory}

Let $\mathrm{Rep}(J(\mathbb Q_p))$ denote the category of smooth $\overline{\mathbb Q_{\ell}}$-representations of $J(\mathbb Q_p)$. Let $\chi$ be a continuous character of the center $\mathrm Z(J(\mathbb Q_p)) \simeq \mathbb Q_{p^2}^{\times}$ and let $V\in \mathrm{Rep}(J(\mathbb Q_p))$. We define \textbf{the maximal quotient of $V$ on which the center acts like $\chi$} as follows. Let us consider the set
$$\Omega := \{W \,|\, W \text{ is a subrepresentation of }V\text{ and }\mathrm Z(J(\mathbb Q_p))\text{ acts like }\chi\text{ on }V/W\}.$$
The set $\Omega$ is stable under arbitrary intersection, so that $W_{\circ} := \bigcap_{W\in \Omega} W \in \Omega$. The maximal quotient is defined by 
$$V_{\chi} := V/W_{\circ}.$$
It satisfies the following universal property. 

\begin{prop}
Let $\chi$ be a continuous character of $\mathrm Z(J(\mathbb Q_p))$ and let $V,V' \in \mathrm{Rep}(J(\mathbb Q_p))$. Assume that $\mathrm Z(J(\mathbb Q_p))$ acts like $\chi$ on $V'$. Then any morphism $V\to V'$ factors through $V_{\chi}$. 
\end{prop}

\begin{proof}
Let $f:V\to V'$ be a morphism of $J(\mathbb Q_p)$-representations. Since $V/\mathrm{Ker}(f) \simeq \mathrm{Im}(f) \subset V'$, the center $\mathrm Z(J(\mathbb Q_p))$ acts like $\chi$ on the quotient $V/\mathrm{Ker}(f)$. Therefore $\mathrm{Ker}(f) \in \Omega$. It follows that $\mathrm{Ker}(f)$ contains $W_{\circ}$ and as a consequence, $f$ factors through $V_{\chi}$.
\end{proof}

The terms $E_1^{a,b}$ of the spectral sequence \eqref{SpectralSequence} consist of representations of the form
$$\mathrm{c-Ind}_{J_{\theta}}^J \,\rho,$$
where $\rho$ is the inflation to $J_{\theta}$ of a representation of the finite group of Lie type $\mathcal J_{\theta}$. We note that such a compactly induced representation does not contain any smooth irreducible subrepresentation of $J(\mathbb Q_p)$. Indeed, the center $\mathrm Z(J(\mathbb Q_p)) \simeq \mathbb Q_{p^2}^{\times}$ does not fix any finite dimensional subspace. In order to rectify this, it is customary to fix a continuous character $\chi$ of $\mathrm{Z}(J(\mathbb Q_p))$ which agrees with the central character of $\rho$ on $\mathrm Z(J(\mathbb Q_p))\cap J_{\theta} \simeq \mathbb Z_{p^2}^{\times}$, and to describe the space $(\mathrm{c-Ind}_{J_{\theta}}^J \,\rho)_{\chi}$ instead.

\begin{lem}\label{AdjonctionForc-Ind}
We have $(\mathrm{c-Ind}_{J_{\theta}}^J \,\rho)_{\chi} \simeq \mathrm{c-Ind}_{\mathrm Z(J(\mathbb Q_p))J_{\theta}}^J \,\chi\otimes\rho$.
\end{lem}

\begin{proof}
By Frobenius reciprocity, the identity map on $\mathrm{c-Ind}_{\mathrm Z(J(\mathbb Q_p))J_{\theta}}^J \,\chi\otimes\rho$ gives a morphism $\chi\otimes \rho \rightarrow \left(\mathrm{c-Ind}_{\mathrm Z(J(\mathbb Q_p))J_{\theta}}^J \,\chi\otimes\rho\right)_{|\mathrm Z(J(\mathbb Q_p))J_{\theta}}$ of $\mathrm Z(J(\mathbb Q_p))J_{\theta}$-representations. Restricting further to $J_{\theta}$, we obtain a morphism $\rho \rightarrow \left(\mathrm{c-Ind}_{\mathrm Z(J(\mathbb Q_p))J_{\theta}}^J \,\chi\otimes\rho\right)_{|J_{\theta}}$. This corresponds to a morphism $\mathrm{c-Ind}_{J_{\theta}}^J \,\rho \rightarrow \mathrm{c-Ind}_{\mathrm Z(J(\mathbb Q_p))J_{\theta}}^J \,\chi\otimes\rho$ of $J(\mathbb Q_p)$-representations by Frobenius reciprocity. Since $\mathrm Z(J(\mathbb Q_p))$ acts via the character $\chi$ on the target space, this morphism factors through a map $(\mathrm{c-Ind}_{J_{\theta}}^J \,\rho)_{\chi} \rightarrow \mathrm{c-Ind}_{\mathrm Z(J(\mathbb Q_p))J_{\theta}}^J \,\chi\otimes\rho$. In order to prove that this is an isomorphism, we build its inverse. The quotient morphism $\mathrm{c-Ind}_{J_{\theta}}^J \,\rho \rightarrow (\mathrm{c-Ind}_{J_{\theta}}^J \,\rho)_{\chi}$ corresponds, via Frobenius reciprocity, to a morphism $\rho \rightarrow (\mathrm{c-Ind}_{J_{\theta}}^J \,\rho)_{\chi\,|J_{\theta}}$ of $J_{\theta}$-representations. Because $\mathrm Z(J(\mathbb Q_p))$ acts via the character $\chi$ on the target space, this arrow may be extended to a morphism $\chi \otimes \rho \rightarrow (\mathrm{c-Ind}_{J_{\theta}}^J \,\rho)_{\chi \, |\mathrm Z(J(\mathbb Q_p))J_{\theta}}$ of $\mathrm Z(J(\mathbb Q_p))J_{\theta}$-representations. By Frobenius reciprocity, this corresponds to a morphism $\mathrm{c-Ind}_{\mathrm Z(J(\mathbb Q_p))J_{\theta}}^J \,\chi\otimes\rho \rightarrow (\mathrm{c-Ind}_{J_{\theta}}^J \,\rho)_{\chi}$, and this is our desired inverse.
\end{proof}

We recall Theorem 2 (supp) from \cite{bushnell} describing certain compactly induced representations. In this paragraph only, let $G$ be any $p$-adic group, and let $L$ be an open subgroup of $G$ which contains the center $\mathrm Z(G)$ and which is compact modulo $\mathrm{Z}(G)$. 

\begin{theo}\label{0-InfinityParts} 
Let $(\sigma,V)$ be an irreducible smooth representation of $L$. There is a canonical decomposition 
$$\mathrm{c-Ind}_{L}^{G} \, \sigma \simeq V_0 \oplus V_{\infty},$$
where $V_0$ is the sum of all supercuspidal subrepresentations of $\mathrm{c-Ind}_{L}^{G} \, \sigma$, and where $V_{\infty}$ contains no non-zero admissible subrepresentation. Moreover, $V_0$ is a finite sum of irreducible supercuspidal subrepresentations of $G$.
\end{theo}

The spaces $V_0$ or $V_{\infty}$ could be zero. Note also that since $G$ is $p$-adic, any irreducible representation is admissible. So in particular, $V_{\infty}$ does not contain any irreducible subrepresentation. However, it may have many irreducible quotients and subquotients. Thus, the space $V_{\infty}$ is in general not $G$-semisimple. Hence, the structure of the compactly induced representation $\mathrm{c-Ind}_{L}^{G} \, \sigma$ heavily depends on the supercuspidal supports of its irreducible subquotients.\\
We go back to our previous notations. Let $0\leq \theta \leq \thetamax$, let $\rho$ be a smooth irreducible representation of $J_{\theta}$ and let $\chi$ be a character of $\mathrm{Z}(J(\mathbb Q_p))$ agreeing with the central character of $\rho$ on $\mathrm{Z}(J(\mathbb Q_p)) \cap J_{\theta}$. Since the group $\mathrm{Z}(J(\mathbb Q_p))J_{\theta}$ contains the center and is compact modulo the center, we have a canonical decomposition 
$$(\mathrm{c-Ind}_{J_{\theta}}^{J}\, \rho)_{\chi} \simeq V_{\rho,\chi,0} \oplus V_{\rho,\chi,\infty}.$$ 
In order to describe the spaces $V_{\rho,\chi,0}$ and $V_{\rho,\chi,\infty}$, we determine the supercuspidal supports of the irreducible subquotients of $\mathrm{c-Ind}_{J_{\theta}}^{J}\, \rho$ through type theory, with the assumption that $\rho$ is inflated from $\mathcal J_{\theta}$. For our purpose, it will be enough to analyze only the case $\theta = \thetamax$. In this case, $\dim V_{\thetamax}^1$ is equal to $0$ or $1$ so that $\mathrm{GU}(V^1_{\thetamax}) = \{1\}$ or $\mathbb F_{p^2}^{\times}$ has no proper parabolic subgroup. In particular, if $\rho$ is a cuspidal representation of $\mathrm{GU}(V^0_{\thetamax})$, then its inflation to the reductive quotient
$$\mathcal J_{\thetamax} \simeq \mathrm{G}(\mathrm{U}(V^0_{\thetamax})\times \mathrm{U}(V^1_{\thetamax}))$$
is also cuspidal. 

In the following paragraphs, we recall a few general facts from type theory. For more details, we refer to \cite{bk} and \cite{morris}. Let $G$ be the group of $F$-rational points of a reductive connected group $\mathbf G$ over a $p$-adic field $F$. A parabolic subgroup $P$ (resp. Levi complement $L$) of $G$ is defined as the group of $F$-rational points of an $F$-rational parabolic subgroup $\mathbf P \subset \mathbf G$ (resp. an $F$-rational Levi complement $\mathbf L \subset \mathbf G$). Every parabolic subgroup $P$ admits a Levi decomposition $P = LU$ where $U$ is the unipotent radical of $P$. We denote by $X^{\mathrm{un}}(G)$ the set of \textbf{unramified characters} of $G$, ie. the continuous characters of $G$ which are trivial $G^{\circ} := \bigcap_{\psi} \mathrm{Ker} |\psi|_F$ where $\psi$ runs over all the $F$-rational algebraic characters of $G$ and $|\,\cdot\,|_F$ is the normalized valuation on $F$. We consider pairs $(L,\tau)$ where $L$ is a Levi complement of $G$ and $\tau$ is a supercuspidal representation of $L$. Two pairs $(L,\tau)$ and $(L',\tau')$ are said to be \textbf{inertially equivalent} if for some $g\in G$ and $\chi \in X^{\mathrm{un}}(G)$ we have $L' = L^g$ and $\tau' \simeq \tau^g \otimes \chi$ where $\tau^g$ is the representation of $L^g$ defined by $\tau^g(l) := \tau(g^{-1}lg)$. This is an equivalence relation, and we denote by $[L,\tau]_G$ or $[L,\tau]$ the inertial equivalence class of $(L,\tau)$ in $G$. The set of all inertial equivalence classes is denoted $\mathrm{IC}(G)$. If $P$ is a parabolic subgroup of $G$, we write $\iota_P^G$ for the normalised parabolic induction functor. Any smooth irreducible representation $\pi$ of $G$ is isomorphic to a subquotient of some parabolically induced representation $\iota_P^G(\tau)$, where $P = LU$ for some Levi complement $L$ and $\tau$ is a supercuspidal representation of $L$. We denote by $\ell(\pi) \in \mathrm{IC}(G)$ the inertial equivalence class $[L,\tau]$. This is uniquely determined by $\pi$ and it is called the \textbf{inertial support} of $\pi$.

Let $\mathfrak s \in \mathrm{IC}(G)$. We denote by $\mathrm{Rep}^{\mathfrak s}(G)$ the full subcategory of $\mathrm{Rep}(G)$ whose objects are the smooth representations of $G$ all of whose irreducible subquotients have inertial support $\mathfrak s$. This definition corresponds to the one given in \cite{bernstein} Proposition-Définition 2.8. If $\mathfrak S \subset \mathrm{IC}(G)$, we write $\mathrm{Rep}^{\mathfrak S}(G)$ for the direct product of the categories $\mathrm{Rep}^{\mathfrak s}(G)$ where $\mathfrak s$ runs over $\mathfrak S$. The following statement is Proposition 2.10 of loc. cit.

\begin{theo}\label{BernsteinDecomposition}
The category $\mathrm{Rep}(G)$ decomposes as the direct product of the subcategories $\mathrm{Rep}^{\mathfrak s}(G)$ where $\mathfrak s$ runs over $\mathrm{IC}(G)$. Moreover, if $\mathfrak S \subset \mathrm{IC}(G)$ then the category $\mathrm{Rep}^{\mathfrak S}(G)$ is stable under direct sums and subquotients. 
\end{theo}

Type theory was then introduced in \cite{bk} in order to describe the categories $\mathrm{Rep}^{\mathfrak s}(G)$ which are called the \textbf{Bernstein blocks}. Let $\mathfrak S$ be a subset of $\mathrm{IC}(G)$. A $\mathfrak S$\textbf{-type} in $G$ is a pair $(K,\rho)$ where $K$ is an open compact subgroup of $G$ and $\rho$ is a smooth irreducible representation of $K$, such that for every smooth irreducible representation $\pi$ of $G$ we have 
$$\pi_{|K} \text{ contains } \rho \iff \ell(\pi) \in \mathfrak S.$$
When $\mathfrak S$ is a singleton $\{\mathfrak s\}$, we call it an $\mathfrak s$-type instead.

\begin{rk}
By Frobenius reciprocity, the condition that $\pi_{|K}$ contains $\rho$ is equivalent to $\pi$ being isomorphic to an irreducible quotient of $\mathrm{c-Ind}_K^G\,\rho$. In fact, we can say a little bit more. Let $K$ be an open compact subgroup of $G$ and let $\rho$ be an irreducible smooth representation of $K$. Let $\mathrm{Rep}_{\rho}(G)$ denote the full subcategory of $\mathrm{Rep}(G)$ whose objects are those representations which are generated by their $\rho$-isotypic component. If $(K,\rho)$ is an $\mathfrak S$-type, then \cite{bk} Theorem 4.3 establishes the equality of categories $\mathrm{Rep}_{\rho}(G) = \mathrm{Rep}^{\mathfrak S}(G)$. By definition of compact induction, the representation $\mathrm{c-Ind}_K^G\, \rho$ is generated by its $\rho$-isotypic vectors. Therefore any irreducible subquotient of $\mathrm{c-Ind}_K^G\,\rho$ has inertial support in $\mathfrak S$.
\end{rk}

An important class of types are those of depth zero, and they are the only ones we shall encounter. First, we recall the following result. If $K$ is a parahoric subgroup of $G$, we denote by $\mathcal K$ its maximal reductive quotient. It is a finite group of Lie type over the residue field of $F$. The following statement is \cite{morris} Proposition 4.1

\begin{prop}\label{Depth-0Types}
Let $K$ be a maximal parahoric subgroup of $G$ and let $\rho$ be an irreducible cuspidal representation of $\mathcal K$, seen as a representation of $K$ by inflation. Let $\pi$ be an irreducible smooth representation of $G$ and assume that $\pi_{|K}$ contains $\rho$. Then $\pi$ is supercuspidal and there exists an irreducible smooth representation $\tilde{\rho}$ of the normalizer $\mathrm N_G(K)$ such that $\tilde{\rho}_{|K}$ contains $\rho$ and $\pi \simeq \mathrm{c-Ind}_{N_G(K)}^G \tilde{\rho}$.
\end{prop}

Such representations $\pi$ are called \textbf{depth-0 supercupidal representations} of $G$. More generally, a smooth irreducible representation $\pi$ of $G$ is said to be of \textbf{depth-0} if it contains a non-zero vector that is fixed by the pro-unipotent radical of some parahoric subgroup of $G$. A \textbf{depth-0 type} in $G$ is a pair $(K,\rho)$ where $K$ is a parahoric subgroup of $G$ and $\rho$ is an irreducible cuspidal representation of $\mathcal K$, inflated to $K$. The name is justified by \cite{morris} Theorem 4.8. 

\begin{theo}
Let $(K,\rho)$ be a depth-0 type. Then there exists a (unique) finite set $\mathfrak S \subset \mathrm{IC}(G)$ such that $(K,\rho)$ is an $\mathfrak S$-type of $G$. 
\end{theo}

Let $K$ be a parahoric subgroup of $G$. Using the Bruhat-Tits building of $G$, one may canonically associate a Levi complement $L$ of $G$ such that $K_L := L \cap K$ is a maximal parahoric subgroup of $L$, whose maximal reductive quotient $\mathcal K_L$ is naturally identified with $\mathcal K$. This is precisely described in \cite{morris} paragraph 2.1. Moreover, we have $L = G$ if and only if $K$ is a maximal parahoric subgroup of $G$. Now, let $(K,\rho)$ be a depth-0 type of $G$ and denote by $\mathfrak S$ the finite subset of $\mathrm{IC}(G)$ such that it is an $\mathfrak S$-type of $G$. Since $\rho$ is a cuspidal representation of $\mathcal K \simeq \mathcal K_L$, we may inflate it to $K_L$. Then, the pair $(K_L,\rho)$ is a depth-0 type of $L$. We say that $(K,\rho)$ is a $G$\textbf{-cover} of $(K_L,\rho)$. By the previous theorem, there is a finite set $\mathfrak S_L \subset \mathrm{IC}(L)$ such that $(K_L,\rho)$ is an $\mathfrak S_L$-type of $L$. Then the proof of Theorem 4.8 in \cite{morris} shows that we have the relation 
$$\mathfrak S = \left\{[M,\tau]_G \,\big|\, [M,\tau]_L \in \mathfrak S_L\right\}.$$
In this set, $M$ is some Levi complement of $L$, therefore it may also be seen as a Levi complement in $G$. Thus, an inertial equivalence class $[M,\tau]_L$ in $L$ gives rise to a class $[M,\tau]_G$ in $G$. Since $K_L$ is maximal in $L$, in virtue of the proposition above any element of $\mathfrak S_L$ has the form $[L,\pi]_L$ for some supercuspidal representation $\pi$ of $L$. In particular, every smooth irreducible representation of $G$ containing the type $(K,\rho)$ has a conjugate of $L$ as cuspidal support. We deduce the following corollary. 

\begin{corol}
Let $(K,\rho)$ be a depth-0 type in $G$ and assume that $K$ is not a maximal parahoric subgroup. Then no smooth irreducible representation $\pi$ of $G$ containing the type $(K,\rho)$ is supercuspidal.
\end{corol}

Thus, up to replacing $G$ with a Levi complement, the study of any depth-0 type $(K,\rho)$ can be reduced to the case where $K$ is a maximal parahoric subgroup. Let us assume that it is the case, and let $\mathfrak S$ be the associated finite subset of $\mathrm{IC}(G)$. While $\mathfrak S$ is in general not a singleton, it becomes one once we modify the pair $(K,\rho)$ a little bit according to \cite{morris} Theorem Variant 4.7. Let $\widehat{K}$ be the maximal open compact subgroup of $\mathrm N_G(K)$. We have $K \subset \widehat{K}$ but in general this inclusion may be strict. Let $\tilde{\rho}$ be a smooth irreducible representation of $\mathrm N_G(K)$ such that $\tilde{\rho}_{|K}$ contains $\rho$. Let $\widehat{\rho}$ be any irreducible component of the restriction $\tilde{\rho}_{|\widehat{K}}$. Eventually, let $\pi := \mathrm{c-Ind}_{\mathrm N_G(K)}^G \, \tilde{\rho}$ be the associated depth-0 supercuspidal representation of $G$. 

\begin{theo}\label{TypeSingleton}
The pair $(\widehat{K},\widehat{\rho})$ is a $[G,\pi]$-type.
\end{theo}

The conclusion does not depend on the choice of $\widehat{\rho}$ as an irreducible component of $\tilde{\rho}_{|\widehat{K}}$. Any one of them affords a type for the same singleton $\mathfrak s = [G,\pi]$. Let us now consider a parahoric subgroup $K$ along with an irreducible representation $\rho$ of its maximal reductive quotient $\mathcal K = K/K^+$, where $K^+$ is the pro-unipotent radical of $K$. Assume that $\rho$ is not cuspidal. Thus, there exists a proper parabolic subgroup $\mathcal P \subset \mathcal K$ with Levi complement $\mathcal L$, and a cuspidal irreducible representation $\tau$ of $\mathcal L$, such that $\rho$ is an irreducible component of the Harish-Chandra induction $\iota_{\mathcal P}^{\mathcal K}\,\tau$. The preimage of $\mathcal P$ via the quotient map $K \twoheadrightarrow \mathcal K$ is a parahoric subgroup $K' \subsetneq K$, whose maximal reductive quotient $\mathcal K' := K'/K'^+$ is naturally identified with $\mathcal L$. We have $K^+ \subset K'^+ \subset K'$ and the intermediate quotient $K'^+/K^+$ is identified with the unipotent radical $\mathcal N$ of $\mathcal P \simeq K'/K^+$. Consider $\rho$ as an irreducible representation of $K$ inflated from $\mathcal K$. The invariants $\rho^{K'^+}$ form a representation of $K'$ which coincides with the inflation of the Harish-Chandra restriction of $\rho$ (as a representation of $\mathcal K$) to $\mathcal L$. Thus, $\rho^{K'^+}$ contains the inflation of $\tau$ to a representation of $K'$. In other words, we have a $K'$-equivariant map 
$$\tau \rightarrow \rho_{|K'}.$$
By Frobenius reciprocity, it gives a map 
$$\mathrm{c-Ind}_{K'}^{K}\,\tau \rightarrow \rho,$$
which is surjective by irreducibility of $\rho$. Applying the functor $\mathrm{c-Ind}_K^G: \mathrm{Rep}(K) \to \mathrm{Rep}(G)$, which is exact, and using transitivity of compact induction, we deduce the existence of a natural surjection 
$$\mathrm{c-Ind}_{K'}^G \, \tau \twoheadrightarrow \mathrm{c-Ind}_K^G\, \rho.$$
Now, $(K',\tau)$ is a depth-0 type in $G$. Let $\mathfrak S \subset \mathrm{IC}(G)$ be the subset such that $(K',\tau)$ is an $\mathfrak S$-type, and let $L$ be the (proper) Levi complement of $G$ associated to $K'$ as in the previous paragraph. By Remark \ref{TypeDefinition}, it follows that any irreducible subquotient of $\mathrm{c-Ind}_K^G\, \rho$ has inertial support in $\mathfrak S$. Since all elements of $\mathfrak S$ are of the form $[L,\pi]$ for some supercuspidal representation $\pi$ of $L$, we reach the following conclusion. 

\begin{prop}\label{NoSupercuspidals}
Let $K$ be a parahoric subgroup of $G$ and let $\rho$ be a non cuspidal irreducible representation of its maximal reductive quotient $\mathcal K$. Then no irreducible subquotient of $\mathrm{c-Ind}_K^G\,\rho$ is supercuspidal.
\end{prop}

We go back to the context of the unitary similitude group $J(\mathbb Q_p)$. We may now determine the inertial support of any irreducible subquotient of a representation of the form $\mathrm{c-Ind}_{J_{\thetamax}}^J \, \rho$ with $\rho$ inflated from a unipotent representation of $\mathrm{GU}(V_{\thetamax}^0)$. In particular, all the terms $E_1^{0,b}$ are of this form according to Corollary \ref{MiddleCohomologyGroups}. More precisely, let $\lambda$ be a partition of $2\thetamax+1$ and let $\Delta_t$ be its $2$-core (see Section 2). Thus $2\thetamax+1 = \frac{t(t+1)}{2} + 2e$ for some $e\geq 0$. The integer $\frac{t(t+1)}{2}$ is odd, so it can be written as $2f+1$ for some $f\geq 0$, and we have $\thetamax = f+e$. Recall the basis of $\mathbf V$ that we fixed in Section 1.1. The images of the vectors $e_{\pm i}$ for $1 \leq i \leq \thetamax$ and of $e_0^{\mathrm{an}}$ in $V_{\thetamax}^0 = \Lambda_{\thetamax} / p\Lambda_{\thetamax}$ define a basis of $V_{\thetamax}^0$, allowing us to identify $\mathrm{GU}(V_{\thetamax}^0)$ with the matrix group $\mathrm{GU}_{2\thetamax+1}(\mathbb F_p)$. The cuspidal support of $\rho_{\lambda}$ is $(L_t,\rho_t)$ according to Section 2. Let $P_t$ be the standard parabolic subgroup with Levi complement $L_t$. By direct computation, one may check that the preimage of $P_t$ in $J_{\thetamax}$ is the parahoric subgroup $J_{f,\ldots,\thetamax} := J_f \cap J_{f+1} \cap \ldots \cap J_{\thetamax}$. Let $L_f$ be the Levi complement of $J(\mathbb Q_p)$ that is associated to the parahoric subgroup $J_{f,\ldots ,\thetamax}$. Let $\mathbf V^f$ be the subspace of $\mathbf V$ generated by $\mathbf V^{\mathrm{an}}$ and by the vectors $e_{\pm 1},\ldots , e_{\pm f}$. It is equipped with the restriction of the hermitian form of $\mathbf V$. Then $L_f \simeq \mathrm{G}(\mathrm{U}(\mathbf V^f)\times\mathrm{U}_1(\mathbb Q_{p})^{e})$.\\
The group $L_f\cap J_{f,\ldots,\thetamax}$ is a maximal parahoric subgroup of $L_f$, and $\rho_t$ can be inflated to it. In particular, the pair $(L_f\cap J_{f,\ldots,\thetamax},\rho_t)$ is a level-0 type in $L_f$. Since we work with unitary groups over an unramified quadratic extension, $L_f\cap J_{f,\ldots,\thetamax}$ is also a maximal compact subgroup of $L_f$. In particular, $(L_f\cap J_{f,\ldots,\thetamax},\rho_t)$ is a type for a singleton of the form $[L_f,\tau_f]_{L_f}$. Then $\tau_f$ has the form
$$\tau_f = \mathrm{c-Ind}_{\mathrm N_{L_f}(L_f\cap J_{f,\ldots,\thetamax})}^{L_f} \, \widetilde{\rho_t},$$
where $\widetilde{\rho_t}$ is some smooth irreducible representation of $\mathrm N_{L_f}(L_f\cap J_{f,\ldots,\thetamax})$ containing $\rho_t$ upon restriction. It follows that if we inflate $\rho_t$ to $J_{f,\ldots ,\thetamax}$ then $(J_{f,\ldots ,\thetamax},\rho_t)$ is a $[L_f,\tau_f]$-type in $J(\mathbb Q_p)$. Moreover the compactly induced representation $\mathrm{c-Ind}_{J_{\thetamax}}^{J} \, \rho_{\lambda}$ is a quotient of $\mathrm{c-Ind}_{J_{f,\ldots ,\thetamax}}^{J}\, \rho_t$. In particular, we reach the following conclusion. 

\begin{prop}
Let $\lambda$ be a partition of $2\thetamax+1$ with $2$-core $\Delta_t$. Write $\frac{t(t+1)}{2} = 2f+1$ for some $f\geq 0$. Any irreducible subquotient of $\mathrm{c-Ind}_{J_{\thetamax}}^{J}\,\rho_{\lambda}$ has inertial support $[L_f,\tau_f]$. 
\end{prop}

In particular, if $f < \thetamax$ then none of these irreducible subquotients are supercuspidal.

Let us keep the notations of the previous paragraph. Since unipotent representations of finite groups of Lie type have trivial central characters, if $\chi$ is an unramified character of $\mathrm Z(J(\mathbb Q_p))$ then $\chi_{\mathrm Z(J(\mathbb Q_p)) \cap J_{\thetamax}}$ coincides with the central character of $\rho_{\lambda}$ inflated to $J_{\thetamax}$. As in Theorem \ref{0-InfinityParts}, we have 
$$\left(\mathrm{c-Ind}_{J_{\thetamax}}^J\,\rho_{\lambda}\right)_{\chi} \simeq V_{\rho_{\lambda},\chi,0} \oplus V_{\rho_{\lambda},\chi,\infty}.$$
If $f < \thetamax$, then no irreducible supercuspidal representation can occur. Thus $V_{\rho_{\lambda},\chi,0} = 0$. \\
On the other hand, assume now that $f = \thetamax$ so that $L_f = J$ and $\rho_{\lambda}$ is equal to the cuspidal representation $\rho_{\Delta_{\thetamax}}$. As seen in Proposition \ref{Normalizers}, we have $\mathrm N_J(J_{\thetamax}) = \mathrm Z(J(\mathbb Q_p))J_{\thetamax}$ unless $n=2$ (thus $\thetamax=0$) in which case $J_0 = J^{\circ}$ and $\mathrm Z(J(\mathbb Q_p))J_0$ is of index $2$ in $\mathrm N_J(J_0) = J$. A representative of the non-trivial coset is given by $g_0$ as defined in Section 1.1. If $n\not = 2$, define 
$$\tau_{\thetamax,\chi} := \mathrm{c-Ind}_{\mathrm Z(J(\mathbb Q_p))J_{\thetamax}}^J\, \chi\otimes\rho_{\lambda}.$$
Then $\tau_{\thetamax,\chi}$ is an irreducible supercuspidal representation of $J(\mathbb Q_p)$, and we have
$$\left(\mathrm{c-Ind}_{J_{\thetamax}}^J\,\rho_{\lambda}\right)_{\chi} \simeq \mathrm{c-Ind}_{\mathrm Z(J(\mathbb Q_p))J_{\thetamax}}^J\,\chi\otimes\rho_{\lambda} = \tau_{\thetamax,\chi}.$$
Thus $V_{\rho_{\lambda},\chi,\infty} = 0$ and $V_{\rho_{\lambda},\chi\infty} = \tau_{\thetamax,\chi}$ in this case.\\
When $n = 2$, $\rho_{\lambda} = \rho_{\Delta_0} = \mathbf 1$ is the trivial representation of $J_0 = J^{\circ}$. Let $\chi_0:J\rightarrow \overline{\mathbb Q_{\ell}}^{\times}$ be the unique non-trivial character of $J(\mathbb Q_p)$ which is trivial on $\mathrm Z(J(\mathbb Q_p))J_0$. Then $\left(\mathrm{c-Ind}_{J_0}^J\,\mathbf 1\right)_{\chi}$ is the sum of an unramified character $\tau_{0,\chi}$ of $J(\mathbb Q_p)$ whose central character is $\chi$, and of the character $\chi_0\tau_{0,\chi}$. Both characters are supercuspidal, and they are the only unramified characters of $J(\mathbb Q_p)$ with central character $\chi$.

According to Proposition \ref{CohomologyOpenBT} and Corollary \ref{MiddleCohomologyGroups}, the terms $E_1^{0,b}$ are a sum of representations of the form 
$$\mathrm{c-Ind}_{J_{\thetamax}}^J\,\rho_{\lambda},$$
with $\lambda$ a partition of $2\thetamax+1$ having $2$-core $\Delta_0$ if $b$ is even, and $\Delta_1$ if $b$ is odd. Moreover, we have 
\begin{align*}
E_2^{0,2(n-1-\thetamax)} \simeq \mathrm{c-Ind}_{J_{\thetamax}}^J \, \mathbf 1, & & E_2^{0,2(n-1-\thetamax)+1} \simeq \mathrm{c-Ind}_{J_{\thetamax}}^J \, \rho_{(2\thetamax,1)}.
\end{align*}
In particular, summing up the discussion of the previous paragraph, we have reached the following statement.

\begin{prop}\label{InertialSupportMiddleCohomology}
Let $\chi$ be an unramified character of $\mathrm Z(J(\mathbb Q_p))$. 
\begin{enumerate}[label={--},noitemsep]
\item Assume that $n\geq 3$. The representation $(E_2^{0,2(n-1-\thetamax)})_{\chi}$ contains no non-zero admissible subrepresentation, and it is not $J(\mathbb Q_p)$-semisimple. Moreover, any irreducible subquotient has inertial support $[L_0,\tau_0]$. If $n\geq 5$, then the same statement holds for $(E_2^{0,2(n-1-\thetamax)+1})_{\chi}$ with the inertial support being $[L_1,\tau_1]$.
\item For $n = 1,2,3,4$, let $b = 0,2,3,5$ respectively. Then $\thetamax = 0$ when $1,2$ and $\thetamax=1$ when $n = 3,4$. Let $\chi$ be an unramified character of $\mathrm{Z}(J(\mathbb Q_p))$. The representation $\tau_{\thetamax,\chi}$ is irreducible supercuspidal, and we have 
$$(E_2^{0,b})_{\chi} \simeq 
\begin{cases}
\tau_{\thetamax,\chi} & \text{if } n = 1,3,4,\\
\tau_{\thetamax,\chi}\oplus\chi_0\tau_{\thetamax,\chi} & \text{if } n=2.
\end{cases}$$ 
\end{enumerate}
\end{prop}

In particular, we deduce the following important corollary.

\begin{corol}
Let $\chi$ be an unramified character of $\mathrm Z(J(\mathbb Q_p))$. If $n\geq 3$ then $\mathrm H_c^{2(n-1-\thetamax)}(\mathcal M^{\mathrm{an}},\overline{\mathbb Q_{\ell}})_{\chi}$ is not $J(\mathbb Q_p)$-admissible. If $n\geq 5$ then the same holds for $\mathrm H_c^{2(n-1-\thetamax)+1}(\mathcal M^{\mathrm{an}},\overline{\mathbb Q_{\ell}})_{\chi}$.
\end{corol}

\subsection{The case $n=3,4$}

Let us focus on the case $\thetamax=1$, that is $n=3$ or $4$. Recall that $N(\Lambda_0)$ denotes the set of lattices $\Lambda \in \mathcal L_0$ with type $t(\Lambda) = t_{\mathrm{max}} = 3$ containing $\Lambda_0$. It has cardinality $\nu(1,2) = p+1$ when $n=3$ and $\nu(2,3) = p^3+1$ when $n=4$. In particular, we may locate the non zero terms $E_{1,\mathrm{alt}}^{a,b}$ of the alternating \v{C}ech spectral sequence as follows.
$$E_{1,\mathrm{alt}}^{a,b} \not = 0 \iff \begin{cases}
  (a,b) \in \{(0,2);(0,3);(-k,4) \,|\, 0\leq k \leq p\} & \text{if } n = 3, \\
  (a,b) \in \{(0,4);(0,5);(-k,6) \,|\, 0\leq k \leq p^3\} & \text{if } n = 4.
\end{cases}$$ 

In Figure 1 below, we draw the shape of the first page $E_{1,\mathrm{alt}}$ for $n=3$. The case of $n=4$ is similar, except that two more $0$ rows should be added at the bottom. To alleviate the notations, we write $\varphi_{-a}$ for the differential $\varphi^{a,2(n-1)}$. 

\begin{figure}[h]
\centering
\begin{tikzcd}
\ldots \arrow{r}{\varphi_4} & E_{1,\mathrm{alt}}^{-3,4} \arrow{r}{\varphi_3} & E_{1,\mathrm{alt}}^{-2,4} \arrow{r}{\varphi_2} & E_{1,\mathrm{alt}}^{-1,4} \arrow{r}{\varphi_1} & \mathrm{c-Ind}^J_{J_{1}}\mathbf{1} \\
\, & \, & \, & \, & \mathrm{c-Ind}^J_{J_{1}}\,\rho_{\Delta_2} \\
\, & \, & \, & \, & \mathrm{c-Ind}^J_{J_{1}}\,\mathbf{1}\\
\, & \, & \, & \, & 0 \\
\, & \, & \, & \, & 0 
\end{tikzcd}
\caption{The first page $E_{1,\mathrm{alt}}$ of the alternating \v{C}ech spectral sequence when $n=3$.}
\end{figure}

Let $i\in \mathbb Z$ such that $ni$ is even. For $\Lambda,\Lambda' \in \mathcal L_i$, we define the distance $d(\Lambda,\Lambda')$ as the smallest integer $d\geq 0$ such that there exists a sequence $\Lambda = \Lambda^0,\ldots ,\Lambda^{d} = \Lambda'$ of lattices of $\mathcal L_i$ with $\{\Lambda^j,\Lambda^{j+1}\}$ being an edge for all $0 \leq j \leq d-1$. This definition makes sense for any $n$. When $\thetamax=1$, any lattice $\Lambda \in \mathcal L_i$ has type $1$ or $3$, and two lattices forming an edge can not have the same type. Therefore, the value of $t(\Lambda^j)$ alternates between $1$ and $3$. In particular, if $t(\Lambda) = t(\Lambda')$ then $d(\Lambda,\Lambda')$ is even. According to \cite{vw1} Proposition 3.7, the simplicial complex $\mathcal L_i$ is in fact a tree. We will use this to prove the following proposition.

\begin{prop}\label{AcyclicProof}
Assume that $n=3$ or $4$. We have $E_2^{-1,2(n-1)} = 0$.
\end{prop}

For now, $n \geq 3$ is still any integer. By Proposition \ref{AlternatingCech}, we may use the alternating \v{C}ech spectral sequence to show that $E_2^{-1,2(n-1)} = \mathrm{Ker}(\varphi_1)/\mathrm{Im}(\varphi_2)$ vanishes. The term $E_1^{a,2(n-1)}$ is the $\overline{\mathbb Q_{\ell}}$-vector space generated by the set $I_{-a+1}$, and $E_{1,\mathrm{alt}}^{a,2(n-1)}$ is the subspace consisting of all the vectors $v = \sum_{\gamma \in I_{-a+1}} \lambda_{\gamma}\gamma$ such that for all $\sigma \in \mathfrak S_{-a+1}$ we have $\lambda_{\sigma(\gamma)} = \mathrm{sgn}(\sigma)\lambda_{\gamma}$. Here the $\lambda_{\gamma}$'s are scalars which are almost all zero. To prove the proposition, let us look at the differential $\varphi_2$. It acts on the basis vectors in the following way.
\begin{align*}
\left.
\begin{array}{c}
(\Lambda,\Lambda,\Lambda)\\
(\Lambda,\Lambda,\Lambda')\\
(\Lambda',\Lambda,\Lambda)
\end{array}
\right\} & \mapsto (\Lambda,\Lambda), & & \forall \Lambda,\Lambda' \in \mathcal L^{(1)} \text{ such that } U_{\Lambda} \cap U_{\Lambda'} \not = \emptyset,\\
(\Lambda,\Lambda',\Lambda) & \mapsto (\Lambda',\Lambda) - (\Lambda,\Lambda) + (\Lambda,\Lambda'), & & \forall \Lambda,\Lambda' \in \mathcal L^{(1)} \text{ such that } U_{\Lambda} \cap U_{\Lambda'} \not = \emptyset, \\
(\Lambda,\Lambda',\Lambda'') & \mapsto (\Lambda',\Lambda'') - (\Lambda,\Lambda'') + (\Lambda,\Lambda'), & & \forall \Lambda,\Lambda',\Lambda'' \in \mathcal L^{(1)} \text{ such that } U_{\Lambda} \cap U_{\Lambda'} \cap U_{\Lambda''} \not = \emptyset.
\end{align*}

We note that for a collection of lattices $\Lambda^1,\ldots ,\Lambda^s \in \mathcal L_i^{(1)}$, the condition $U_{\Lambda^1} \cap \ldots \cap U_{\Lambda^s} \not = \emptyset$ is equivalent to $d(\Lambda^j,\Lambda^{j'}) = 2$ for all $1\leq j \not = j' \leq s$. Towards a contradiction, we assume that $\mathrm{Im}(\varphi_2) \subsetneq \mathrm{Ker}(\varphi_1)$. Let $v \in \mathrm{Ker}(\varphi_1) \setminus \mathrm{Im}(\varphi_2)$. Since $v \in E^{-1,2(n-1)}_{1,\mathrm{alt}}$, it decomposes under the form
$$v = \sum_{j=1}^{r} \lambda_{j} (\gamma_j - \tau(\gamma_j)),$$
where $r\geq 1$, the $\gamma_j$'s are of the form $(\Lambda,\Lambda')$ with $d(\Lambda,\Lambda') = 2$, the scalars $\lambda_j$'s are non zero and $\tau \in \mathfrak S_2$ is the transposition. We may assume that $r$ is minimal among all the vectors in the complement $\mathrm{Ker}(\varphi_1) \setminus \mathrm{Im}(\varphi_2)$. In particular, there exists a single $i\in \mathbb Z$ such that $ni$ is even, and for all $1\leq j \leq r$ the lattices in $\gamma_j$ belong to $\mathcal L_i^{(1)}$. We may further assume $i=0$ without loss of generality. We say that an element $\gamma \in I_2$ occurs in $v$ if $\gamma = \gamma_j$ or $\tau(\gamma_j)$ for some $1\leq j \leq r$. Similarly, we say that a lattice $\Lambda \in \mathcal L^{(1)}_0$ occurs in $v$ if it is a constituent of some $\gamma_j$.

\begin{lem}
Let $\gamma = (\Lambda,\Lambda') \in I_2$ be an element occuring in $v$. Then there exists $\Lambda'' \in \mathcal L^{(1)}_0$ such that $(\Lambda,\Lambda'') \in I_2$ occurs in $v$ and $d(\Lambda',\Lambda'') = 4$.
\end{lem}

\begin{proof}
Let us write $(\Lambda,\Lambda^j) \in I_2, 1\leq j \leq s$ for the various elements occuring in $v$ whose first component is $\Lambda$. Up to reordering the $\gamma_j$'s and swapping them with $\tau(\gamma_j)$ if necessary, we may assume that $(\Lambda,\Lambda^j) = \gamma_j$ for all $1\leq j \leq s$, and that $\Lambda^1 = \Lambda'$. The coordinate of $\varphi_1(v)$ along the basis vector $(\Lambda) \in I_1$ is equal to $-2\sum_{j=1}^{s} \lambda_j$. Since $\varphi_1(v) = 0$, this sum is zero. Since $\lambda_1 \not = 0$ by hypothesis, we have in particular $s\geq 2$. For all $2\leq j \leq s$, we have $2\leq d(\Lambda',\Lambda^j) \leq 4$ by the triangular inequality. Towards a contradiction, assume that $d(\Lambda',\Lambda^j) = 2$ for all $2\leq j \leq s$. In particular, $\delta_j := (\Lambda^j,\Lambda,\Lambda') \in I_3$ for all $2\leq j \leq s$. Consider the vector 
$$w := \frac{1}{3}\sum_{j=2}^s \sum_{\sigma \in \mathfrak S_6} \mathrm{sgn}(\sigma)\lambda_j\sigma(\delta_j) \in E_{1,\mathrm{alt}}^{-2,2(n-1)}.$$
Then we compute
\begin{align*}
\varphi_2(w) & = -\lambda_1((\Lambda,\Lambda') - (\Lambda',\Lambda)) - \sum_{j=2}^s \lambda_j((\Lambda,\Lambda^j) - (\Lambda^j,\Lambda)) + \sum_{j=2}^s \lambda_j((\Lambda',\Lambda^j) - (\Lambda^j,\Lambda'))\\
& = - \sum_{j=1}^s \lambda_j(\gamma_j - \tau(\gamma_j)) + \sum_{j=2}^s \lambda_j((\Lambda',\Lambda^j) - (\Lambda^j,\Lambda')).
\end{align*}
In particular, we get 
$$v + \varphi_2(w) = \sum_{j=2}^{s} \lambda_j((\Lambda^j,\Lambda') - (\Lambda',\Lambda^j)) + \sum_{j=s+1}^r \lambda_j(\gamma_j-\tau(\gamma_j)) \in \mathrm{Ker}(\varphi_1) \setminus \mathrm{Im}(\varphi_2),$$
which contradicts the minimality of $r$.
\end{proof}

From now on, let us assume that $n= 3$ or $4$, so that $\mathcal L_0$ is a tree. To conclude the proof of the proposition, let us pick $\Lambda = \Lambda^0 \in \mathcal L^{(1)}_0$ which occurs in $v$, say in a pair $(\Lambda,\Lambda') \in I_2$. Write $\Lambda^1 := \Lambda'$. By induction, we build a sequence $(\Lambda^k)_{k\geq 0}$ of lattices in $\mathcal L^{(1)}_0$ such that for all $k$, the pair $(\Lambda^{k},\Lambda^{k+1})$ occurs in $v$ and we have $d(\Lambda^0,\Lambda^{k}) = 2k$. It follows that the $\Lambda^k$'s are pairwise distinct, and it leads to a contradiction since only a finite number of such lattices can occur in $v$.\\
Let us assume that $\Lambda^0,\ldots ,\Lambda^k$ are already built for some $k\geq 1$. Since $(\Lambda^{k-1},\Lambda^k)$ occurs in $v$, so does $(\Lambda^k, \Lambda^{k-1})$. By the Lemma applied to latter pair, there exists $\Lambda^{k+1} \in \mathcal L_0^{(1)}$ such that the pair $(\Lambda^{k},\Lambda^{k+1}) \in I_2$ occurs in $v$ and $d(\Lambda^{k-1},\Lambda^{k+1}) = 4$. By the triangular inequality, we have 
\begin{align*}
d(\Lambda^0,\Lambda^{k+1}) & \leq d(\Lambda^0,\Lambda^{k}) + d(\Lambda^{k},\Lambda^{k+1}) = 2k+2 = 2(k+1),\\
d(\Lambda^0,\Lambda^{k+1}) & \geq |d(\Lambda^0,\Lambda^k) - d(\Lambda^k,\Lambda^{k+1})| = 2(k-1).
\end{align*}
Thus $d(\Lambda^0,\Lambda^{k+1}) = 2(k-1),2k$ or $2(k+1)$. We prove that it must be equal to the latter. 

Assume that $d(\Lambda^0,\Lambda^{k+1}) = 2(k-1)$. There exists a path $\Lambda^0 = L^0,\ldots , L^{2(k-1)} = \Lambda^{k+1}$. We obtain a cycle
\begin{center}
\begin{tikzcd}[column sep = 1em, row sep = 1em]
& \Lambda^0\cap\Lambda^1 \arrow[r,dash] & \Lambda^1 \arrow[r,dash] & \ldots \arrow[r,dash] & \Lambda^{k-1} \arrow[r,dash] & \Lambda^{k-1}\cap\Lambda^{k} \arrow[dr,dash] & \\
\Lambda^0 \arrow[dr,dash] \arrow[ur,dash] & & & & & & \Lambda^k \\
& L^1 \arrow[r,dash] & L^2 \arrow[r,dash] & \ldots \arrow[r,dash] & L^{2(k-1)} = \Lambda^{k+1} \arrow[r,dash] & \Lambda^{k}\cap \Lambda^{k+1} \arrow[ur,dash] & 
\end{tikzcd}
\end{center}
Since $\mathcal L_0$ is a tree, this cycle must be trivial, ie. the lower and upper paths, which are of the same length, are the same. In particular, we have $\Lambda^{k-1} = \Lambda^{k+1}$, which is absurd since $d(\Lambda^{k-1},\Lambda^{k+1}) = 4$. 

Assume that $d(\Lambda^0,\Lambda^{k+1}) = 2k$. There exists a path $\Lambda^0 = L_0,\ldots , L^{2k} = \Lambda^{k+1}$. We obtain a cycle
\begin{center}
\begin{tikzcd}[column sep = 1em, row sep = 1em]
& \Lambda^0\cap\Lambda^1 \arrow[r,dash] & \Lambda^1 \arrow[r,dash] & \ldots \arrow[r,dash] & \Lambda^{k-1} \cap \Lambda^k \arrow[r,dash] & \Lambda^{k} \arrow[dr,dash] & \\
\Lambda^0 \arrow[dr,dash] \arrow[ur,dash] & & & & & & \Lambda^k\cap\Lambda^{k+1} \\
& L^1 \arrow[r,dash] & L^2 \arrow[r,dash] & \ldots \arrow[r,dash] & L^{2k-1} \arrow[r,dash] & L^{2k} = \Lambda^{k+1} \arrow[ur,dash] & 
\end{tikzcd}
\end{center}
Since $\mathcal L_0$ is a tree, this cycle must be trivial, ie. the lower and upper paths, which are of the same length, are the same. In particular, we have $\Lambda^{k} = \Lambda^{k+1}$, which is absurd since $d(\Lambda^k,\Lambda^{k+1}) = 2$.

Thus, we have $d(\Lambda^0,\Lambda^{k+1}) = 2(k+1)$ so that $\Lambda^{k+1}$ meets all the requirements. It concludes the proof of Proposition \ref{AcyclicProof}.\\

In particular, we obtain the following statement. 

\begin{theo}\label{AlmostHighestCohomology}
Assume that $n=3$ or $4$. We have 
$$\mathrm H_c^{2(n-1)-1}(\mathcal M^{\mathrm{an}},\overline{\mathbb Q_{\ell}}) \simeq \mathrm{c-Ind}_{J_1}^{J} \, \rho_{\Delta_2},$$
with the rational Frobenius $\tau$ acting like multiplication by $-p^{2(n-1)-1}$.
\end{theo}

\section{The cohomology of the supersingular locus of the Shimura variety for $n = 3, 4$}

\subsection{The Hochschild-Serre spectral sequence induced by $p$-adic uniformization}

In this section, $n \geq 1$ is still any integer. We recover the notations of Section 3. Let $\xi: \mathbb G \rightarrow W_{\xi}$ be a finite-dimensional irreducible algebraic $\overline{\mathbb Q_{\ell}}$-representation of $\mathbb G$. Such representations have been classified in \cite{harris} Chapter III.2. We think of $\mathbb V_{\overline{\mathbb Q_{\ell}}} := \mathbb V \otimes \overline{\mathbb Q_{\ell}}$ as a representation of $\mathbb G$, whose dual is denoted by $\mathbb V_0$. Using the alternating form $\langle\cdot,\cdot\rangle$, we have an isomorphism $\mathbb V_0 \simeq \mathbb V_{\overline{\mathbb Q_{\ell}}} \otimes c^{-1}$, where $c$ is the multiplier character of $G$. Then, $W_{\xi}$ can be described as follows.

\begin{prop}\label{ClassificationAlgebraicRepresentations}
There exists unique integers $t(\xi), m(\xi) \geq 0$ and an idempotent $\epsilon(\xi) \in \mathrm{End}(\mathbb V_0^{\otimes m(\xi)})$ such that
$$W_{\xi} \simeq c^{t(\xi)}\otimes\epsilon(\xi)(\mathbb V_0^{\otimes m(\xi)}).$$
\end{prop}

The weight $w(\xi)$ is defined by 
$$w(\xi) := m(\xi) - 2t(\xi) \in \mathbb Z.$$
To any $\xi$ as above, we can associate a local system $\mathcal L_{\xi}$ which is defined on the tower $(\mathrm{S}_{K^p})_{K^p}$ of Shimura varieties. We denote by $\overline{\mathcal L_{\xi}}$ its restriction to the special fiber $\overline{\mathrm S}_{K^p}$. Let $A_{K^p}$ be the universal abelian scheme over $\mathrm{S}_{K^p}$. We write $\pi_{K^p}^m : A_{K^p}^m \to \mathrm{S}_{K^p}$ for the structure morphism of the $m$-fold product of $A_{K^p}$ with itself over $\mathrm{S}_{K^p}$. If $m=0$ it is just the identity on $\mathrm{S}_{K^p}$. According to \cite{harris} Chapter III.2, we have an isomorphism
$$\mathcal L_{\xi} \simeq \epsilon(\xi)\epsilon_{m(\xi)} \left( \mathrm{R}^{m(\xi)}({\pi_{K^p}^{m(\xi)}})_{*}\overline{\mathbb Q_{\ell}}(t(\xi))\right),$$
where $\epsilon_{m(\xi)}$ is some idempotent. In particular, if $\xi$ is the trivial representation of $\mathbb G$ then $\mathcal L_{\xi} = \overline{\mathbb Q_{\ell}}$.

We fix an irreducible algebraic representation $\xi: \mathbb G \rightarrow W_{\xi}$ as above. We associate the space $\mathcal A_{\xi}$ of \textbf{automorphic forms of $I$ of type $\xi$ at infinity}. Explicitly, it is given by 
$$\mathcal A_{\xi} = \left\{f: I(\mathbb A_{f})\rightarrow W_{\xi} \, \middle| \, \begin{array}{l}
f \text{ is } I(\mathbb A_{f}) \text{-smooth by right translations} \\
\text{and } \forall \gamma \in I(\mathbb Q), f(\gamma\,\cdot) = \xi(\gamma)f(\cdot) \end{array}\right\}.$$

\begin{notation}
Let $\mathrm{Sh}_{K_0K^p}^{\mathrm{an}} := (\mathrm S_{K^p} \otimes_{\mathbb Z_{p^2}} \mathbb Q_{p^2})^{\mathrm{an}}$ denote the analytification of the generic fiber of $\mathrm S_{K^p}$, on which the analytified local system $\mathcal L_{\xi}^{\mathrm{an}}$ is defined. Let $(\widehat{\mathrm S}_{K^p})^{\mathrm{ss},\mathrm{an}} \subset \mathrm{Sh}_{K_0K^p}^{\mathrm{an}}$ denote the analytical tube of the supersingular locus, or in other words the generic fiber of the formal scheme $(\widehat{\mathrm S}_{K^p})^{\mathrm{ss}}$. We write $\mathrm{H}^{\bullet}((\widehat{\mathrm S}_{K^p})^{\mathrm{ss},\mathrm{an}},\mathcal L_{\xi}^{\mathrm{an}})$ for the cohomology of $(\widehat{\mathrm S}_{K^p})^{\mathrm{ss},\mathrm{an}} \otimes \mathbb C_p$ with coefficients in $\mathcal L_{\xi}^{\mathrm{an}}$.
\end{notation}

In \cite{fargues} Théorème 4.5.12, Fargues builds a spectral sequence associated to the $p$-adic uniformization theorem in order to compute the cohomology of $(\widehat{\mathrm S}_{K^p})^{\mathrm{ss},\mathrm{an}}$.

\begin{theo}\label{FarguesSpectralSequence}
There is a $W$-equivariant spectral sequence 
$$F_2^{a,b}(K^p) : \mathrm{Ext}_{J}^a \left (\mathrm H_c^{2(n-1)-b}(\mathcal M^{\mathrm{an}}, \overline{\mathbb Q_{\ell}})(n-1), \mathcal A^{K^p}_{\xi}\right) \implies \mathrm H^{a+b}((\widehat{\mathrm S}_{K^p})^{\mathrm{ss},\mathrm{an}}, \mathcal L_{\xi}^{\mathrm{an}}).$$
These spectral sequences are compatible as the open compact subgroup $K^p$ varies in $\mathbb G(\mathbb A_f^p)$. 
\end{theo}

The $W$-action on $F_2^{a,b}(K^p)$ is inherited from the cohomology group $\mathrm H_c^{2(n-1)-b}(\mathcal M^{\mathrm{an}}, \overline{\mathbb Q_{\ell}})(n-1)$. By the compatibility with variation of the level $K^p$, we may take the limit and obtain a $\mathbb G(\mathbb A_f^p) \times W$-equivariant spectral sequence $F_2^{a,b} := \varinjlim_{K^p} F_2^{a,b}(K^p)$. Since $\thetamax$ is the semisimple rank of $J(\mathbb Q_p)$, the terms $F_2^{a,b}(K^p)$ are zero for $a > \thetamax$ according to \cite{fargues} Lemme 4.4.12. Therefore, the non-zero terms $F_2^{a,b}$ are located in the finite strip delimited by $0 \leq a \leq \thetamax$ and $0 \leq b \leq 2(n-1)$. Let us look at the abutment of the sequence. Since $\mathrm S_{K^p}$ is smooth, Berkovich's comparison theorem, cf \cite{berk2} Corollary 3.6, gives an isomorphism
$$\mathrm{H}^{a+b}(\overline{\mathrm S}_{K^p}^{\mathrm{ss}} \otimes \, \mathbb F, \overline{\mathcal L_{\xi}}) \simeq \mathrm H^{a+b}((\widehat{\mathrm S}_{K^p})^{\mathrm{ss},\mathrm{an}}, \mathcal L_{\xi}^{\mathrm{an}}).$$
Since $\overline{\mathrm S}_{K^p}^{\mathrm{ss}}$ has dimension $\thetamax$, the cohomology $\mathrm H^{\bullet}((\widehat{\mathrm S}_{K^p})^{\mathrm{ss},\mathrm{an}}, \mathcal L_{\xi}^{\mathrm{an}})$ is concentrated in degrees $0$ to $2\thetamax$.

Let $\mathcal A(I)$ denote the set of all automorphic representations of $I$ counted with multiplicities. We write $\widecheck{\xi}$ for the dual of $\xi$. We also define 
$$\mathcal A_{\xi}(I) := \{\Pi \in \mathcal A(I) \,|\, \Pi_{\infty} = \widecheck{\xi}\}.$$ 
According to \cite{fargues} Section 4.6, we have an identification 
$$\mathcal A_{\xi}^{K_p} \simeq \bigoplus_{\Pi\in\mathcal A_{\xi}(I)} \Pi_p \otimes (\Pi^p)^{K_p}.$$
It yields, for every $a$ and $b$, an isomorphism 
$$F_2^{a,b}(K^p) \simeq \bigoplus_{\Pi\in\mathcal A_{\xi}(I)} \mathrm{Ext}_{J}^a \left (\mathrm H_c^{2(n-1)-b}(\mathcal M^{\mathrm{an}}, \overline{\mathbb Q_{\ell}})(n-1), \Pi_p\right) \otimes (\Pi^p)^{K_p}.$$
Taking the limit over $K^p$, we deduce that 
$$F_2^{a,b} \simeq \bigoplus_{\Pi\in\mathcal A_{\xi}(I)} \mathrm{Ext}_{J}^a \left (\mathrm H_c^{2(n-1)-b}(\mathcal M^{\mathrm{an}}, \overline{\mathbb Q_{\ell}})(n-1), \Pi_p\right) \otimes \Pi^p.$$
The spectral sequence defined by the terms $F_2^{a,b}$ computes $\mathrm H^{a+b}((\widehat{\mathrm S})^{\mathrm{ss},\mathrm{an}}, \mathcal L_{\xi}^{\mathrm{an}}) := \varinjlim_{K^p} \mathrm H^{a+b}((\widehat{\mathrm S}_{K^p})^{\mathrm{ss},\mathrm{an}}, \mathcal L_{\xi}^{\mathrm{an}})$. It is isomorphic to $\mathrm{H}^{a+b}(\overline{\mathrm S}^{\mathrm{ss}} \otimes \, \mathbb F, \overline{\mathcal L_{\xi}}) := \varinjlim_{K^p} \mathrm{H}^{a+b}(\overline{\mathrm S}_{K^p}^{\mathrm{ss}} \otimes \, \mathbb F, \overline{\mathcal L_{\xi}})$.

Recall from Corollary \ref{SecondPage} that we have a decomposition 
$$\mathrm H_c^b(\mathcal M^{\mathrm{an}},\overline{\mathbb Q_{\ell}}) \simeq \bigoplus_{b \leq b' \leq 2(n-1)} E_2^{b-b',b'},$$
and $E_2^{b-b',b'}$ corresponds to the eigenspace of $\tau$ associated to the eigenvalue $(-p)^{b'}$. Accordingly, we have a decomposition 
$$F_2^{a,b} \simeq \bigoplus_{\substack{2(n-1)-b\, \leq \\
b' \leq \, 2(n-1)}} \, \bigoplus_{\Pi\in\mathcal A_{\xi}(I)} \mathrm{Ext}_{J}^a \left (E_2^{2(n-1)-b-b',b'}(n-1), \Pi_p\right) \otimes \Pi^p.$$
For $\Pi\in\mathcal A_{\xi}(I)$, we denote by $\omega_{\Pi}$ the central character. We define 
$$\delta_{\Pi_p} := \omega_{\Pi_p}(p^{-1}\cdot\mathrm{id})p^{-w(\xi)} \in \overline{\mathbb Q_{\ell}}^{\times}.$$
Let $\iota$ be any isomorphism $\overline{\mathbb Q_{\ell}} \simeq \mathbb C$, and write $|\cdot|_{\iota} := |\iota(\cdot)|$. The center of $I(\mathbb Q)$ is identified with $\mathbb E^{\times}$, and the element $p^{-1}\cdot\mathrm{id} \in \mathrm{Z}(J(\mathbb Q_p))$ is the image of $p^{-1} \in \mathbb E^{\times} \simeq \mathrm{Z}(I(\mathbb Q)) \hookrightarrow \mathrm Z(J(\mathbb Q_p))$. We have $\omega_{\Pi}(p^{-1}) = 1$. Moreover, for any finite place $q \not = p$, the element $p^{-1}$ lies inside the maximal compact subgroup of $\mathrm Z(I(\mathbb Q_q))$, so $|\omega_{\Pi_q}(p^{-1})|_{\iota} = 1$. Besides $\Pi_{\infty} = \widecheck{\xi}$, so we have
$$|\omega_{\Pi_p}(p^{-1}\cdot\mathrm{id})|_{\iota} = |\omega_{\widecheck{\xi}}(p^{-1})|_{\iota}^{-1} = |\omega_{\xi}(p^{-1})|_{\iota} = |p^{w(\xi)}|_{\iota} = p^{w(\xi)}.$$
The last equality comes from the isomorphism $W_{\xi} \simeq c^{t(\xi)}\otimes\epsilon(\xi)(\mathbb V_0^{\otimes m(\xi)})$, see Proposition \ref{ClassificationAlgebraicRepresentations}. In particular $|\delta_{\Pi_p}|_{\iota} = 1$ for any isomorphism $\iota$.

\begin{prop}\label{FrobeniusActionOnHom}
The $W$-action on $\mathrm{Ext}^a_{J} (E_2^{2(n-1)-b-b',b'}(n-1), \Pi_p)$ is trivial on the inertia $I$, and the Frobenius element $\mathrm{Frob}$ acts like multiplication by $(-1)^{-b'}\delta_{\Pi_p}p^{-b'+2(n-1)+w(\xi)}$.
\end{prop}

\begin{proof}
Let us write $X := E_2^{2(n-1)-b-b',b'}(n-1)$. By convention, the action of $\mathrm{Frob}$ on a space $\mathrm{Ext}^a_{J}(X,\Pi_p)$ is induced by functoriality of $\mathrm{Ext}$ applied to $\mathrm{Frob}^{-1}:X\rightarrow X$. Let us consider a projective resolution of $X$ in the category of smooth representations of $J(\mathbb Q_p)$ 
\begin{center}
\begin{tikzcd}
\ldots \arrow{r}{u_3} & P_2 \arrow{r}{u_2} & P_1 \arrow{r}{u_1} & P_0 \arrow{r}{u_0} & X \arrow{r} & 0.
\end{tikzcd}
\end{center}
Since $\mathrm{Frob}^{-1}$ commutes with the action of $J(\mathbb Q_p)$, we can choose a lift $\mathcal F = (\mathcal F_i)_{i\geq 0}$ of $\mathrm{Frob}^{-1}$ to a morphism of chain complexes. 
\begin{center}
\begin{tikzcd}
\ldots \arrow{r}{u_3} & P_2 \arrow{r}{u_2} \arrow{d}{\mathcal F_2} & P_1 \arrow{r}{u_1} \arrow{d}{\mathcal F_1} & P_0 \arrow{r}{u_0} \arrow{d}{\mathcal F_0} & X \arrow{r} \arrow{d}{\mathrm{Frob}^{-1}} & 0 \\
\ldots \arrow{r}{u_3} & P_2 \arrow{r}{u_2} & P_1 \arrow{r}{u_1} & P_0 \arrow{r}{u_0} & X \arrow{r} & 0
\end{tikzcd}
\end{center}
After applying $\mathrm{Hom}_J(\cdot,\Pi_p)$ and forgetting about the first term, we obtain a morphism $\mathcal F^{*}$ of chain complexes.
\begin{center}
\begin{tikzcd}
0 \arrow{r} & \mathrm{Hom}_J(P_0,\Pi_p) \arrow{d}{\mathcal F_0^{*}} \arrow{r} & \mathrm{Hom}_J(P_1,\Pi_p) \arrow{d}{\mathcal F_1^{*}} \arrow{r} & \mathrm{Hom}_J(P_2,\Pi_p) \arrow{d}{\mathcal F_2^{*}} \arrow{r} & \ldots \\
0 \arrow{r} & \mathrm{Hom}_J(P_0,\Pi_p) \arrow{r} & \mathrm{Hom}_J(P_1,\Pi_p) \arrow{r} & \mathrm{Hom}_J(P_2,\Pi_p) \arrow{r} & \ldots
\end{tikzcd}
\end{center}
Here $\mathcal F_i^{*}f(v) := f(\mathcal F_i(v))$. It induces morphisms on the cohomology 
$$\mathcal F_i^{*}: \mathrm{Ext}^i_{J}(X,\Pi_p) \rightarrow \mathrm{Ext}^i_J(X,\Pi_p),$$
which do not depend on the choice of the lift $\mathcal F$. Recall that $\mathrm{Frob}$ is the composition of $\varphi$ and $p\cdot\mathrm{id} \in J(\mathbb Q_p)$. Since $\varphi$ is multiplication by the scalar $(-1)^{b'}p^{b'-2(n-1)}$ on $X$, we may choose the lift $\mathcal F_i := (-1)^{b'}p^{-b'+2(n-1)} (p^{-1}\cdot\mathrm{id})$ for all $i$.\\
Consider an element of $\mathrm{Ext}^i_{J}(X,\Pi_p)$ represented by a morphism $f:P_i\to \Pi_p$. For any $v\in P_i$ we have
$$\mathcal F_i^{*}f(v) = f(\mathcal F_i(v)) = (-1)^{-b'}p^{-b'+2(n-1)}f((p^{-1}\cdot\mathrm{id})\cdot v) = (-1)^{-b'}p^{-b'+2(n-1)}\omega_{\Pi_p}(p^{-1}\cdot\mathrm{id})f(v).$$ 
It follows that $\mathrm{Frob}$ acts on $\mathrm{Ext}^i_{J}(X,\Pi_p)$ via multiplication by the scalar $(-1)^{-b'}\delta_{\Pi_p}p^{-b'+2(n-1)+w(\xi)}$.
\end{proof}

In general, the Hochschild-Serre spectral sequence has many differentials between non-zero terms. However, focusing on the diagonal defined by $a+b=0$, it is possible to compute $\mathrm{H}^{0}(\overline{\mathrm S}^{\mathrm{ss}} \otimes \, \mathbb F, \overline{\mathcal L_{\xi}})$. Recall that $X^{\mathrm{un}}(J)$ denotes the set of unramified characters of $J(\mathbb Q_p)$, ie. the characters which are trivial on $J^{\circ}$. If $x \in \overline{\mathbb Q_{\ell}}^{\times}$ is any non-zero scalar, we denote by $\overline{\mathbb Q_{\ell}}[x]$ the $1$-dimensional representation of $W$ where the inertia $I$ acts trivially and $\mathrm{Frob}$ acts like multiplication by $x$.

\begin{prop}\label{H0Shimura}
We have an isomorphism of $\mathbb G(\mathbb A_f^p)\times W$-representations
$$\mathrm{H}^{0}(\overline{\mathrm S}^{\mathrm{ss}} \otimes \, \mathbb F, \overline{\mathcal L_{\xi}}) \simeq \bigoplus_{\substack{\Pi\in\mathcal A_{\xi}(I) \\ \Pi_p \in X^{\mathrm{un}}(J)}} \Pi^p \otimes \overline{\mathbb Q_{\ell}}[\delta_{\Pi_p}p^{w(\xi)}].$$
\end{prop}

\begin{proof} 
The only non-zero term $F_2^{a,b}$ on the diagonal $a+b=0$ is $F_2^{0,0}$. Since there is no non-zero arrow pointing at nor coming from this term, it is untouched in all the successive pages of the sequence. Therefore we have an isomorphism 
$$F_2^{0,0} \simeq \mathrm{H}^{0}(\overline{\mathrm S}^{\mathrm{ss}} \otimes \, \mathbb F, \overline{\mathcal L_{\xi}}).$$
Using Proposition \ref{HighestDegreeCohomologyGroup}, we also have isomorphisms 
\begin{align*}
F_2^{0,0} & \simeq \bigoplus_{\Pi\in\mathcal A_{\xi}(I)} \mathrm{Hom}_{J} \left (\mathrm H_c^{2(n-1)}(\mathcal M^{\mathrm{an}}, \overline{\mathbb Q_{\ell}})(n-1), \Pi_p\right) \otimes \Pi^p \\
& \simeq  \bigoplus_{\Pi\in\mathcal A_{\xi}(I)} \mathrm{Hom}_{J} \left ((\mathrm{c-Ind}_{J^{\circ}}^J \, \mathbf 1)(n-1), \Pi_p\right) \otimes \Pi^p \\
& \simeq  \bigoplus_{\Pi\in\mathcal A_{\xi}(I)} \mathrm{Hom}_{J^{\circ}} \left (\mathbf 1(n-1), \Pi_p{}_{|J^{\circ}}\right) \otimes \Pi^p.
\end{align*}
Thus, only the automorphic representations $\Pi\in\mathcal A_{\xi}(I)$ with $\Pi_p^{J^{\circ}} \not = 0$ contribute to the sum. Consider such a $\Pi$. The irreducible representation $\Pi_p$ is generated by a $J^{\circ}$-invariant vector. Since $J^{\circ}$ is normal in $J(\mathbb Q_p)$, the whole representation $\Pi_p$ is trivial on $J^{\circ}$. Thus, it is an irreducible representation of $J/J^{\circ} \simeq \mathbb Z$. Therefore, it is an unramified character. Moreover the $W$-representation $V^0_{\Pi} := \mathrm{Hom}_{J^{\circ}} \left (\mathbf 1(n-1), \Pi_p\right)$ is $1$-dimensional and the Frobenius action was described in Proposition \ref{FrobeniusActionOnHom}.
\end{proof}

\subsection{The case $n = 3,4$}

In this section, we assume that $\thetamax=1$, ie. $n=3$ or $4$. Let $\xi$ be an irreducible finite dimensional algebraic representation of $\mathbb G$. The semisimple rank of $J(\mathbb Q_p)$ is $1$, therefore the terms $F_2^{a,b}$ are zero for $a>1$. In particular, the spectral sequence already degenerates on the second page. Since it computes the cohomology of the supersingular locus $\overline{\mathrm S}^{\mathrm{ss}}$ which is $1$-dimensional, we also have $F_2^{0,b} = 0$ for $b\geq 3$, and $F_2^{1,b} = 0$ for $b\geq 2$. In Figure 2, we draw the second page $F_2$ and we write between brackets the \textit{complex modulus} of the possible eigenvalues of $\mathrm{Frob}$ on each term under any isomorphism $\iota:\overline{\mathbb Q_{\ell}} \simeq \mathbb C$, as computed in Proposition \ref{FrobeniusActionOnHom}. 

\begin{rk}
The fact that no eigenvalue of complex modulus $p^{w(\xi)}$ appears in $F_2^{0,1}$ nor in $F_2^{1,1}$ follows from Proposition \ref{AcyclicProof}, where we proved that $E_2^{-1,2(n-1)} = 0$.
\end{rk}

\begin{figure}[h]
\centering
\begin{tikzcd}
F_2^{0,2}[p^{w(\xi)+2},p^{w(\xi)}] & 0 \\
F_2^{0,1}[p^{w(\xi)+1}] & F_2^{1,1}[p^{w(\xi)+1}]\\
F_2^{0,0}[p^{w(\xi)}] & F_2^{1,0}[p^{w(\xi)}]
\end{tikzcd}
\caption{The second page $F_2$ with the complex modulus of possible eigenvalues of $\mathrm{Frob}$ on each term.}
\end{figure}

\begin{prop}\label{EigenvaluesOnH2}
We have $F_2^{1,1} = 0$ and the eigenspaces of $\mathrm{Frob}$ on $F_2^{0,2}$ attached to any eigenvalue of complex modulus $p^{w(\xi)}$ are zero.
\end{prop}

\begin{proof}
By the machinery of spectral sequences, there is a $\mathbb G(\mathbb A_f^p)\times W$-subspace of $\mathrm H^2(\overline{\mathrm S}^{\mathrm{ss}} \otimes \, \mathbb F, \overline{\mathcal L_{\xi}})$ isomorphic to $F_2^{1,1}$, and the quotient by this subspace is isomorphic to $F_2^{0,2}$. We prove that all eigenvalues of $\mathrm{Frob}$ on $\mathrm H^2(\overline{\mathrm S}^{\mathrm{ss}} \otimes \, \mathbb F, \overline{\mathcal L_{\xi}})$ have complex modulus $p^{w(\xi)+2}$. The proposition then readily follows.\\
We need the Ekedahl-Oort stratification on the supersingular locus of the Shimura variety. Let $K^p \subset G(\mathbb A_f^p)$ be small enough. In \cite{vw2} Sections 3.3 and 6.3, the authors define the Ekedahl-Oort stratification on $\mathcal M_{\mathrm{red}}$ and on $\overline{\mathrm{S}}_{K^p}^{\mathrm{ss}}$ respectively, and they are compatible via the $p$-adic uniformization isomorphism. For $n=3$ or $4$, the stratification on the supersingular locus take the following form 
$$\overline{\mathrm{S}}_{K^p}^{\mathrm{ss}} = \overline{\mathrm{S}}_{K^p}^{\mathrm{ss}}[1] \sqcup \overline{\mathrm{S}}_{K^p}^{\mathrm{ss}}[3].$$
The stratum $\overline{\mathrm{S}}_{K^p}^{\mathrm{ss}}[1]$ is closed and $0$-dimensional, whereas the other stratum $\overline{\mathrm{S}}_{K^p}^{\mathrm{ss}}[3]$ is open, dense and $1$-dimensional. In particular, we have a Frobenius equivariant isomorphism between the cohomology groups of highest degree
$$\mathrm{H}^{2}(\overline{\mathrm S}_{K^p}^{\mathrm{ss}} \otimes \, \mathbb F, \overline{\mathcal L_{\xi}}) \simeq \mathrm{H}^{2}_c(\overline{\mathrm{S}}_{K^p}[3] \otimes \, \mathbb F, \overline{\mathcal L_{\xi}}).$$
According the \cite{vw2} Section 5.3, the closed Bruhat-Tits strata $\mathcal M_{\Lambda}$ and 
$\overline{\mathrm{S}}_{K^p,\Lambda,k}$ also admit an Ekedahl-Oort stratification of a similar form, and we have a decomposition 
$$\overline{\mathrm{S}}_{K^p}^{\mathrm{ss}}[3] = \bigsqcup_{\substack{1 \leq k \leq s \\ [\Lambda] \in \Gamma_k\backslash \mathcal L^{(1)}}} \overline{\mathrm{S}}_{K^p,\Lambda,k}[3],$$
into a finite disjoint union of open and closed subvarieties (we used the notations of Section 3). As a consequence, we have the following Frobenius equivariant isomorphisms
$$\mathrm{H}^{2}_c(\overline{\mathrm{S}}_{K^p}[3] \otimes \, \mathbb F, \overline{\mathcal L_{\xi}}) \simeq \bigoplus_{\substack{1 \leq k \leq s \\ [\Lambda] \in \Gamma_k\backslash \mathcal L^{(1)}}} \mathrm{H}^{2}_c(\overline{\mathrm{S}}_{K^p,\Lambda,k}[3] \otimes \, \mathbb F, \overline{\mathcal L_{\xi}}) \simeq \bigoplus_{\substack{1 \leq k \leq s \\ [\Lambda] \in \Gamma_k\backslash \mathcal L^{(1)}}} \mathrm{H}^{2}(\overline{\mathrm{S}}_{K^p,\Lambda,k} \otimes \, \mathbb F, \overline{\mathcal L_{\xi}})$$
where the last isomorphism follows from the stratification on the closed Bruhat-Tits strata $\overline{\mathrm{S}}_{K^p,\Lambda,k}$. Now, recall that the local system $\mathcal L_{\xi}$ is given by
$$\mathcal L_{\xi} \simeq \epsilon(\xi)\epsilon_{m(\xi)} \left( \mathrm{R}^{m(\xi)}({\pi_{K^p}^{m(\xi)}})_{*}\overline{\mathbb Q_{\ell}}(t(\xi))\right).$$
It implies that $\overline{\mathcal L_{\xi}}$ is pure of weight $w(\xi)$. Since the variety $\overline{\mathrm{S}}_{K^p,\Lambda,k}$ is smooth and projective, it follows that all the eigenvalues of $\mathrm{Frob}$ on the cohomology group $\mathrm{H}^{2}(\overline{\mathrm{S}}_{K^p,\Lambda,k} \otimes \, \mathbb F, \overline{\mathcal L_{\xi}})$ have complex modulus $p^{w(\xi)+2}$ under any isomorphism $\iota:\overline{\mathbb Q_{\ell}} \simeq \mathbb C$. The result follows by taking the limit over $K^p$.
\end{proof}

In this paragraph, let us compute the term 
\begin{align*}
F_2^{1,0} & \simeq \bigoplus_{\Pi\in\mathcal A_{\xi}(I)} \mathrm{Ext}^1_{J} \left (\mathrm H_c^{2(n-1)}(\mathcal M^{\mathrm{an}},\overline{\mathbb Q_{\ell}})(n-1), \Pi_p\right) \otimes \Pi^p \\
& \simeq 
\bigoplus_{\Pi\in\mathcal A_{\xi}(I)} \mathrm{Ext}^1_{J} \left (\mathrm{c-Ind}_{J^{\circ}}^J\,\mathbf 1(n-1), \Pi_p\right) \otimes \Pi^p.
\end{align*}
Let $\mathrm{St}_J$ denote the Steinberg representation of $J(\mathbb Q_p)$. 

\begin{prop}\label{ComputationExt1}
Let $\pi$ be an irreducible smooth representation of $J(\mathbb Q_p)$. Then 
$$\mathrm{Ext}_{J}^1(\mathrm{c-Ind}_{J^{\circ}}^J\,\mathbf 1,\pi) = 
\begin{cases}
\overline{\mathbb Q_{\ell}} & \text{if } \exists \chi \in X^{\mathrm{un}}(J), \pi \simeq \chi\cdot\mathrm{St}_J,\\
0 & \text{otherwise.}
\end{cases}$$
\end{prop}

In order to prove this proposition, we need a few general facts about restriction of smooth representations to normal subgroups. Let $G$ be a locally profinite group and let $H$ be a closed normal subgroup. If $(\sigma,W)$ is a representation of $H$, for $g\in G$ we define the representation $(\sigma^g,W)$ by $\sigma^{g}:h\mapsto \sigma(g^{-1}hg)$. The representation $\sigma$ is irreducible if and only if $\sigma^g$ is for any (or for all) $g\in G$. 

\begin{lem}\label{UsefulLemma}
Assume that $\mathrm Z(G)H$ has finite index in $G$. 
\begin{enumerate}[label=\upshape (\arabic*)]
\item Let $\pi$ be a smooth irreducible admissible representation of $G$. There exists a smooth irreducible representation $\sigma$ of $H$, an integer $r\geq 1$ and $g_1,\ldots,g_r \in G$ such that 
$$\pi_{|H} \simeq \sigma^{g_1} \oplus \ldots \oplus \sigma^{g_r}.$$
Moreover $r\leq [\mathrm Z(G)H:G]$, and for any $g\in G$ there exists some $1 \leq i \leq r$ such that $\sigma^g \simeq \sigma^{g_i}$.
\item Assume furthermore that $G/H$ is abelian. Let $\pi_1$ and $\pi_2$ be two smooth admissible irreducible representations of $G$. The three following statements are equivalent.
\begin{enumerate}[label={--},noitemsep,topsep=0pt]
\item $(\pi_1)_{|H} \simeq (\pi_2)_{|H}$.
\item There exists a smooth character $\chi$ of $G$ which is trivial on $H$ such that $\pi_2 \simeq \chi\cdot\pi_1$. 
\item $\mathrm{Hom}_H(\pi_1,\pi_2) \not = 0$. 
\end{enumerate}
\item Assume that $G/H$ is abelian and that $[\mathrm Z(G)H:G] = 2$. Let $g_0 \in G \setminus \mathrm Z(G)H$ and let $\pi$ be a smooth admissible irreducible representation of $G$. If there exists an irreducible representation $\sigma$ of $H$ such that $\pi_{|H} \simeq \sigma \oplus \sigma^{g_0}$, then $\sigma \not \simeq \sigma^{g_0}$.
\end{enumerate}
\end{lem}

\begin{proof}
For (1) and (2), we refer to \cite{Renard} VI.3.2 Proposition. The result there is stated in the context of a $p$-adic group $G$ with normal subgroup $H = {}^0G$ such that $G/{}^0G \simeq \mathbb Z^d$ for some $d\geq 0$, but the same arguments work as verbatim in the generality of the lemma. Admissibility of the representations involved is assumed only in order to apply Schur's lemma, insuring for instance the existence of central characters of smooth irreducible representations. In particular, if $G/K$ is at most countable for any open compact subgroup $K$ of $G$, then it is not necessary to assume admissibility.\\
Let us prove (3). Assume towards a contradiction that $\pi_{|H} \simeq \sigma \oplus \sigma^{g_0}$ and that $\sigma \simeq \sigma^{g_0}$. We build a smooth admissible irreducible representation $\Pi$ of $G$ such that $\Pi_{|H} = \sigma$, which results in a contradiction in regards to (2) since $\mathrm{Hom}_H(\Pi,\pi) \not = 0$ but $\Pi_{|H} \not \simeq \pi_{|H}$. Let $\chi$ be the central character of $\pi$. Then $\chi_{|\mathrm Z(G) \cap H}$ coincides with the central character of $\sigma$. Let $W$ denote the underlying vector space of $\sigma$. By hypothesis, there exists a linear automorphism $f:W\rightarrow W$ such that for every $h \in H$ and $w\in W$, 
$$f(\sigma(g_0^{-1}hg_0)\cdot w) = \sigma(h)\circ f(w).$$
Let us write $g_0^2 = z_0h_0$ for some $z_0 \in \mathrm Z(G)$ and $h_0 \in H$. We define $\varphi := f^2\circ \sigma(h_0)^{-1}$. Then for all $h\in H$ and $w\in W$, we have 
\begin{align*}
\varphi(\sigma(h)\cdot w) = f^2(\sigma(h_0^{-1}h)\cdot w) & = f^2(\sigma(h_0^{-1}hh_0)\sigma(h_0^{-1})\cdot w) \\
& = f^2(\sigma(g_0^{-2}hg_0^2)\sigma(h_0^{-1})\cdot w)\\
& = \sigma(h) \circ f^2(\sigma(h_0)^{-1}\cdot w) \\
& = \sigma(h) \circ \varphi(w).
\end{align*}
Thus $\varphi:\sigma\xrightarrow{\sim} \sigma$. By Schur's lemma we have $\varphi = \lambda\cdot\mathrm{id}$ for some $\lambda \in \overline{\mathbb Q_{\ell}}$. Up to replacing $f$ by $(\chi(z_0)\lambda^{-1})^{1/2}f$, we may assume that $\varphi = \chi(z_0)\cdot\mathrm{id}$, ie. $f^2 = \chi(z_0)\sigma(h_0)$. \\
We build a $G$-representation $\Pi$ on $W$ which extends $\sigma$. Let $g \in G$ and define 
$$\Pi(g) = 
\begin{cases}
\chi(z)\sigma(h) & \text{if } g = zh \in Z(G)H,\\
\chi(z) f\circ \sigma(h) & \text{if } g = g_0zh \in g_0Z(G)H.
\end{cases}$$
Then one may check that $\Pi$ is a well defined group morphism $G \rightarrow \mathrm{GL}(W)$. The fact that it is smooth irreducible and admissible follows from $\Pi_{|H} \simeq \sigma$ by construction, and it concludes the proof.
\end{proof}


We may now move on to the proof of Proposition \ref{ComputationExt1}. 

\begin{proof}
First, let us compute $\mathrm{Ext}_{J^{\circ}}^1(\mathbf 1,\sigma)$ for any irreducible representation $\sigma$ of $J^{\circ}$ with trivial central character. Let $J^1 = \mathrm U(\mathbf V)$ denote the unitary group of $\mathbf V$ (recall that $J = \mathrm{GU}(\mathbf V)$ is the group of unitary similitudes). Then $J^1(\mathbb Q_p)$ is a normal subgroup both of $J^{\circ}$ and of $J(\mathbb Q_p)$. Moreover, $J^{\circ}/J^1(\mathbb Q_p)$ is isomorphic to the image of the multiplier $c_{|J^{\circ}}:J^{\circ} \rightarrow \mathbb Z_p^{\times}$, in particular it is compact and abelian. Thus, we have 
$$\mathrm{Ext}_{J^{\circ}}^1(\mathbf 1,\sigma) \simeq \mathrm{Ext}_{J^1}^1(\mathbf 1,\sigma_{|J^1(\mathbb Q_p)})^{J^{\circ}/J^1(\mathbb Q_p)}.$$
Since $\sigma$ has trivial central character, the $J^{\circ}$-action on $\mathrm{Ext}_{J^1}^1(\mathbf 1,\sigma_{|J^1(\mathbb Q_p)})$ is actually trivial on $\mathrm Z(J^{\circ})J^1(\mathbb Q_p)$. Since $\mathbb Q_{p^2}/\mathbb Q_p$ is unramified, we actually have $\mathrm Z(J^{\circ})J^1(\mathbb Q_p) = J^{\circ}$. Hence, $J^{\circ}$ acts trivially on $\mathrm{Ext}_{J^1}^1(\mathbf 1,\sigma_{|J^1(\mathbb Q_p)})$.\\
Since $J^1$ is an algebraic group, we may use Theorem 2 of \cite{noriprasad}, a generalization of a duality theorem of Schneider and Stühler, to finish the computation. Namely, we have
$$\mathrm{Ext}_{J^1}^1(\mathbf 1,\sigma_{|J^1(\mathbb Q_p)}) \simeq \mathrm{Hom}_{J^1}(\sigma_{|J^1(\mathbb Q_p)},D(\mathbf 1))^{\vee},$$
where $D$ denotes the Aubert-Zelevinsky involution in $J^1(\mathbb Q_p)$. We note that $D(\mathbf 1) = \mathrm{St}_{J^1}$ is the Steinberg representation of $J^1(\mathbb Q_p)$. Let $\mathrm{St}_{J^{\circ}}$ denote the representation of $J^{\circ} = \mathrm Z(J^{\circ})J^1(\mathbb Q_p)$ obtained by letting the center act trivially on $\mathrm{St}_{J^1}$. We have proved that for any irreducible representation $\sigma$ of $J^{\circ}$ with trivial central character, we have
$$\mathrm{Ext}_{J^{\circ}}^1(\mathbf 1,\sigma) \simeq \mathrm{Hom}_{J^1}(\sigma_{|J^1(\mathbb Q_p)},\mathrm{St}_{J^1})^{\vee} \simeq 
\begin{cases}
\overline{\mathbb Q_{\ell}} & \text{if } \sigma \simeq \mathrm{St}_{J^{\circ}}, \\
0 & \text{else}.
\end{cases}$$
Now, let $\pi$ be an irreducible representation of $J(\mathbb Q_p)$. By Frobenius reciprocity we have 
$$\mathrm{Ext}_{J}^1(\mathrm{c-Ind}_{J^{\circ}}^J\,\mathbf 1,\pi) \simeq \mathrm{Ext}_{J^{\circ}}^1(\mathbf 1,\pi_{|J^{\circ}}).$$
By functoriality of $\mathrm{Ext}$, we have $\mathrm{Ext}_{J^{\circ}}^1(\mathbf 1,\pi_{|J^{\circ}}) = 0$ if the central character of $\pi$ is not unramified. Thus, let us now assume that the central character is unramified. By the above, $\mathrm{Ext}_{J}^1(\mathrm{c-Ind}_{J^{\circ}}^J\,\mathbf 1,\pi)$ is non zero if and only if $\pi_{|J^{\circ}}$ contains $\mathrm{St}_{J^{\circ}}$. Besides, as will be proved in Lemma \ref{RestrictionSteinberg}, we have $(\mathrm{St}_J)_{|J^{\circ}} = \mathrm{St}_{J^{\circ}}$. Thus, Lemma \ref{UsefulLemma} (2) implies that $\pi_{|J^{\circ}}$ contains $\mathrm{St}_{J^{\circ}}$ if and only if $\pi \simeq \chi\cdot\mathrm{St}_J$ for some unramified character $\chi \in X^{\mathrm{un}}(J)$. Since $\mathrm{Ext}_{J}^1(\mathrm{c-Ind}_{J^{\circ}}^J\,\mathbf 1,\chi\cdot\mathrm{St}_J) \simeq \overline{\mathbb Q_{\ell}}$, we are done.
\end{proof}

\begin{lem}\label{RestrictionSteinberg}
We have $(\mathrm{St}_{J})_{|J^{\circ}} \simeq \mathrm{St}_{J^{\circ}}$.
\end{lem}

\begin{proof}
Since the Steinberg representation $\mathrm{St}_J$ has trivial central character, it is enough to prove that $(\mathrm{St}_{J})_{|J^1(\mathbb Q_p)} \simeq \mathrm{St}_{J^1}$. The Steinberg representation $\mathrm{St}_J$ (resp. $\mathrm{St}_{J^1}$) can be characterized as the unique irreducible representation $\rho$ of $J(\mathbb Q_p)$ (resp. of $J^1(\mathbb Q_p)$) such that $\mathrm{Ext}^2_{J}(\mathbf 1,\rho) \not = 0$ (resp. $\mathrm{Ext}^1_{J^1}(\mathbf 1,\rho) \not = 0$). The gap between the degrees of the $\mathrm{Ext}$ groups for $J(\mathbb Q_p)$ and for $J^1(\mathbb Q_p)$ is explained by the non-compactness of the center of $J(\mathbb Q_p)$. By \cite{noriprasad} Proposition 3.4 we have 
$$\mathrm{Ext}^2_J(\mathbf 1,\mathrm{St}_J) \simeq \mathrm{Ext}^1_{J,\mathbf 1}(\mathbf 1,\mathrm{St}_J) \oplus \mathrm{Ext}^2_{J,\mathbf 1}(\mathbf 1,\mathrm{St}_J),$$
where the $\mathrm{Ext}$ groups on the right-hand side are taken in the category of smooth representations of $J(\mathbb Q_p)$ on which the center acts trivially. Equivalently, this is the category of smooth representations of $J(\mathbb Q_p)/\mathrm Z(J(\mathbb Q_p))$. Consider the normal subgroup $\mathrm Z(J(\mathbb Q_p))J^1(\mathbb Q_p)/\mathrm Z(J(\mathbb Q_p)) \simeq J^1(\mathbb Q_p)/\mathrm Z(J(\mathbb Q_p))\cap J^1(\mathbb Q_p) = J^1(\mathbb Q_p) / \mathrm Z(J^1(\mathbb Q_p))$. The quotient group is isomorphic to $J(\mathbb Q_p)/\mathrm Z(J(\mathbb Q_p))J^1(\mathbb Q_p)$, which is trivial if $n$ is odd and $\mathbb Z/2\mathbb Z$ is $n$ is even. Thus, we have 
\begin{align*}
\mathrm{Ext}_{J,\mathbf 1}^{\bullet}(\mathbf 1,\mathrm{St}_J) & \simeq \mathrm{Ext}_{J/\mathrm Z(J)}^{\bullet}(\mathbf 1,\mathrm{St}_J) \\
& \simeq \mathrm{Ext}_{J^1/\mathrm Z(J^1)}^{\bullet}(\mathbf 1,(\mathrm{St}_J)_{|J^1(\mathbb Q_p)})^{J(\mathbb Q_p)/\mathrm Z(J(\mathbb Q_p))J^1(\mathbb Q_p)} \\
& \simeq \mathrm{Ext}_{J^1,\mathbf 1}^{\bullet}(\mathbf 1,(\mathrm{St}_J)_{|J^1(\mathbb Q_p)})^{J(\mathbb Q_p)/\mathrm Z(J(\mathbb Q_p))J^1(\mathbb Q_p)} \\
& \simeq \mathrm{Ext}_{J^1}^{\bullet}(\mathbf 1,(\mathrm{St}_J)_{|J^1(\mathbb Q_p)})^{J(\mathbb Q_p)/\mathrm Z(J(\mathbb Q_p))J^1(\mathbb Q_p)},
\end{align*}
the last line following from the same Proposition 3.4 as above, but applied to $J^1(\mathbb Q_p)$. In \cite{fargues} Lemme 4.4.12, it is explained that $\mathrm{Ext}^i_{J^1}(\pi_1,\pi_2)$ vanishes for any smooth representations $\pi_1,\pi_2$ of $J^1(\mathbb Q_p)$ as soon as $i$ is greater than the semisimple rank of $J(\mathbb Q_p)$, that is $1$ in our case. Hence, $\mathrm{Ext}^2_{J,\mathbf 1}(\mathbf 1,\mathrm{St}_J) = 0$ and we have 
$$\mathrm{Ext}^2_J(\mathbf 1,\mathrm{St}_J) \simeq \mathrm{Ext}^1_{J,\mathbf 1}(\mathbf 1,\mathrm{St}_J) \simeq \mathrm{Ext}_{J^1}^1(\mathbf 1,(\mathrm{St}_J)_{|J^1(\mathbb Q_p)})^{J(\mathbb Q_p)/\mathrm Z(J(\mathbb Q_p))J^1(\mathbb Q_p)}.$$
In particular, the right-hand side is non zero, which proves that $(\mathrm{St}_J)_{|J^1(\mathbb Q_p)}$ contains $\mathrm{St}_{J^1}$. It remains that to justify that $(\mathrm{St}_J)_{|J^1(\mathbb Q_p)}$ is irreducible. If $n$ is odd so that $\mathrm Z(J(\mathbb Q_p))J^1(\mathbb Q_p) = J(\mathbb Q_p)$, it is automatic. If $n$ is even, in virtue of point (3) of Lemma \ref{UsefulLemma}, it remains to justify that for any $g\in J(\mathbb Q_p)$ we have $\mathrm{St}_{J^1}^g \simeq \mathrm{St}_{J^1}$. This follows from the following computation 
$$\mathrm{Ext}^1_{J^1}(\mathbf 1,\mathrm{St}_{J^1}^g) = \mathrm{Ext}^1_{J^1}(\mathbf 1^{g^{-1}},\mathrm{St}_{J^1}) = \mathrm{Ext}^1_{J^1}(\mathbf 1,\mathrm{St}_{J^1}) \not = 0.$$
\end{proof}

We may now compute the cohomology of the supersingular locus. Recall the supercuspidal representation $\tau_1$ of the Levi complement $M_1 \subset J(\mathbb Q_p)$ that we defined in Section 4.2. When $n=3$ or $4$, we actually have $M_1 = J(\mathbb Q_p)$ and 
$$\tau_1 = \mathrm{c-Ind}_{\mathrm{N}_J(J_1)}^{J} \, \widetilde{\rho_{\Delta_2}}$$
is a supercuspidal representation of $J(\mathbb Q_p)$, where $\mathrm{N}_J(J_1) = \mathrm{Z}(J(\mathbb Q_p))J_1$ (see Proposition \ref{Normalizers}) and $\widetilde{\rho_{\Delta_2}}$ is the inflation of $\rho_{\Delta_2}$ to $\mathrm{N}_J(J_1)$.

\begin{theo}\label{CohomologyBasicStratum}
Assume that $n=3$ or $4$. There are $\mathbb G(\mathbb A_f^p) \times W$-equivariant isomorphisms
\begin{align*}
\mathrm{H}^{0}(\overline{\mathrm S}^{\mathrm{ss}} \otimes \, \mathbb F, \overline{\mathcal L_{\xi}}) & \simeq \bigoplus_{\substack{\Pi\in\mathcal A_{\xi}(I) \\ \Pi_p \in X^{\mathrm{un}}(J)}} \Pi^p \otimes \overline{\mathbb Q_{\ell}}[\delta_{\Pi_p}p^{w(\xi)}], \\
\mathrm{H}^{1}(\overline{\mathrm S}^{\mathrm{ss}} \otimes \, \mathbb F, \overline{\mathcal L_{\xi}}) & \simeq \bigoplus_{\substack{\Pi\in\mathcal A_{\xi}(I) \\ \exists \chi \in X^{\mathrm{un}}(J),\\ \Pi_p = \chi\cdot\mathrm{St}_J}} \Pi^p \otimes \overline{\mathbb Q_{\ell}}[\delta_{\Pi_p}p^{w(\xi)}] \oplus \bigoplus_{\substack{\Pi\in\mathcal A_{\xi}(I) \\ \exists \chi \in X^{\mathrm{un}}(J),\\ \Pi_p = \chi\cdot\tau_1}} \Pi^p \otimes \overline{\mathbb Q_{\ell}}[-\delta_{\Pi_p}p^{w(\xi)+1}],\\
\mathrm{H}^{2}(\overline{\mathrm S}^{\mathrm{ss}} \otimes \, \mathbb F, \overline{\mathcal L_{\xi}}) & \simeq \bigoplus_{\substack{\Pi\in\mathcal A_{\xi}(I) \\ \Pi_p^{J_1}\not = 0}} \Pi^p \otimes \overline{\mathbb Q_{\ell}}[\delta_{\Pi_p}p^{w(\xi)+2}].
\end{align*}
\end{theo}

\begin{proof}
The statement regarding $\mathrm{H}^{0}(\overline{\mathrm S}^{\mathrm{ss}} \otimes \, \mathbb F, \overline{\mathcal L_{\xi}})$ was already proved in Proposition \ref{H0Shimura}. Let us prove the statement regarding $\mathrm{H}^{2}(\overline{\mathrm S}^{\mathrm{ss}} \otimes \, \mathbb F, \overline{\mathcal L_{\xi}})$ first. By Proposition \ref{EigenvaluesOnH2}, we have 
$$\mathrm{H}^{2}(\overline{\mathrm S}^{\mathrm{ss}} \otimes \, \mathbb F, \overline{\mathcal L_{\xi}}) \simeq F_2^{0,2} \simeq \bigoplus_{\Pi\in\mathcal A_{\xi}(I)} \mathrm{Hom}_{J} \left (E_2^{0,2(n-2)}(n-1), \Pi_p\right) \otimes \Pi^p.$$
The term $E_2^{0,2(n-2)}$ is isomorphic to $\mathrm{c-Ind}_{J_1}^{J}\,\mathbf 1$. Therefore, by Frobenius reciprocity we have 
$$\mathrm{Hom}_{J} \left (E_2^{0,b}(n-1), \Pi_p\right) \simeq \mathrm{Hom}_{J_1} \left (\mathbf 1(n-1), \Pi_p\right).$$
Hence, only the automorphic representations $\Pi \in \mathcal A_{\xi}(I)$ with $\Pi_p^{J_1} \not = 0$ contribute to $F_2^{0,2}$. Such a representation $\Pi_p$ is said to be \textbf{$J_1$-spherical}. Since $J_1$ is a special maximal compact subgroup of $J(\mathbb Q_p)$, according to \cite{minguez} 2.1, we have $\dim(\pi^{J_1}) = 1$ for every smooth irreducible $J_1$-spherical representation $\pi$ of $J(\mathbb Q_p)$. The result follows using Proposition \ref{FrobeniusActionOnHom} to describe the eigenvalues of $\mathrm{Frob}$.\\

We now prove the statement regarding $\mathrm{H}^{1}(\overline{\mathrm S}^{\mathrm{ss}} \otimes \, \mathbb F, \overline{\mathcal L_{\xi}})$. By the Hochschild-Serre spectral sequence, there exists a $G(\mathbb A_f^p)\times W$-subspace $V'$ of this cohomology group such that 
$$V' \simeq F_2^{1,0} \quad \text{and} \quad \mathrm{H}^{1}(\overline{\mathrm S}^{\mathrm{ss}} \otimes \, \mathbb F, \overline{\mathcal L_{\xi}})/V' \simeq F_2^{0,1}.$$
We have 
\begin{align*}
F_2^{1,0} & \simeq \bigoplus_{\Pi\in\mathcal A_{\xi}(I)} \mathrm{Ext}^1_{J} \left (\mathrm H_c^{2(n-1)}(\mathcal M^{\mathrm{an}},\overline{\mathbb Q_{\ell}})(n-1), \Pi_p\right) \otimes \Pi^p \\
& \simeq \bigoplus_{\Pi\in\mathcal A_{\xi}(I)} \mathrm{Ext}^1_{J} \left (\mathrm{c-Ind}_{J^{\circ}}^J\,\mathbf 1(n-1), \Pi_p\right) \otimes \Pi^p\\
& \simeq \bigoplus_{\substack{\Pi\in\mathcal A_{\xi}(I) \\ \exists \chi \in X^{\mathrm{un}}(J),\\ \Pi_p = \chi\cdot\mathrm{St}_J}} \Pi^p \otimes \overline{\mathbb Q_{\ell}}[\delta_{\Pi_p}p^{w(\xi)}],
\end{align*}
according to Proposition \ref{ComputationExt1}, and with the eigenvalues of $\mathrm{Frob}$ being given by Proposition \ref{FrobeniusActionOnHom}. On the other hand, we have 
$$F_2^{0,1} \simeq \bigoplus_{\Pi\in\mathcal A_{\xi}(I)} \mathrm{Hom}_{J} \left (E_2^{0,2(n-1)-1}(n-1), \Pi_p\right) \otimes \Pi^p.$$
By Proposition \ref{FrobeniusActionOnHom}, $\mathrm{Frob}$ acts on a summand of $F_2^{0,1}$ by the scalar $-\delta_{\Pi_p}p^{w(\xi)+1}$. Since $\mathrm{Frob}_{|V'}$ has no eigenvalue of complex modulus $p^{w(\xi)+1}$, the quotient actually splits so that $F_2^{0,1}$ is naturally a subspace of $\mathrm{H}^{1}(\overline{\mathrm S}^{\mathrm{ss}} \otimes \, \mathbb F, \overline{\mathcal L_{\xi}})$. It remains to compute it. We have 
$$E_2^{0,2(n-1)-1} \simeq \mathrm{c-Ind}_{J_1}^{J}\, \rho_{\Delta_2}.$$
Hence, we have an isomorphism
\begin{align*}
F_2^{0,1} & \simeq \bigoplus_{\Pi\in\mathcal A_{\xi}(I)} \mathrm{Hom}_{J} \left (\mathrm{c-Ind}_{J_1}^{J}\,\rho_{\Delta_2}(n-1), \Pi_p\right) \otimes \Pi^p\\
& \simeq \bigoplus_{\Pi\in\mathcal A_{\xi}(I)} \mathrm{Hom}_{J_1} \left (\rho_{\Delta_2}(n-1), \Pi_p{}_{|J_1}\right) \otimes \Pi^p.
\end{align*}
It follows that only the automorphic representations $\Pi\in\mathcal A_{\xi}(I)$ whose $p$-component $\Pi_p$ contains the supercuspidal representation $\rho_{\Delta_2}$ when restricted to $J_1$, contribute to the sum. According to Proposition \ref{Depth-0Types}, such $\Pi_p$ are precisely those of the form $\chi\cdot\tau_1$ for some $\chi \in X^{\mathrm{un}}(J)$. By the Mackey formula we have 
\begin{align*}
\mathrm{Hom}_{J} \left (\mathrm{c-Ind}_{J_1}^{J}\,\rho_{\Delta_2}, \chi\cdot\tau_1\right)& \simeq \mathrm{Hom}_{J_1} \left (\rho_{\Delta_2}, \tau_1{}_{|J_1}\right)\\
& \simeq \mathrm{Hom}_{J_1} \left (\rho_{\Delta_2}, (\mathrm{c-Ind}_{\mathrm{N}_J(J_1)}^{J} \, \widetilde{\rho_{\Delta_2}})_{|J_1}\right) \\
& \simeq \bigoplus_{h\in J_1 \backslash J(\mathbb Q_p) / \mathrm{N}_J(J_1)} \mathrm{Hom}_{J_1\cap {}^h\mathrm{N}_J(J_1)}(\rho_{\Delta_2},{}^h\widetilde{\rho_{\Delta_2}}),
\end{align*}
where in the last formula we omitted to write the restrictions to $J_1\cap {}^h\mathrm{N}_J(J_1)$. We used the fact that $\chi_{|J_1}$ is trivial. Since $\widetilde{\rho_{\Delta_2}}$ is just the inflation of $\rho_{\Delta_2}$ from $J_1$ to $\mathrm{N}_J(J_1) = \mathrm{Z}(J(\mathbb Q_p))J_1$, we have a bijection
$$\mathrm{Hom}_{J_1\cap {}^h\mathrm{N}_J(J_1)}(\rho_{\Delta_2},{}^h\widetilde{\rho_{\Delta_2}}) \simeq \mathrm{Hom}_{\mathrm{N}_J(J_1)\cap {}^h\mathrm{N}_J(J_1)}(\widetilde{\rho_{\Delta_2}},{}^h\widetilde{\rho_{\Delta_2}}).$$
Now, $\mathrm{N}_J(J_1)$ contains the center, is compact modulo the center, and $\tau_1 = \mathrm{c-Ind}_{\mathrm{N}_J(J_1)}^{J}\, \widetilde{\rho_{\Delta_2}}$ is supercuspidal. It follows that an element $h\in J(\mathbb Q_p)$ intertwines $\widetilde{\rho_{\Delta_2}}$ if and only if $h\in \mathrm{N}_J(J_1)$ (see for instance \cite{bushnellbook} 11.4 Theorem along with Remarks 1 and 2). Therefore, only the trivial double coset contributes to the sum and we have 
$$\mathrm{Hom}_{J} \left (\mathrm{c-Ind}_{J_1}^{J}\,\rho_{\Delta_2}, \chi\cdot\tau_1\right) \simeq \mathrm{Hom}_{J_1}(\rho_{\Delta_2},\rho_{\Delta_2}) \simeq \overline{\mathbb Q_{\ell}}.$$
To sum up, we have 
$$F_2^{0,1} \simeq \bigoplus_{\substack{\Pi\in\mathcal A_{\xi}(I) \\ \exists \chi \in X^{\mathrm{un}}(J),\\ \Pi_p = \chi\cdot\tau_1}} \Pi^p \otimes \overline{\mathbb Q_{\ell}}[-\delta_{\Pi_p}p^{w(\xi)+1}].$$
It concludes the proof.

\end{proof}

\newpage

\printbibliography[heading=bibintoc, title={Bibliography}]

\end{document}